\documentclass{birkjour}

\usepackage{amsthm, amssymb, amsmath, latexsym, amsfonts, mathcomp}
\usepackage{wrapfig}
\usepackage{xcolor}

\usepackage{hyperref}

\usepackage{comment}
\usepackage{xcolor}

\allowdisplaybreaks

\theoremstyle{plain}

\newtheorem{proposition}{Proposition}

\theoremstyle{remark}

\def\Z{\mathbb{Z}}
\def\R{\mathbb{R}}

\def\L{\ell}

\usepackage{mathtools}     
\usepackage{suffix}         
\usepackage{relsize}

\DeclarePairedDelimiterX\MeijerM[3]{\lparen}{\rparen}%
{\begin{smallmatrix}#1 \\ #2\end{smallmatrix}\delimsize\vert\,#3}

\newcommand\MeijerG[8][]{%
  G^{\,#2,#3}_{#4,#5}\MeijerM[#1]{#6}{#7}{#8}}

\WithSuffix\newcommand\MeijerG*[7]{%
  G^{\,#1,#2}_{#3,#4}\MeijerM*{#5}{#6}{#7}}

\begin{document}

\title[Divergence Free Polar Wavelets]{Divergence Free Polar Wavelets for the \\[5pt] Analysis and Representation of Fluid Flows}

\author[Christian Lessig]{Christian Lessig}

\address{Otto-von-Guericke Universit{\"a}t Magdeburg}

\email{lessig@isg.cs.uni-magdeburg.de}

\subjclass{Primary 42C40; Secondary 76B99}

\keywords{divergence freedom, wavelets}

\date{May 23, 2018}

\begin{abstract}
  We present a Parseval tight wavelet frame for the representation and analysis of velocity vector fields of incompressible fluids.
  Our wavelets have closed form expressions in the frequency and spatial domains, are divergence free in the ideal, analytic sense, have a multi-resolution structure and fast transforms, and an intuitive correspondence to common flow phenomena.
  Our construction also allows for well defined directional selectivity, e.g. to model the behavior of divergence free vector fields in the vicinity of boundaries or to represent highly directional features like in a von K{\'a}rm{\'a}n vortex street.
  We demonstrate the practicality and efficiency of our construction by analyzing the representation of different divergence free vector fields in our wavelets.
\end{abstract}

\maketitle


\section{Introduction}
\label{sec:introduction}

Turbulent fluids exhibit features at many different scales with fine details often being localized to small neighborhoods.
An efficient analysis and representation of fluid velocity vector fields hence has to adapt to these multiple scales and exploit their locality.
For incompressible fluids, it also should respect the divergence free nature of the vector field.

In this work, we present a vector-valued wavelet frame for the analysis and representation of incompressible fluids.
Our wavelets form a Parseval tight frame, are divergence free in the ideal, analytic sense, have an intuitive correspondence to vortices, closed form expressions in frequency and space, an associated multi-resolution structure, and fast transforms with filter taps that can be computed in closed form.
Furthermore, and in contrast to existing work, our construction also allows for well-defined angular selectivity, ranging from isotropic functions that can be interpreted as classical vortices to angularly highly selective, divergence free wavelets that provide a vector-valued analogue of curvelets.
 As we demonstrate, the directional wavelets are well suited to model divergence free vector fields in the vicinity of boundaries and around highly anisotropic features like those occurring in the wake behind an object.

The expedient properties of our wavelets result from a construction that exploits the intrinsic structure of divergence free vector fields in the Fourier domain, namely that the Fourier transform $\hat{\vec{u}}(\xi)$ is tangential to the sphere
\setlength{\columnsep}{4pt}
\begin{wrapfigure}[11]{r}[0pt]{0.4\columnwidth}
  \centering
  \vspace{-11pt}
  \includegraphics[width=0.4\columnwidth]{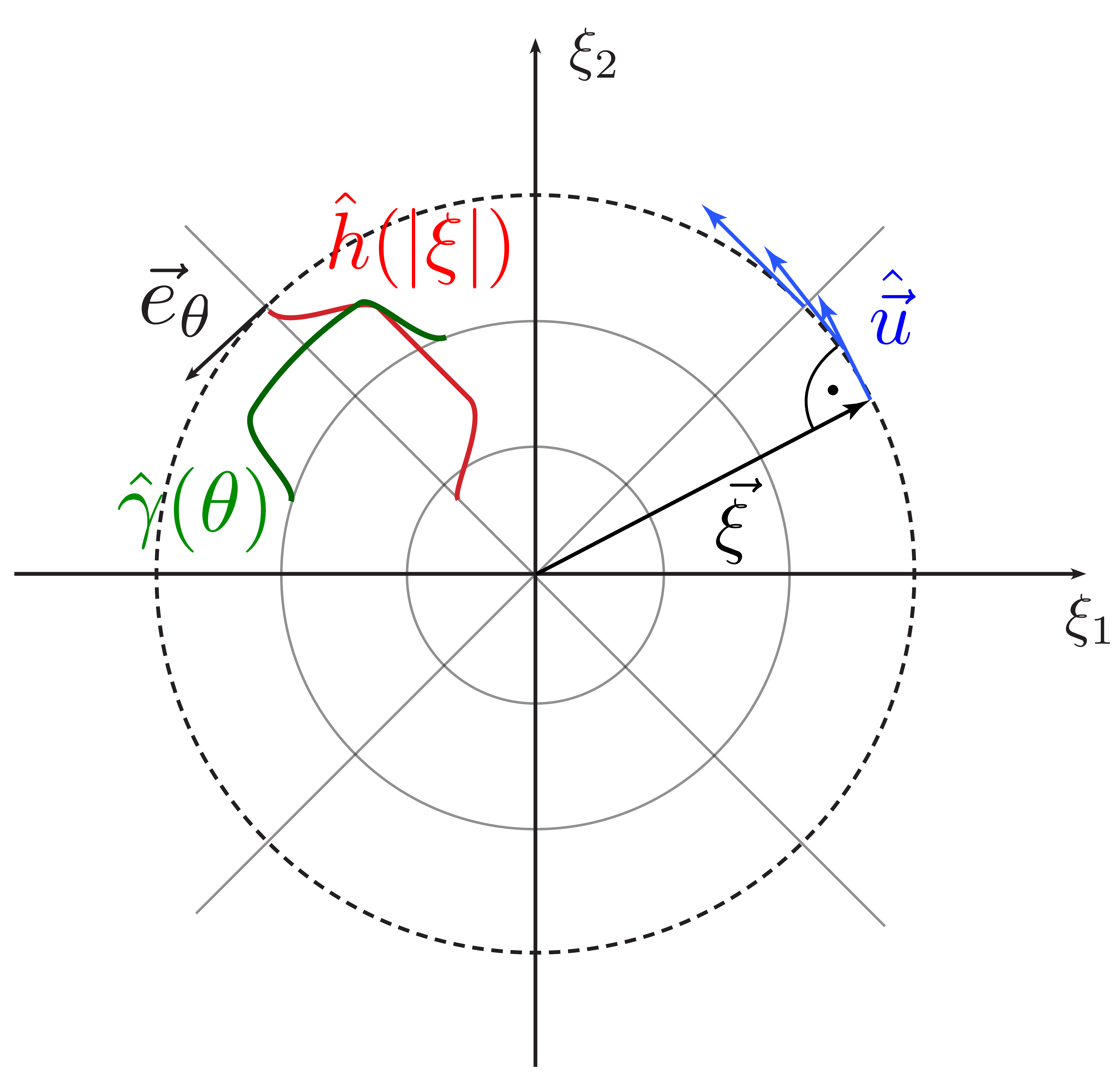}
  \vspace{-10pt}
\end{wrapfigure}
$\smash{S_{\xi}^{d-1}}$ in frequency space $\mathbb{R}_{\xi}^d$, see the inset figure.
To exploit this, we define our wavelets as $\smash{\hat{\vec{\psi}}(\xi)} = \hat{\psi}(\xi) \, \vec{\tau}$ where the vector component $\vec{\tau}$ is chosen such that it is tangential to $\smash{S_{\xi}^{d-1}}$ and $\hat{\psi}(\xi)$ is a scalar window function controlling the frequency localization.
We use bandlimited windows defined separably in polar or spherical coordinates as $\hat{\psi}(\xi) = \hat{\gamma}(\theta) \, \hat{h}(\vert \xi \vert)$ and which are hence compatible with the divergence freedom constraint.
Among other things, this enables us to obtain closed form expressions for the wavelets in the spatial domain.
The construction is particular simple in $\mathbb{R}^2$ where the angular basis vector $\vec{e}_{\theta_{\xi}}$ spans all non-radial vector fields, that is all whose inverse Fourier transform is divergence free, and hence our wavelets are given by $\hat{\psi}(\xi) = -i \,\hat{\gamma}(\theta_{\xi}) \, \hat{h}(\vert \gamma \vert) \,  \vec{e}_{\theta_{\xi}}$.

Since we use bandlimited window functions, our wavelets do not have compact support in the spatial domain.
One hence typically needs a large number of frame functions around the actual region of interest and it also implies that the filter taps for the fast transform usually have to be truncated, resulting in limited accuracy.
Nonetheless, our numerical results demonstrate that one can work efficiently with our wavelets and using only algorithms that operate in the spatial domain where it is easier to exploit sparsity.


Divergence free wavelets were first developed in the 1980s for four-dimensional spacetime~\cite{Federbush1986,Federbush1987a,Federbush1988}.
In the 1990s, several compactly supported, divergence-free wavelets for Euclidean space were proposed~\cite{LemarieRieusset1992,Battle1993,Battle1994,Urban1995}.
In comparison to our construction, the definitions of these wavelets are considerably more complicated and they do not provide well defined angular selectivity.
Our wavelets, in contrast, do not have compact support.
Closest to the present construction is recent work by Bostan, Unser and Ward~\cite{Bostan2015} who also proposed a tight, divergence free wavelet frame using window functions that are separable in polar (or spherical) coordinates.
In the 1990s a similar construction was proposed by Suter~\cite{Suter1994}, but to our knowledge without an implementation.
In contrast to our approach, Bostan, Unser and Ward~\cite{Bostan2015} enforce divergence freedom numerically by removing the radial frequency component of wavelets defined in Cartesian coordinates.
This makes an efficient spatial implementation of their wavelets difficult and prevents them from considering angular localization.
In the scalar setting, polar-separable windows and the resulting scalar polar wavelets have recently been studied systematically by Unser and co-workers~\cite{Chenouard2012,Unser2013,Ward2014}, extending previous work on steerable wavelets.
These results provide an important basis for our work.


\section{Divergence Free Polar Wavelets}
\label{sec:construction}

We begin the section by recalling the scalar polar wavelets recently introduced by Unser and co-workers~\cite{Chenouard2012,Unser2013,Ward2014}.
We then detail the construction of our divergence free wavelets in two dimensions before turning to the $3$-dimensional setting.
In Sec.~\ref{sec:div_free:fast_transform} we detail the computation of the filter taps required for the fast transform.
Our conventions have been relegated to Appendix~\ref{sec:conventions}.

\subsection{Scalar Polar Wavelet}
\label{sec:construction:scalar}

Polar wavelets~\cite{Chenouard2012,Unser2013,Ward2014,Lessig2018a} are separable in spherical coordinates in the Fourier domain $\mathbb{R}_{\xi}^d$.
They are hence given by $\hat{\psi}_s(\xi) = \hat{\gamma}_s(\bar{\xi}) \, \hat{h}_s(\vert \xi \vert)$ where $\hat{\gamma}_s(\bar{\xi})$ is an angular window over directions and $\smash{\hat{h}_s(\vert \xi \vert)}$ a radial one.
The multi-index $s = (j_s,k_s,t_s)$ describes the scale, position and orientation of the wavelets and, without loss of generality, we assume there are $M_j$ different orientations generated as rotations of one angular mother window $\hat{\gamma}(\bar{\xi})$.


\begin{figure}
  \begin{center}
  \includegraphics[width=0.24\columnwidth]{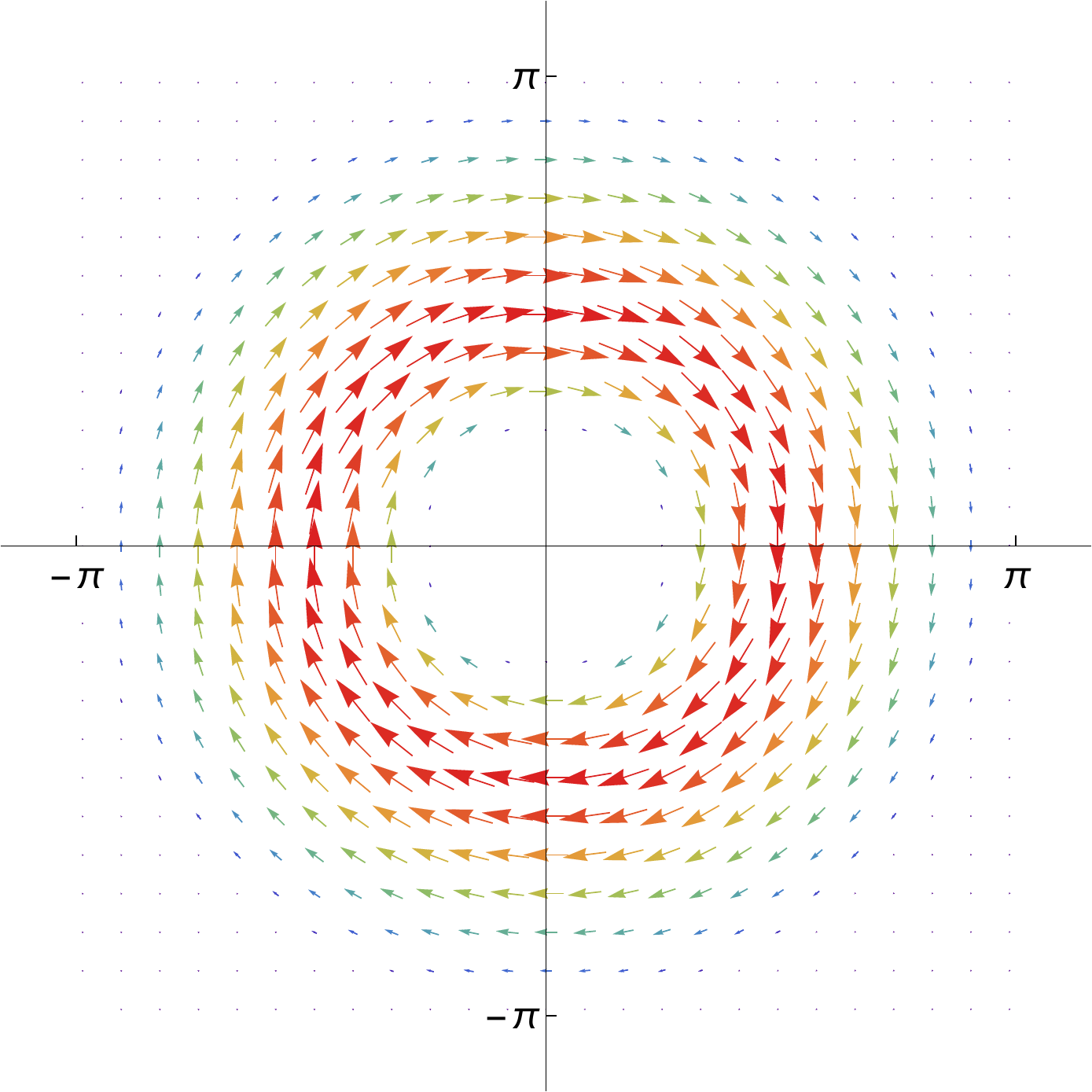}
  \includegraphics[width=0.24\columnwidth]{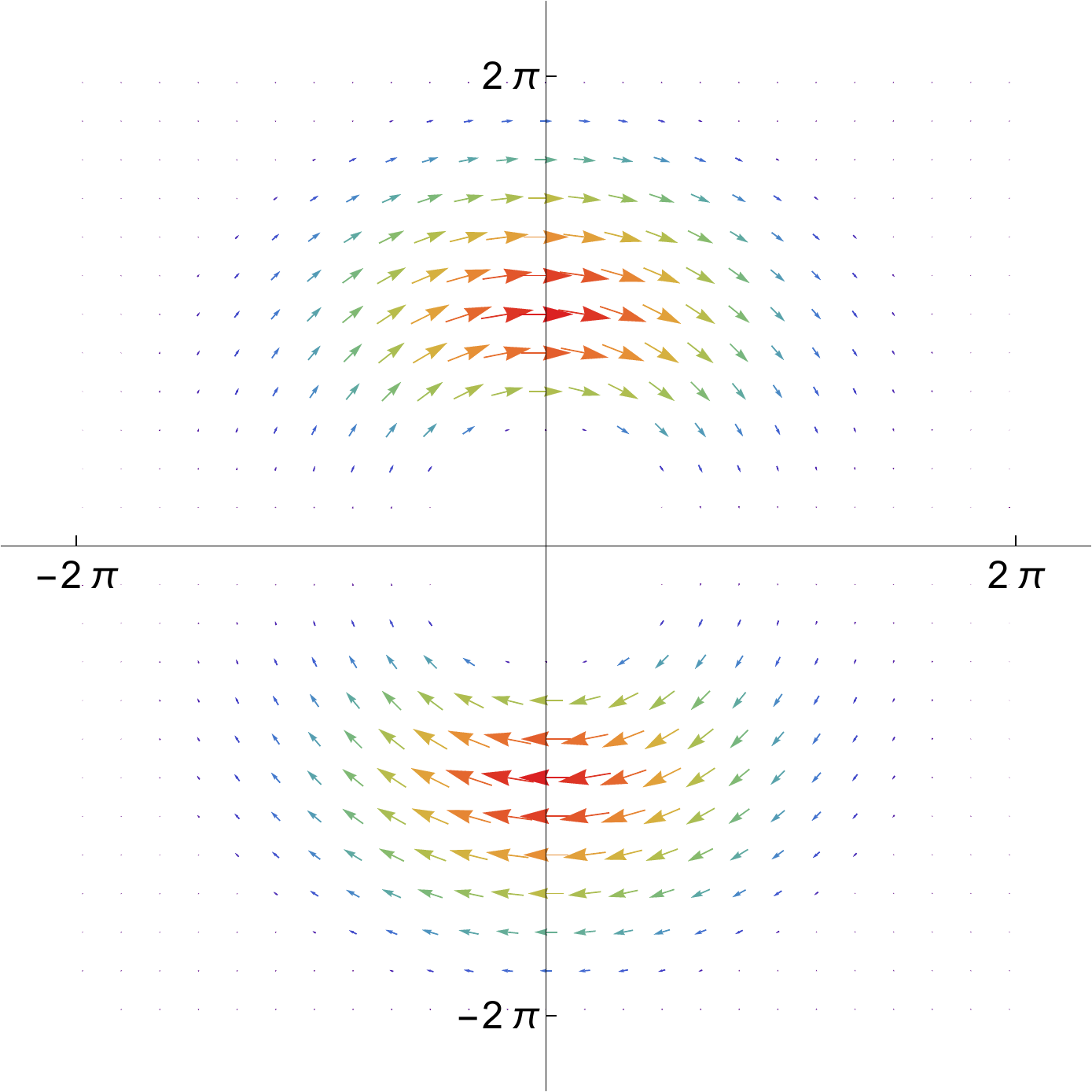}
  \includegraphics[width=0.24\columnwidth]{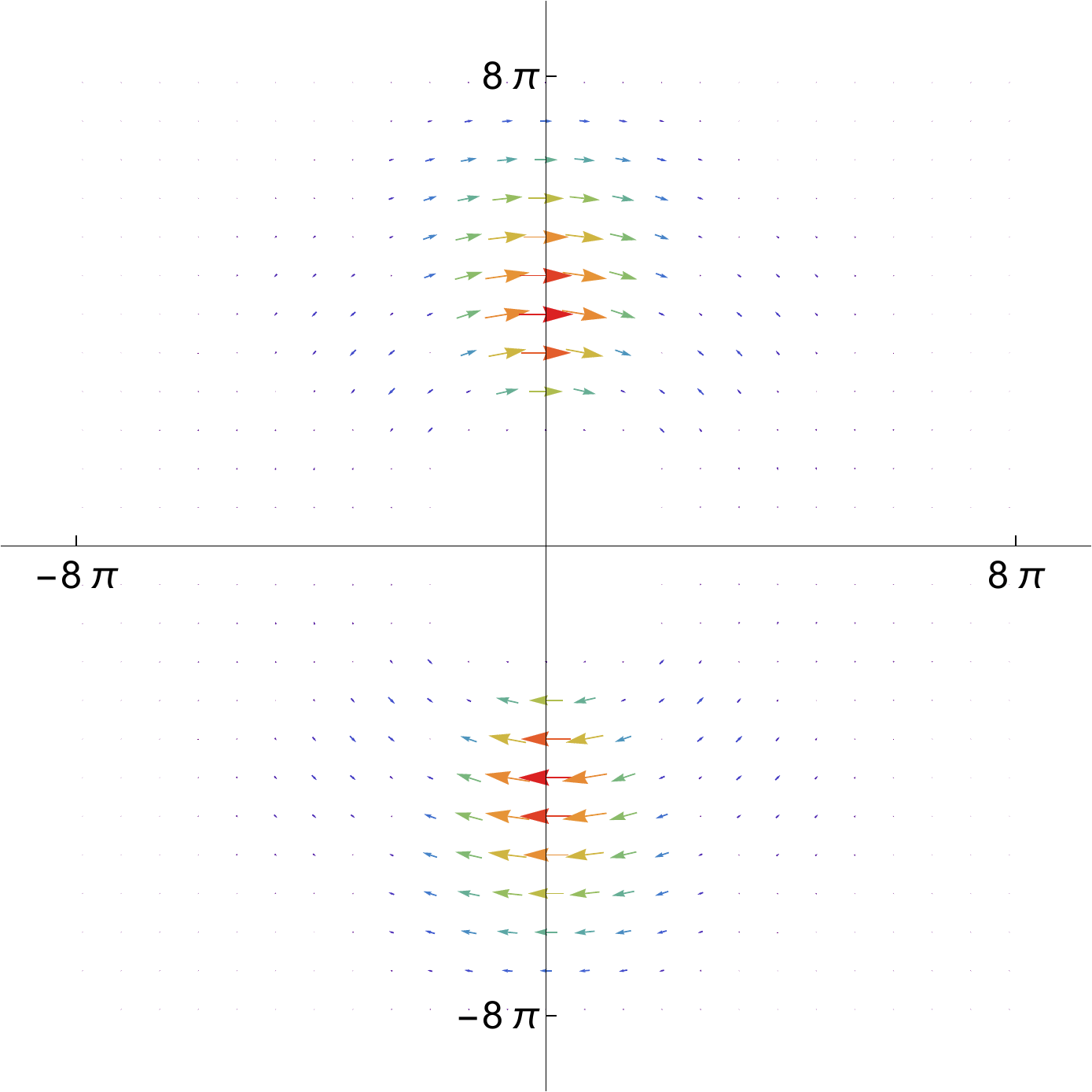}
  \includegraphics[width=0.24\columnwidth]{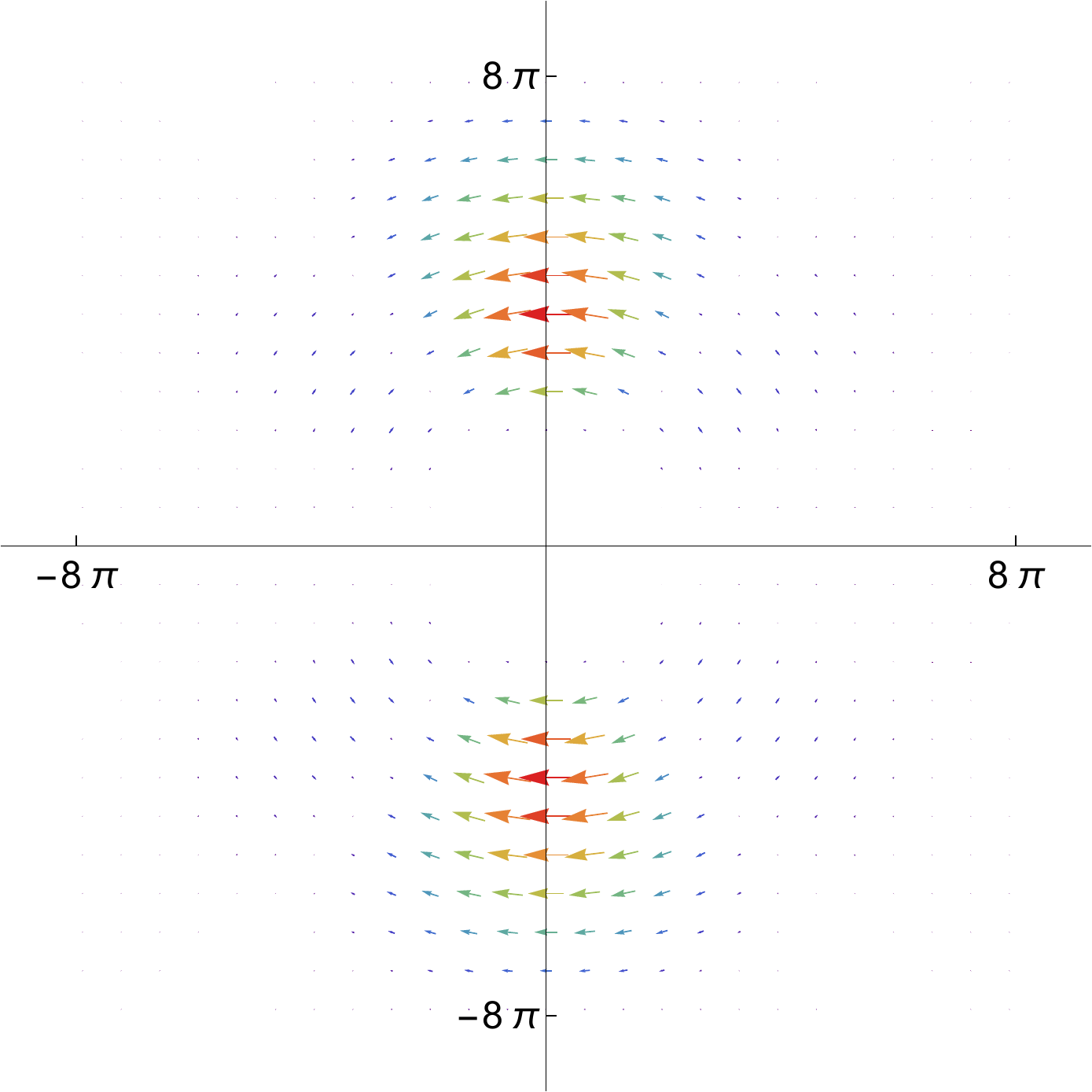}
  \\
  \includegraphics[width=0.24\columnwidth]{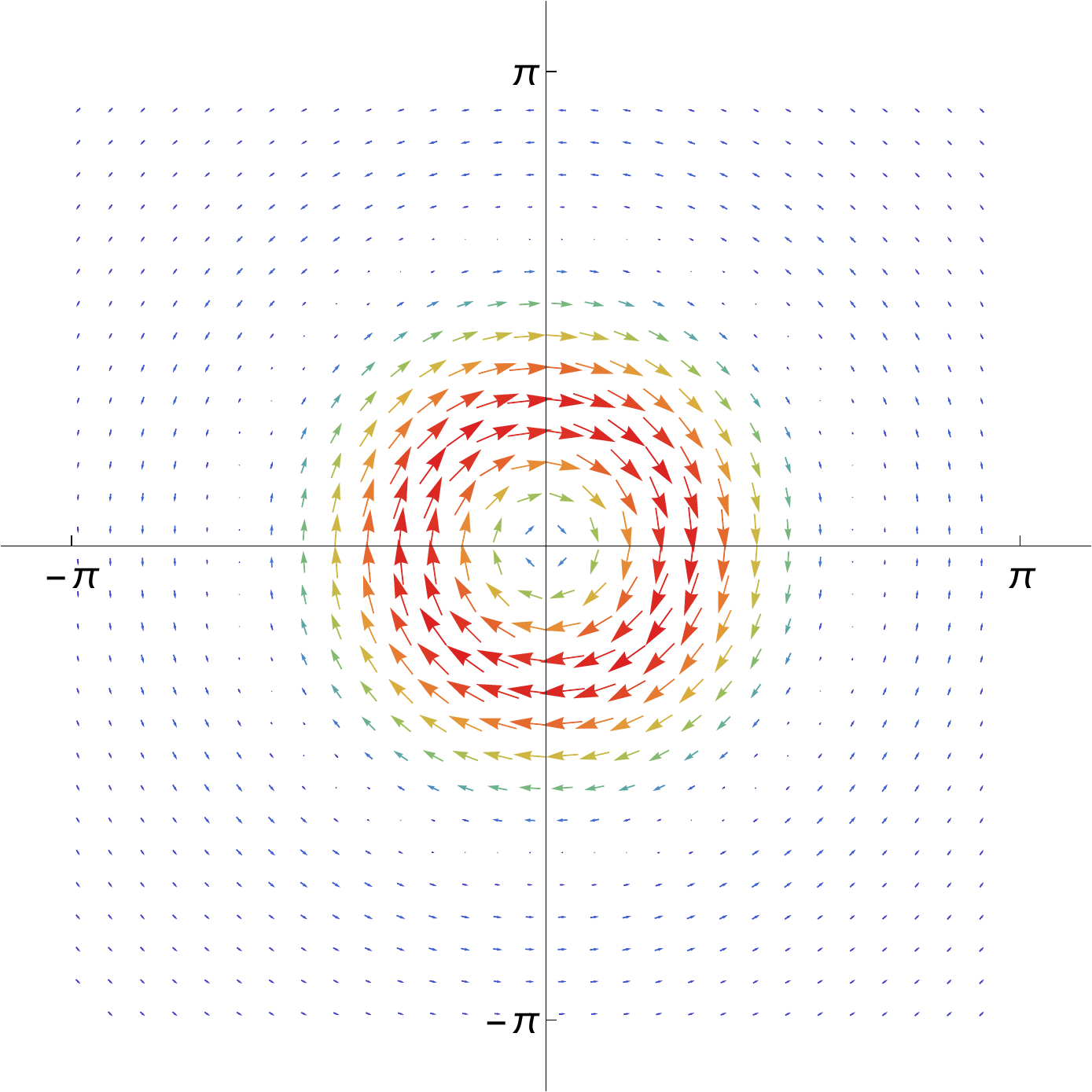}
  \includegraphics[width=0.24\columnwidth]{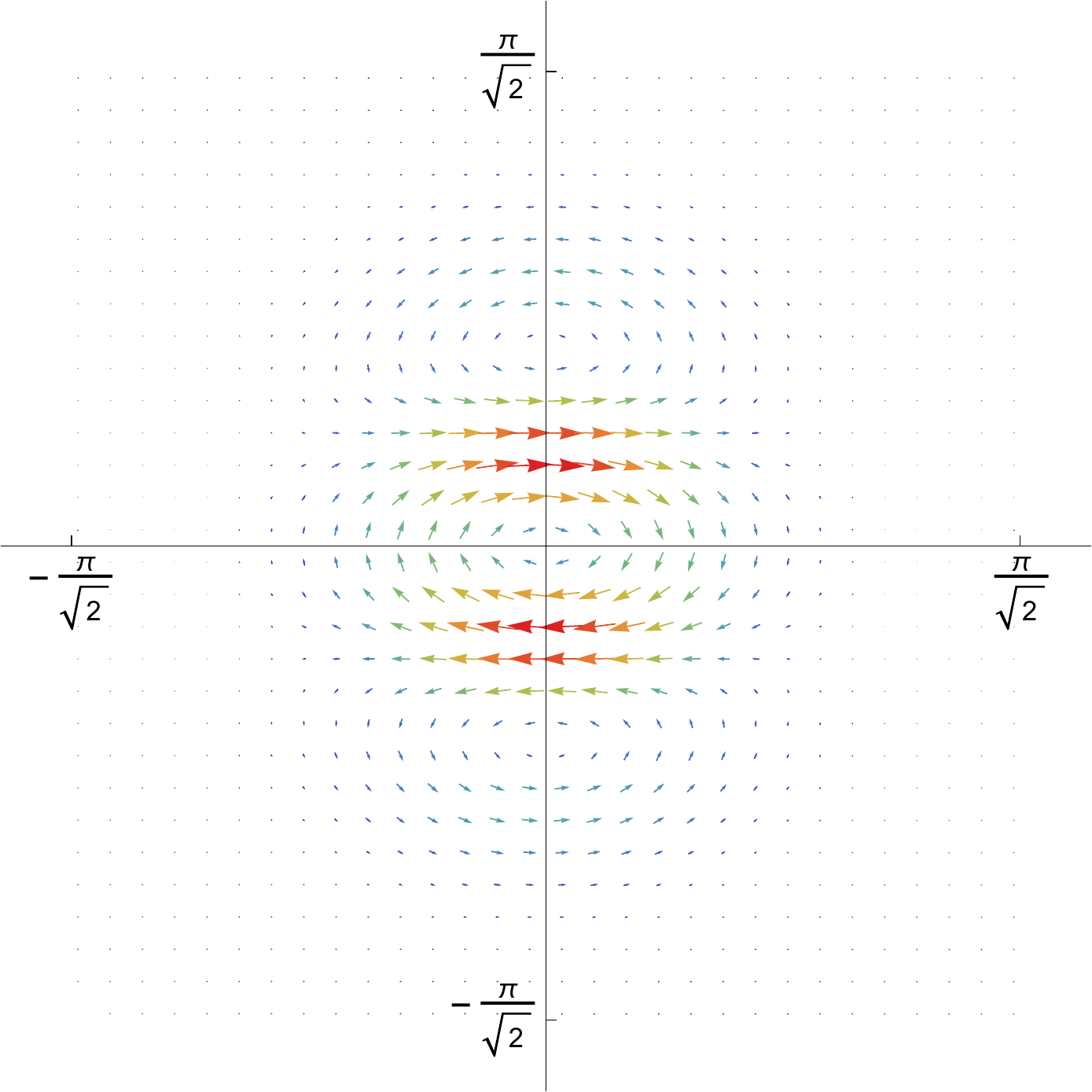}
  \includegraphics[width=0.24\columnwidth]{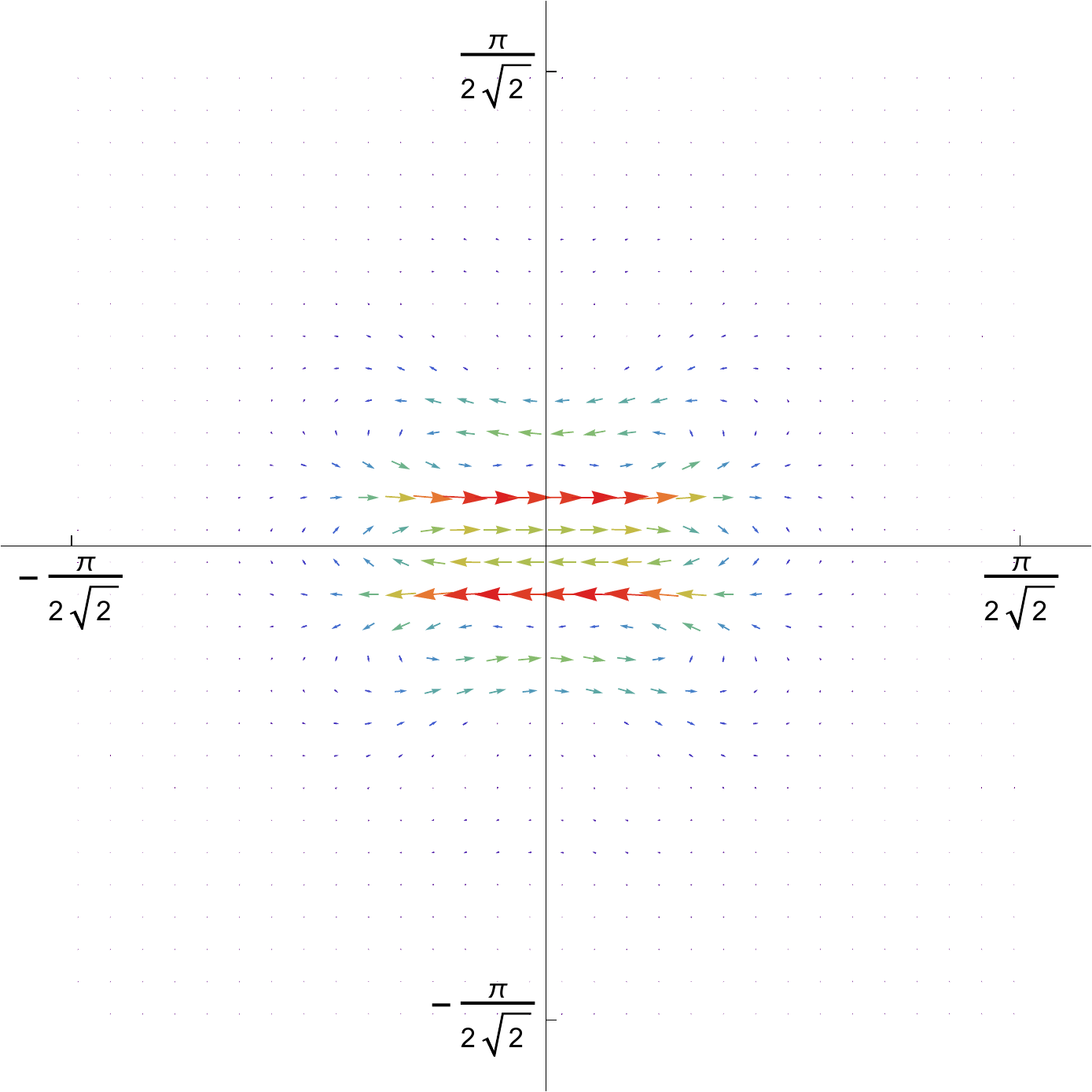}
  \includegraphics[width=0.24\columnwidth]{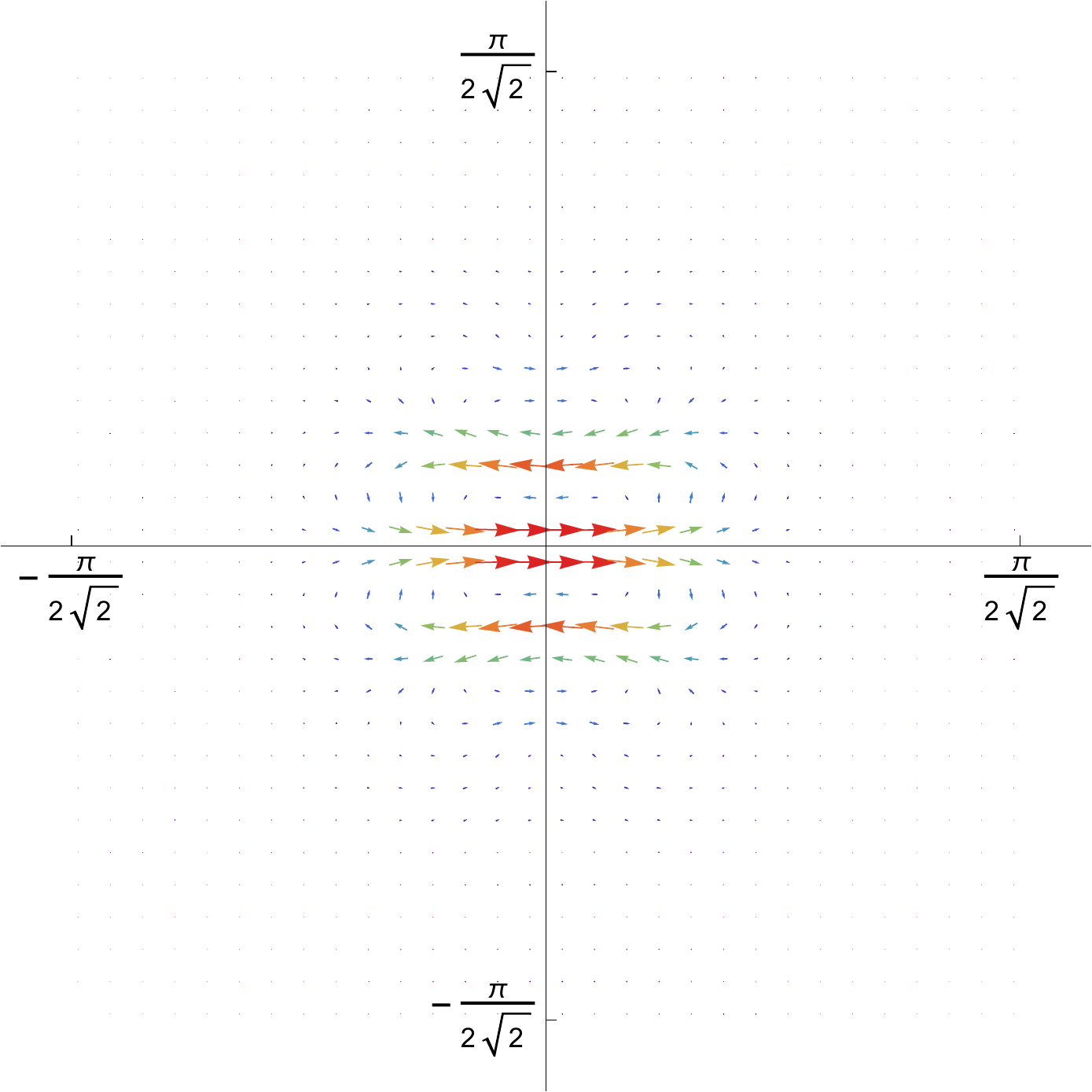}
  \end{center}
  \caption{Divergence free wavelets in the Fourier (top) and spatial domains (bottom). The left columns shows an isotropic function that can be interpreted as an isolated vortex. The remaining ones provide examples of directional wavelets for the horizontal orientation with different angular localizations and for shearing- (middle two columns) and streaming-like flows (right column).}
  \label{fig:psis:2d}
\end{figure}

In two dimensions, the angular window is naturally written as a Fourier series.
Polar wavelets are there hence given by
\begin{align}
  \label{eq:polarlets:2d}
  \hat{\psi}_{s}(\xi)
  = \! \Bigg( \sum_m \beta_{m}^{j_s,t_s} \, e^{i m \theta_{\xi}} \Bigg) \hat{h}\big(2^{-j_s} \vert \xi \vert \big) \, e^{-i \langle \xi , 2^{-j_s} k_s \rangle}
\end{align}
with $\beta_{n}^{j_s,t_s} = \beta_{n}^{j_s} \, e^{-i n t_s (2\pi / M_{j_s})}$, i.e. the different orientations are obtained by rotating a fixed, possibly level-dependent window function.
In the simplest case $\beta_n = \delta_{n 0}$ and one obtains isotropic functions.
The polar wavelets defined by Eq.~\ref{eq:polarlets:2d} also have a closed form expression in the spatial domain~\cite{Unser2013,Lessig2018a}.

The construction in two dimensions carries over to the $3$-dimensional setting when the Fourier series is replaced by a spherical harmonics expansion.
In $\mathbb{R}^3$, polar wavelets are hence defined by
\begin{align}
  \label{eq:polarlets:3d}
  \hat{\psi}_{s}(\xi)
  = \Bigg( \sum_{l,m} \kappa_{lm}^{j_s,t_s} \, y_{lm}(\bar{\xi}) \Bigg) \hat{h}(2^{-j_s} \vert \xi \vert) \, e^{-i \langle \xi , 2^{-j_s} k_s \rangle}
\end{align}
and one can again derive a closed form expression in the spatial domain~\cite{Lessig2018a}.

Under suitable admissibility conditions for $\hat{h}(\vert \xi \vert)$ and the coefficients $\beta_{j,n}^t$ and $\kappa_{lm}^{jt}$ defining the angular window functions $\hat{\gamma}_s(\bar{\xi})$, the wavelets in Eq.~\ref{eq:polarlets:2d} and Eq.~\ref{eq:polarlets:3d} form tight frames for $L_2(\mathbb{R}^2)$~\cite{Unser2013} and $L_2(\mathbb{R}^3)$~\cite{Ward2014}, respectively.
As usual, in practice on typically uses scaling functions, denoted as $\phi(x)$, on a coarsest level to represent the low frequency part of a signal.
We refer to the original works for more details.

\subsection{Divergence free polar wavelets in 2D}
\label{sec:construction:2d}

For a $2$-dimensional vector field to be divergence free it has to be tangential to the circle $S_{\vert \xi \vert}^1$ in frequency space for every radius $\vert \xi \vert$ for which it is defined.
These vector fields are hence naturally described in polar coordinates since then the radial vector component with respect to $\vec{e}_r$ has to vanish and only the angular one with respect to $\vec{e}_{\theta}$ can be nonzero.
A divergence free polar mother wavelet thus has to have the form
\label{eq:psi:divfree:2d}
\begin{align}
  \label{eq:psi:divfree:2d:frequency}
	\hat{\vec{\psi}}(\xi)
	= - i \, \hat{\gamma}(\theta_{\xi}) \, \hat{h}(\vert \xi \vert) \, \vec{e}_{\theta}
	= - i \Bigg( \! \sum_{n = -N}^N \beta_n \, e^{i n \theta_{\xi}} \Bigg) \, \hat{h}(\vert \xi \vert) \, \vec{e}_{\theta}
\end{align}
where the factor of $-i$ ensures that $\hat{\vec{\psi}}(\xi)$ is real-valued in space.
To obtain the spatial representation, we write $\vec{e}_{\theta_{\xi}}$ as $\vec{e}_{\theta_{\xi}} = (\sin{\theta_{\xi}}, -\cos{\theta_{\xi}})$ with respect to canonical Cartesian coordinates so that we can compute the inverse transform independently for each coordinate.
Then writing $\sin{\theta_{\xi}}$ and $-\cos{\theta_{\xi}}$ as a Fourier series and combining it with those of $\hat{\gamma}(\theta_{\xi})$ we can evaluate the inverse Fourier transform in polar coordinates, see Appendix~\ref{sec:appendix:spatial:2d} for details.
This yields
\begin{align}
  \label{eq:psi:divfree:2d:space}
	\vec{\psi}(x) &= \frac{1}{2}
    \sum_{\sigma \in \pm 1} \sum_{m} i^{m + \sigma} \, \beta_m  \, e^{i (m + \sigma) \theta_x}  h_{m + \sigma}(\vert x \vert)
  \begin{pmatrix}
    \textrm{-}\sigma
    \\
    i
  \end{pmatrix}
\end{align}
where the $h_{m + \sigma}(\vert x \vert)$ are the inverse Hankel transforms of $\hat{h}(\vert \xi \vert)$.
For the Simoncelli window~\cite{Portilla2000} these also have closed form expressions, which are available in the reference implementation.

Eq.~\ref{eq:psi:divfree:2d:space} defines a family of wavelet functions that is by construction divergence free and  whose angular localization can be controlled using the $\beta_m$, see Fig.~\ref{fig:psis:2d} for examples.
We will return to suitable choices of the angular localization at the end of the section.
In the special case when $\hat{\gamma}(\theta_{\xi})$ is isotropic, i.e. $\beta_m = \delta_{0m}$, one obtains $\vec{\psi}(x) = \hat{h}_1(\vert x \vert) \, \vec{e}_{\theta_x}$.
This is an isotropic, divergence free vector field in space that can be interpreted as an isolated vortex, see Fig.~\ref{fig:psis:2d}, left.

\begin{figure}[t]
  \begin{center}
  \includegraphics[width=0.24\textwidth]{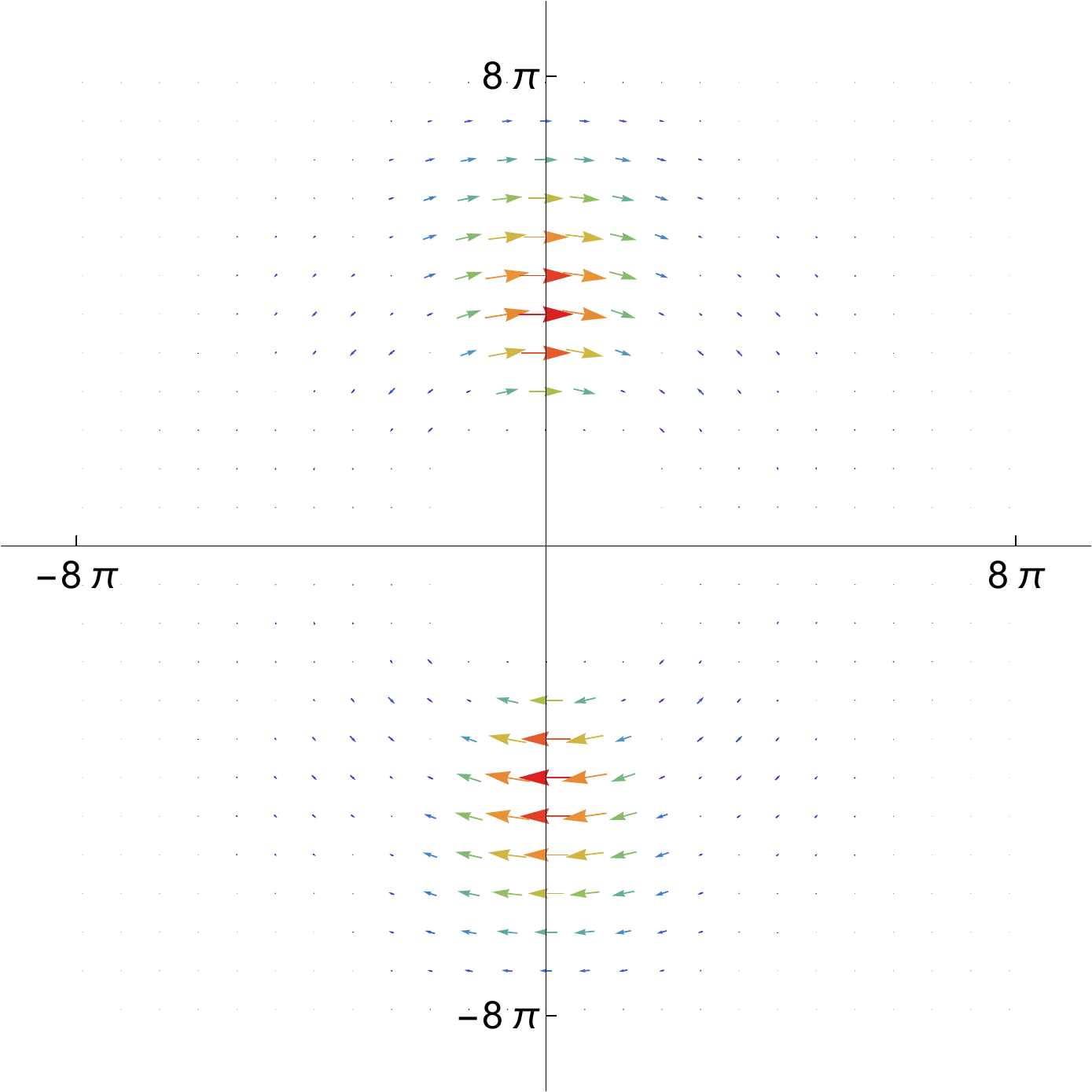}
  \includegraphics[width=0.24\textwidth]{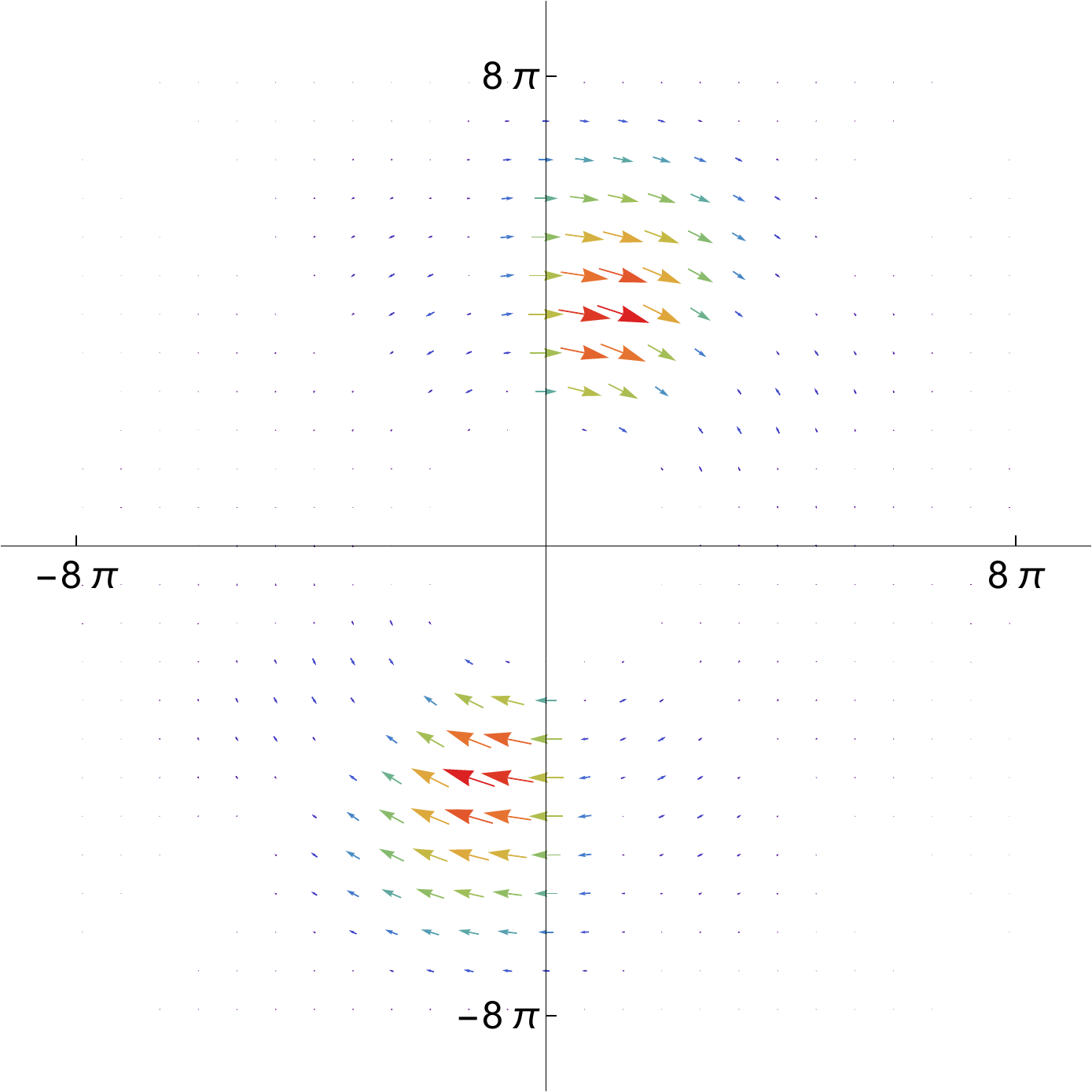}
  \includegraphics[width=0.24\textwidth]{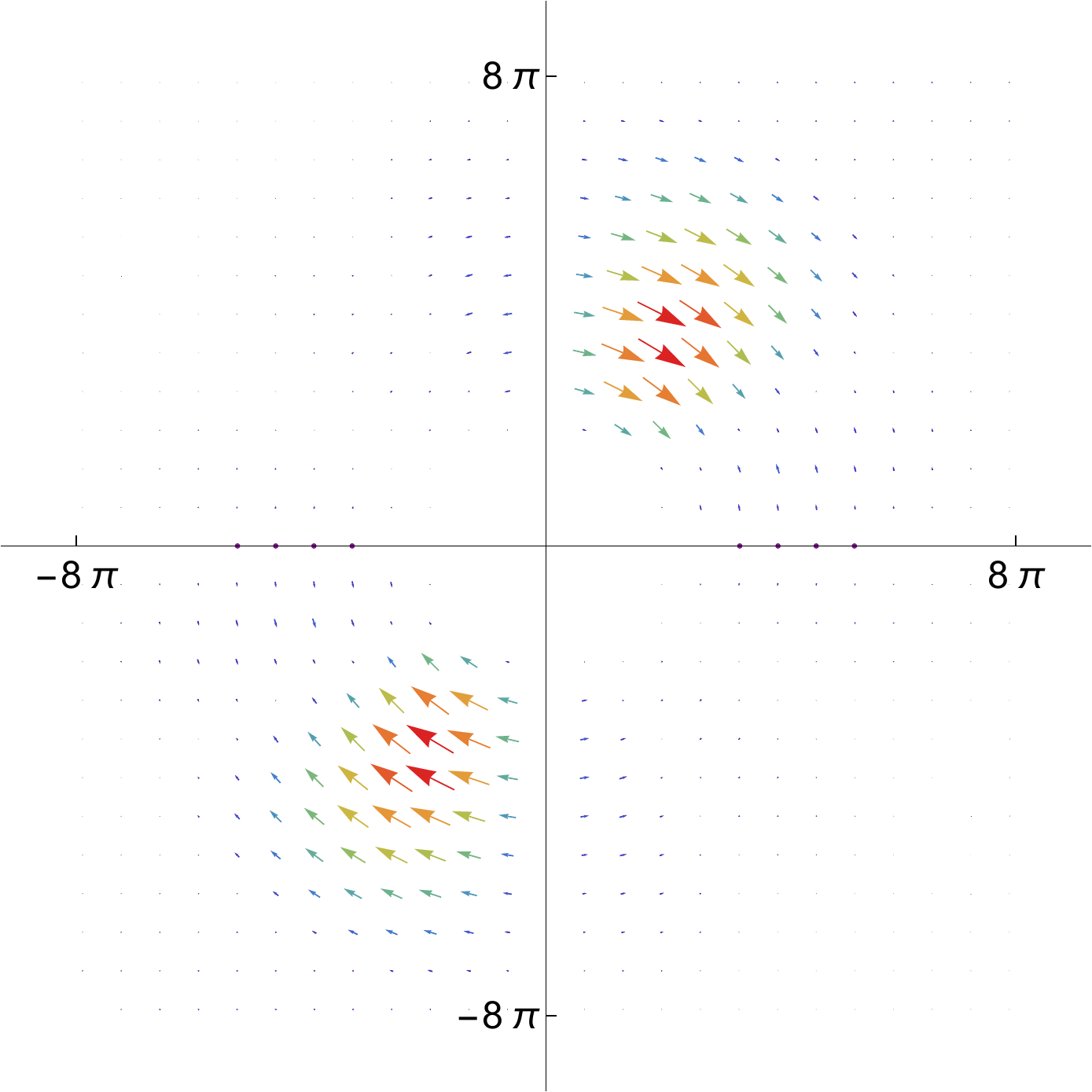}
  \includegraphics[width=0.24\textwidth]{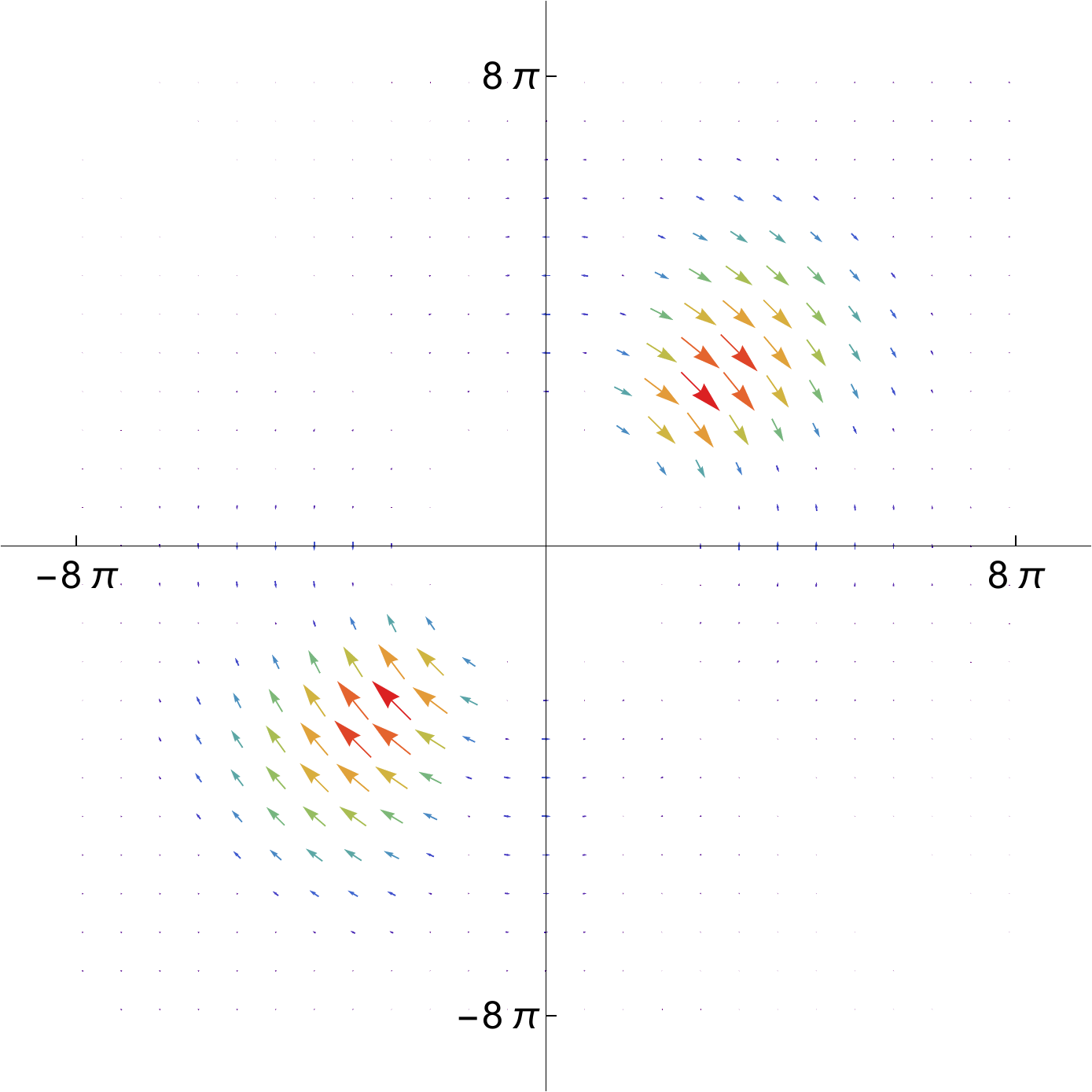}
  \includegraphics[width=0.24\textwidth]{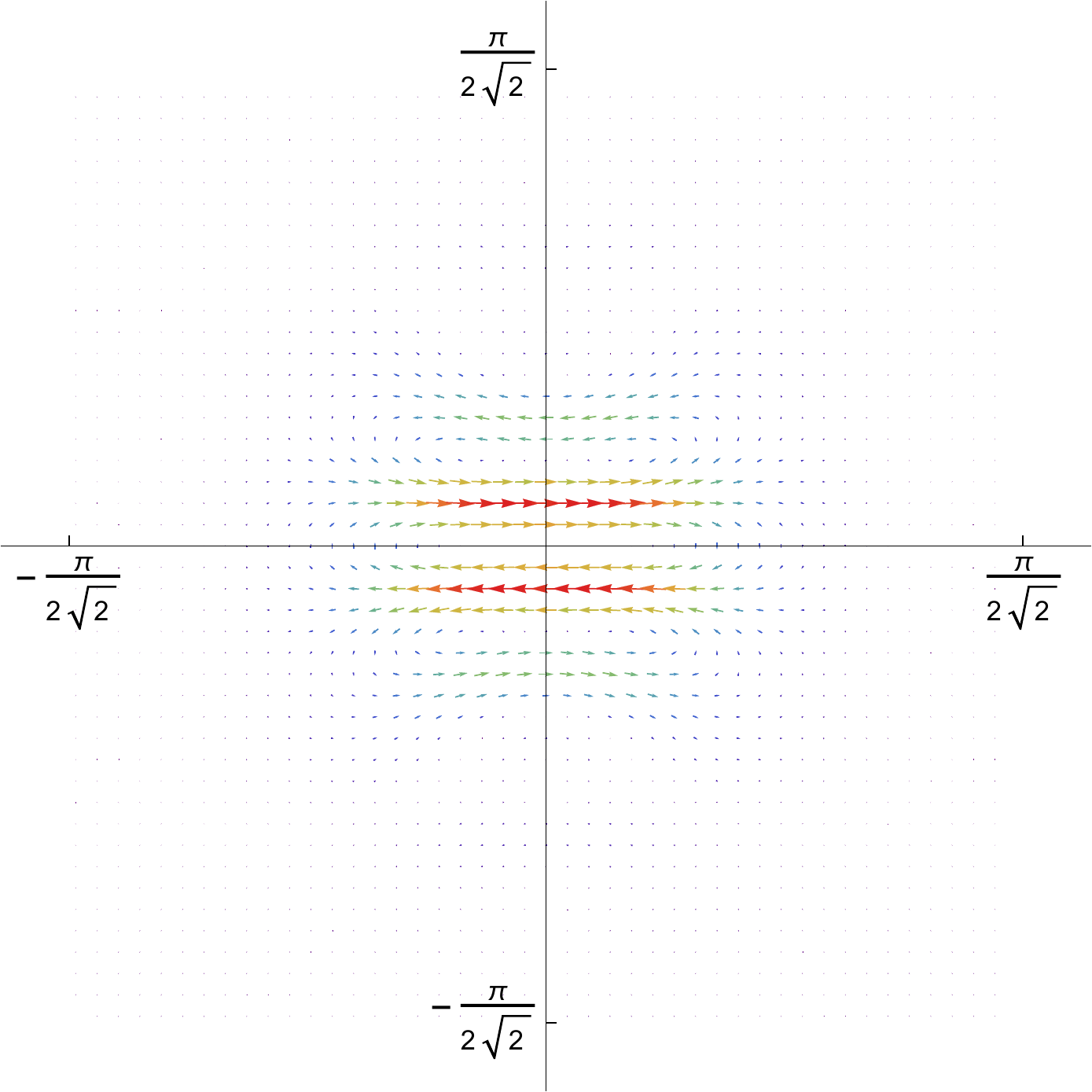}
  \includegraphics[width=0.24\textwidth]{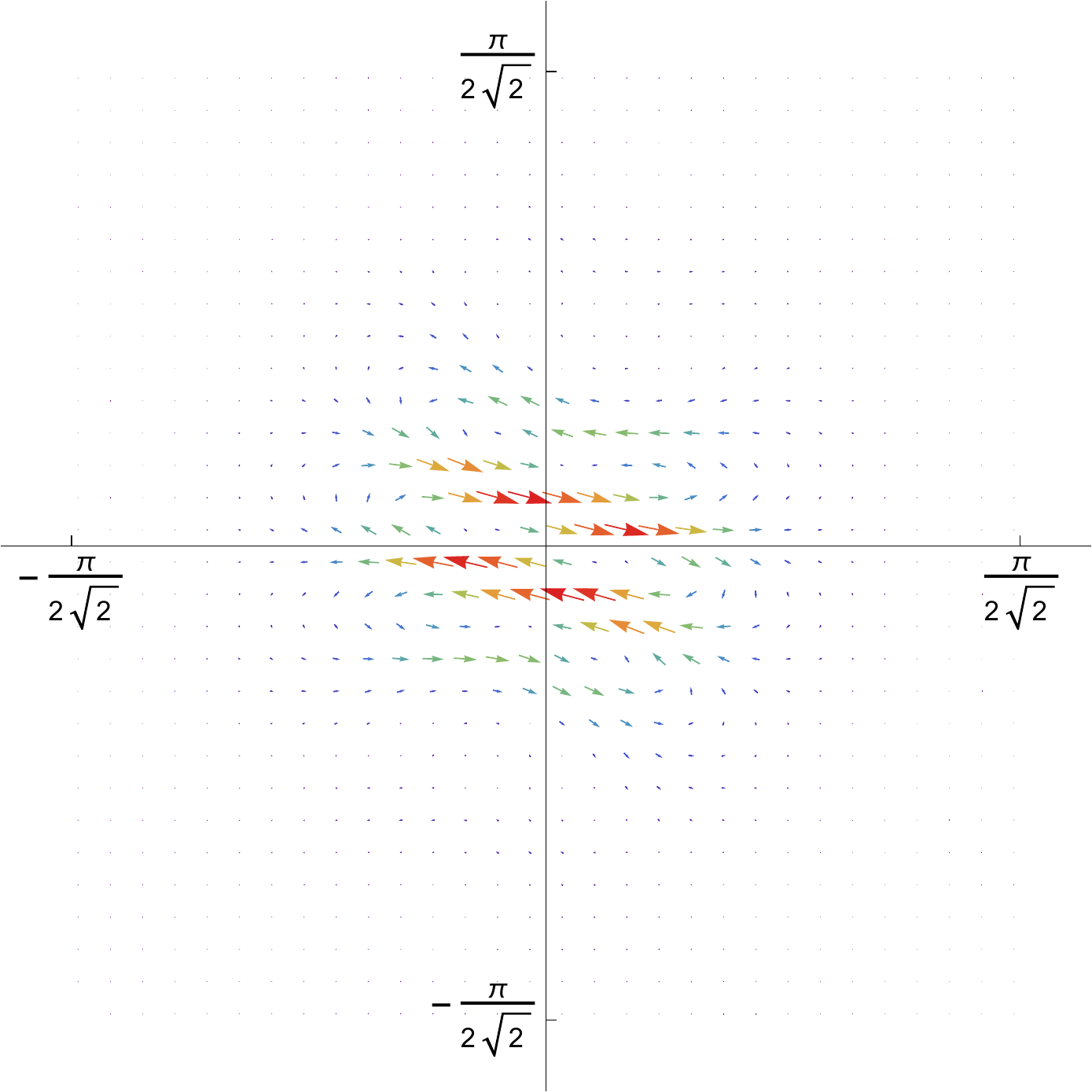}
  \includegraphics[width=0.24\textwidth]{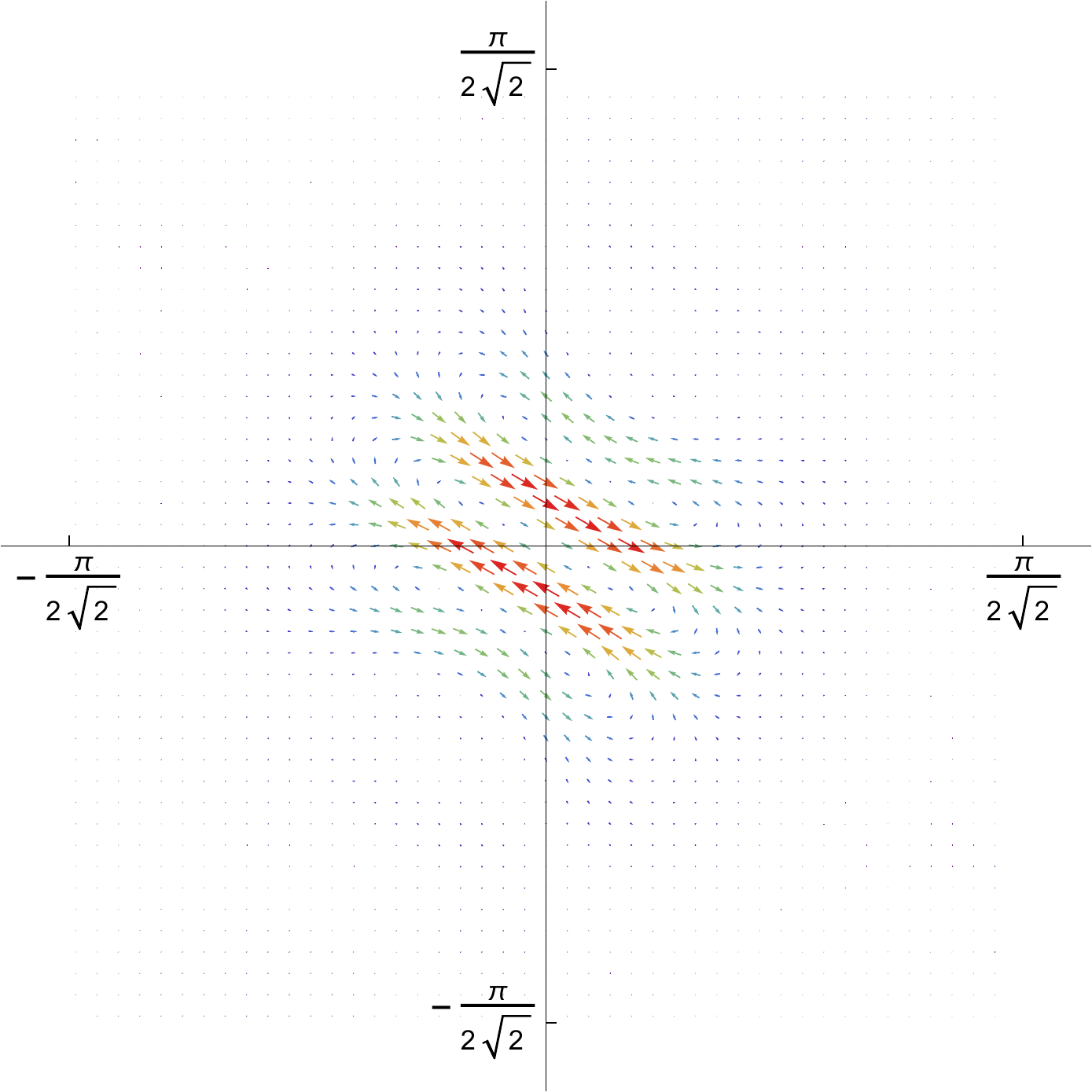}
  \includegraphics[width=0.24\textwidth]{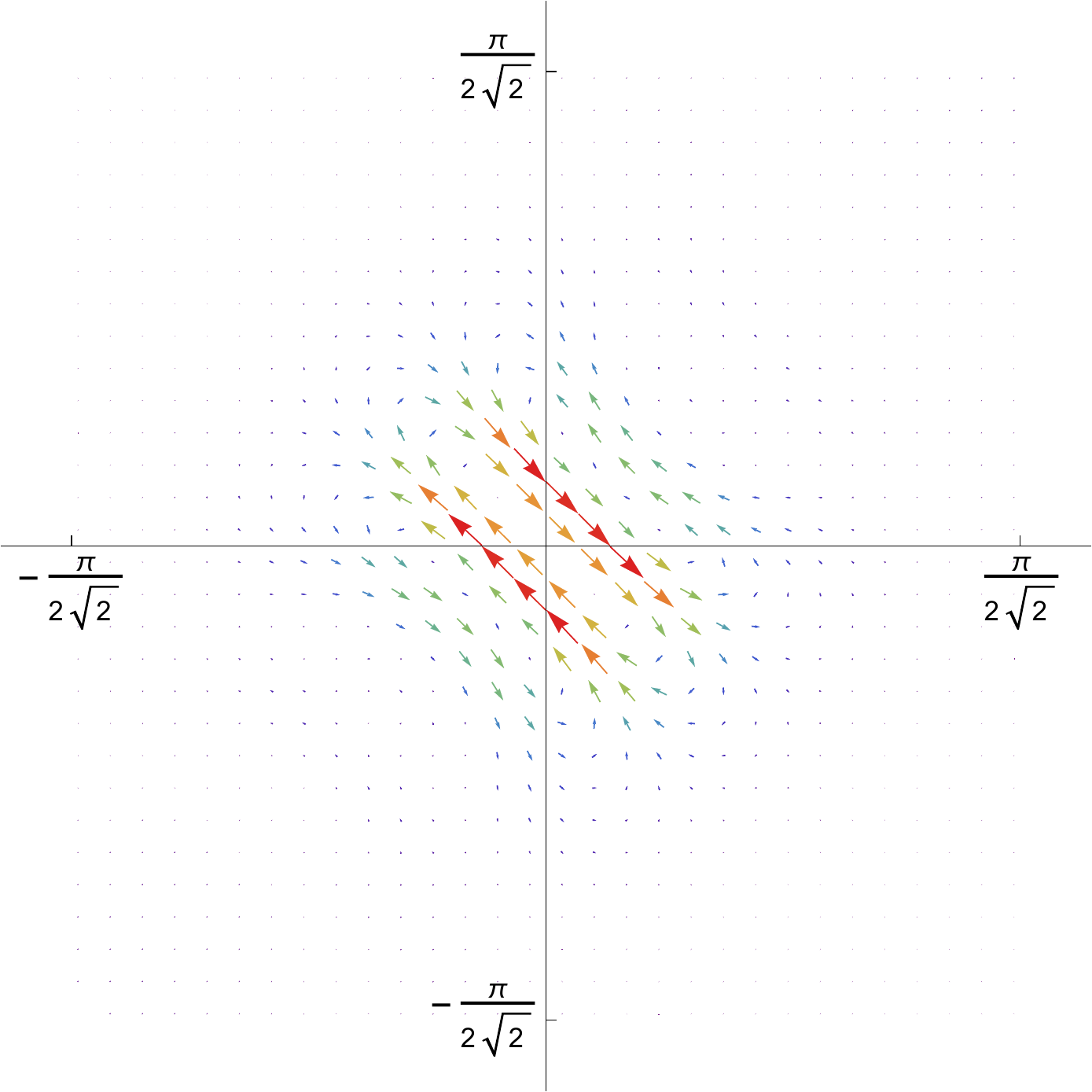}
  \end{center}
  \caption{Different orientations for directional wavelets for $j=3$ in the frequency (top) and spatial (bottom) domains. As shown in Sec~\ref{sec:construction:2d:directional} and Sec.~\ref{sec:experiments}, significant coefficients are only obtained when a flow has anisotropic structures aligned with the wavelet.}
  \label{fig:psis:ns:2d:ko}
\end{figure}

\subsubsection{Properties of wavelets}

A critical property of the wavelet defined in Eq.~\ref{eq:psi:divfree:2d:frequency} is that these generate a Parseval tight frame for the space $L_2^{\textrm{div}}(\mathbb{R}^{2,2})$ of divergence free vector fields with finite $L_2$-norm.
Hence, most of the conveniences of an orthonormal basis are available and in particular primary and dual frame functions coincide and Parseval's identity holds so that the norm of the expansion coefficients equals those of the signal.
See for example Daubechies classical treatise~\cite[Ch. 3]{Daubechies1992} or the expository articles by Kovacevic and Chebira~\cite{Kovacevic2007,Kovacevic2007a} for more details on frames.

\begin{proposition}
    \label{prop:polarlet:2d:parseval}
    Let $U_j$ be the $(M_j \times 2 N_j + 1)$-dimensional matrix formed by the angular localization coefficients $\beta_{n}^{j,t} = \beta_{n}^{j} \, e^{-i n t (2\pi / M_{j})}$ for the $M_j$ different orientations, and let $D_j$ be a diagonal matrix of size $(2 N_j + 1) \times (2 N_j + 1)$.
  When the Cald{\`e}ron admissibility condition $\sum_{j \in \mathbb{Z}} \big\vert \hat{h}( 2^{-j} \vert \xi \vert ) \big\vert^2 = 1$, $\forall \xi \in \mathbb{R}^2$ is satisfied and $U_j^H U_j = D_j$ with $\mathrm{tr}( D_j ) = 1$ for all levels $j$,
  then any divergence free vector field $\vec{u}(x) \in L_2^{\mathrm{div}}(\mathbb{R}^{2,2})$ has the representation
  \begin{subequations}
  \begin{align}
    \label{eq:polarlet:2d:parseval}
    \vec{u}(x) = \sum_{j \in \mathbb{Z}} \sum_{k \in \mathbb{Z}^2} \sum_{t=1}^{M_j} \big\langle \vec{u}(y) , \vec{\psi}_{j,k,t}(y) \big\rangle \, \vec{\psi}_{j,k,t}(x)
  \end{align}
  with frame functions
  \begin{align}
    \vec{\psi}_{j,k,t}(x) = \frac{2^j}{2\pi} \, \vec{\psi}\big( R_{2\pi t / M_j} (2^j x - k) \big)
  \end{align}
  \end{subequations}
  where $\vec{\psi}(x)$ is given by Eq.~\ref{eq:psi:divfree:2d:space} and $R_{2\pi t / M_j}$ is the rotation by $2\pi t / M_j$.
\end{proposition}

\begin{proof}
Taking the Fourier transform of Eq.~\ref{eq:polarlet:2d:parseval}, using Parseval's theorem, and with Eq.~\ref{eq:psi:divfree:2d:frequency} one obtains
\begin{subequations}
\begin{align}
    \hat{u}(\xi) \, \vec{e}_{\theta}
    &= \sum_{j \in \mathbb{Z}} \sum_{k \in \mathbb{Z}^2} \sum_{t=1}^{M_j} \Big\langle \hat{u}(\eta) \, \vec{e}_{\theta} , \hat{\psi}_{j,k,t}(\eta) \, \vec{e}_{\theta} \Big\rangle \, \hat{\psi}_{j,k,t}(\xi) \, \vec{e}_{\theta}
    \\
    \label{prop:polarlet:2d:parseval:2}
    &= \underbrace{\sum_{j \in \mathbb{Z}} \sum_{k \in \mathbb{Z}^2} \sum_{t=1}^{M_j} \Big\langle \hat{u}(\eta) , \hat{\psi}_{j,k,t}(\eta) \Big\rangle \, \hat{\psi}_{j,k,t}(\xi)}_{\textrm{scalar polar wavelet expansion of $\hat{u}(\xi)$}} \, \vec{e}_{\theta}
\end{align}
\end{subequations}
with the inner product in the second line now being those for scalar functions and $\hat{\vec{u}}(\eta) = \hat{u}(\eta) \, \vec{e}_{\theta}$ for some scalar $\hat{u}(\eta)$ by the divergence freedom of $\vec{u}(x)$.
In Eq.~\ref{prop:polarlet:2d:parseval:2}, $\vec{e}_{\theta}$ is fixed and the other terms provide a scalar polar wavelet expansion of the scalar magnitude function $\hat{u}(\xi)$.
Since the proposition uses the same assumptions required for scalar polar wavelets to form a Parseval tight frame, see~\cite[Proposition 4.1]{Unser2013} for the isotropic case and~\cite[Proposition 4.2]{Unser2013} for the anisotropic one, the result follows.
\end{proof}

As usual, in practice one uses scaling functions $\phi_s(x)$ for a signal's low frequency part and a variation of Proposition~\ref{prop:polarlet:2d:parseval} holds in this case.
Next to being Parseval tight, our divergence free wavelets satisfy other useful properties:
\begin{enumerate}
  \item \emph{Intuitive correspondence:} The isotropic frame functions can be interpreted as isolated vortices and the directional wavelets are similar to flows along boundaries.
  \item \emph{Ideal divergence freedom:} Our wavelets are divergence free in the ideal, analytic sense and, with $\hat{h}(\vert \xi \vert)$ from~\cite{Portilla2000}, which yields a closed form expression for the radial windows $h_m(\vert x \vert)$ in the spatial domain, they can be evaluated to arbitrary precision.
  \item \emph{Analytic, polar frequency representation:} The analytic, polar separable construction of $\smash{\hat{\vec{\psi}}(\xi)}$ enables the closed form computation of differential operators such as curl, i.e. the vorticity of the fluid velocity vector field.
  \item \emph{Closed form spatial representation:} The wavelets have closed form expressions in the spatial domain, which are useful for example for reconstruction and interpolation~\cite{Lessig2018a} and to compute advection.
  \item \emph{Multi-resolution structure and fast transform:} The multi-resolution structure and associated fast transform of our wavelets enable efficient analysis and reconstruction.
\end{enumerate}

\begin{figure}
  \includegraphics[trim={80 100 30 20},clip,width=0.32\textwidth]{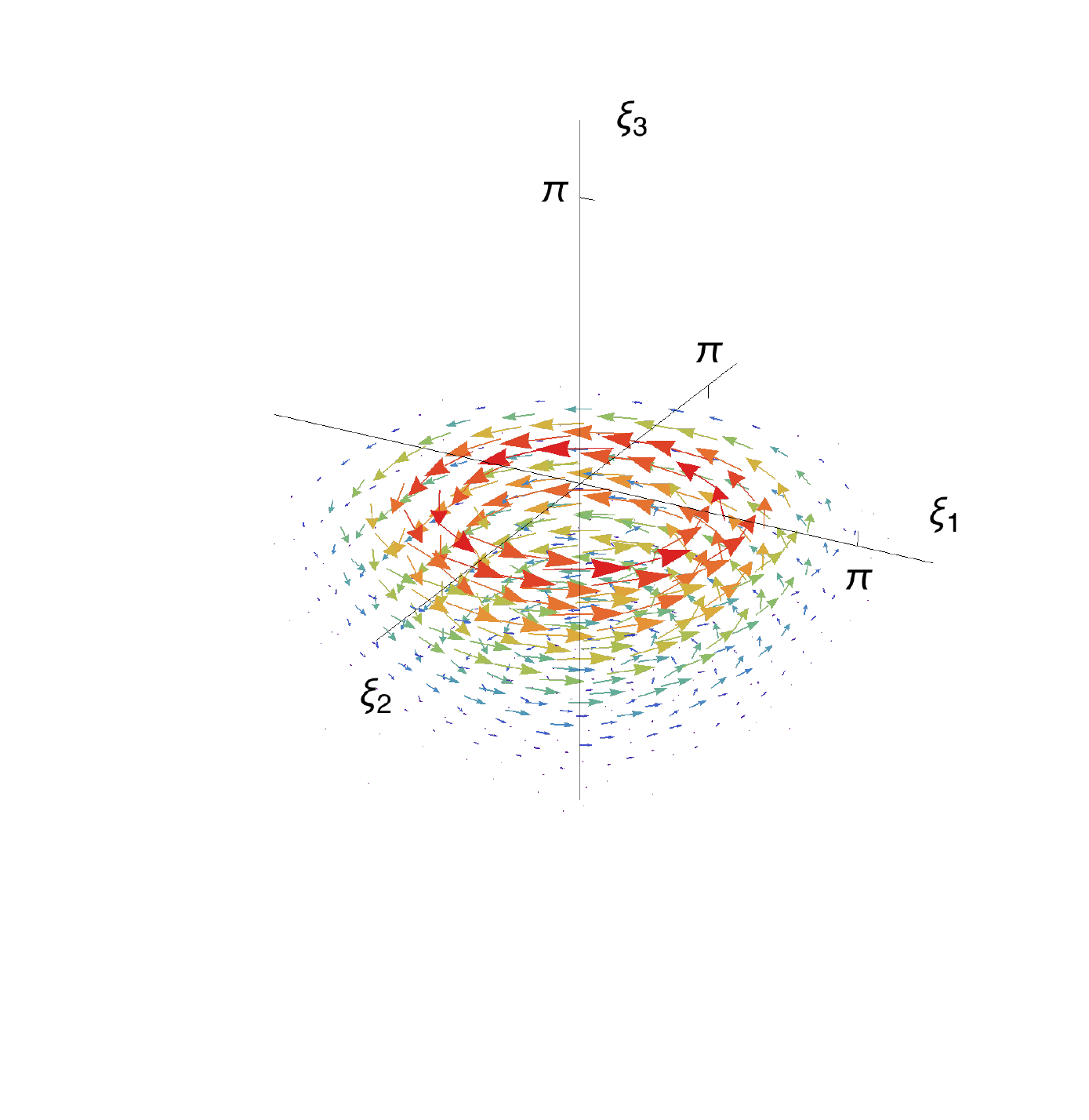}
  \includegraphics[trim={80 100 30 20},clip,width=0.32\textwidth]{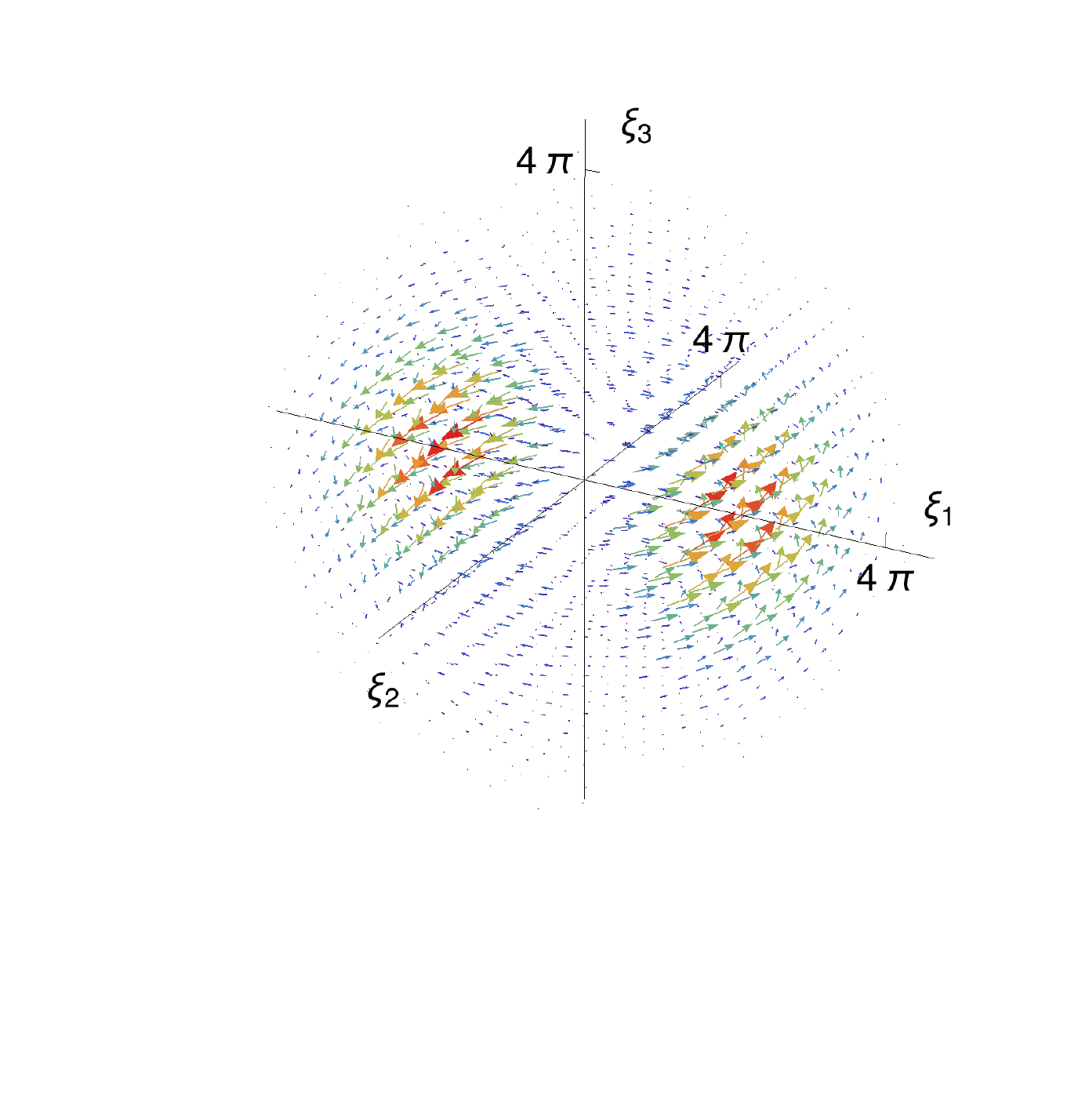}
  \includegraphics[trim={80 100 30 20},clip,width=0.32\textwidth]{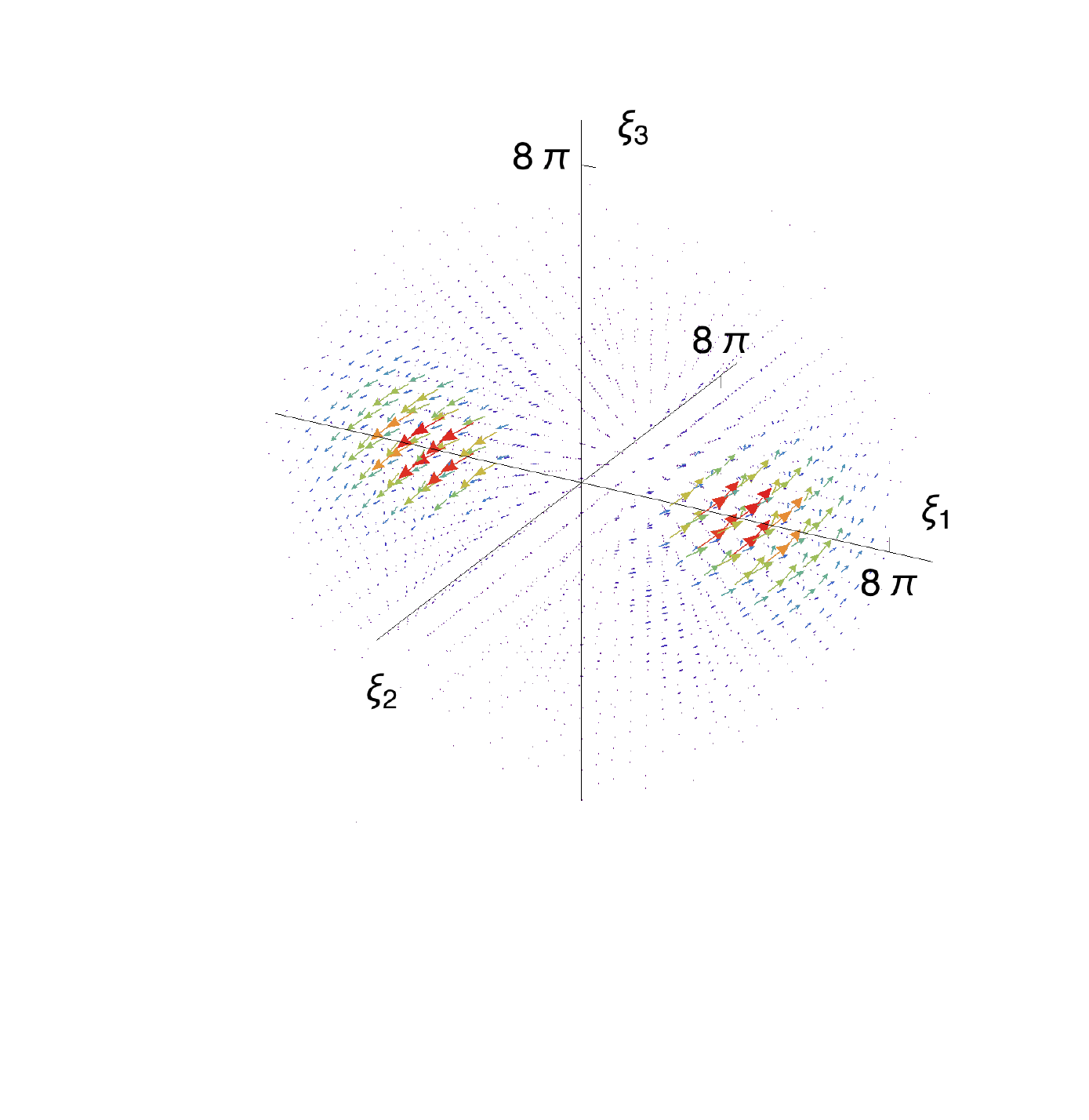}
  \includegraphics[trim={80 100 30 20},clip,width=0.32\textwidth]{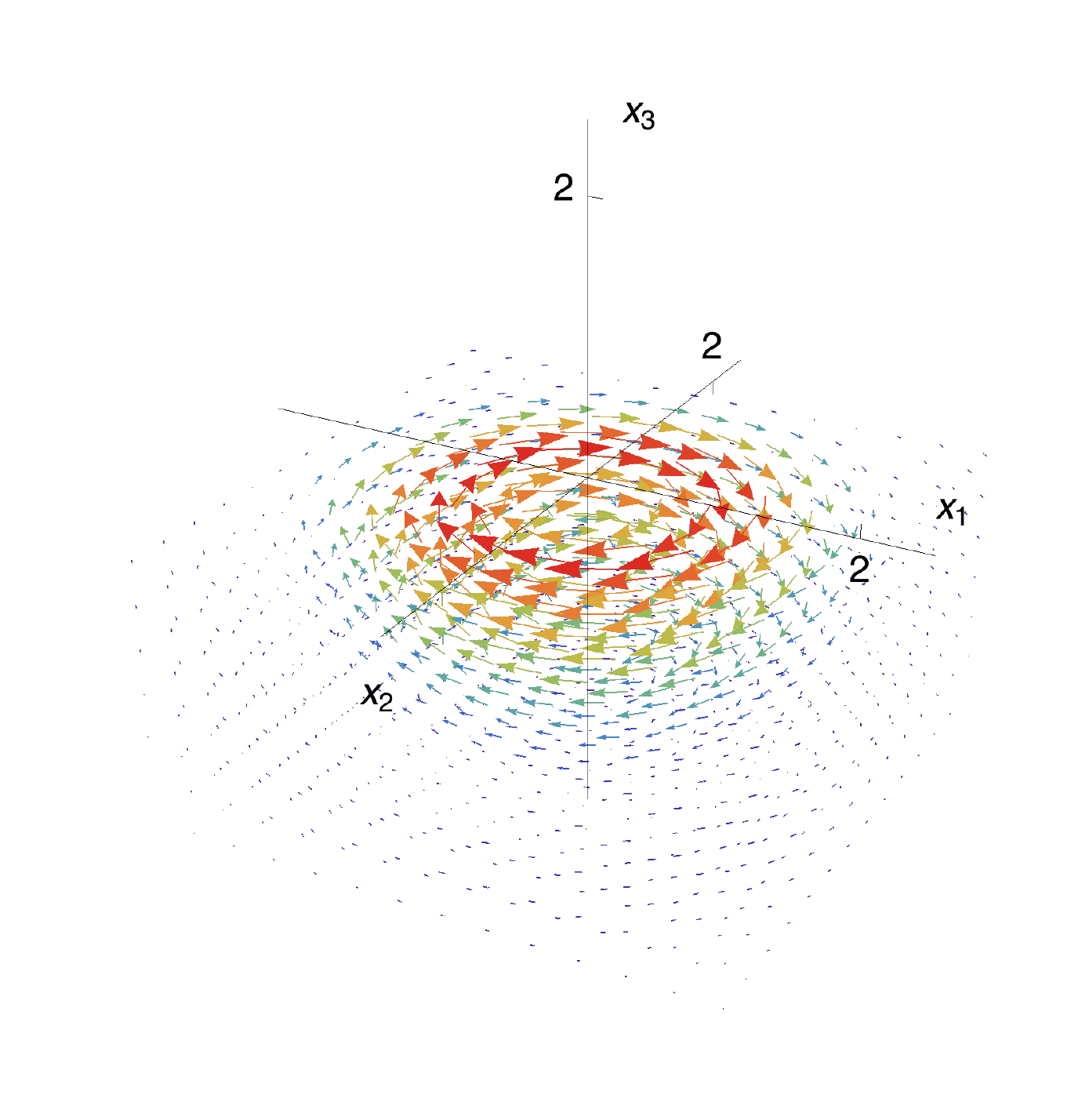}
  \includegraphics[trim={80 100 30 20},clip,width=0.32\textwidth]{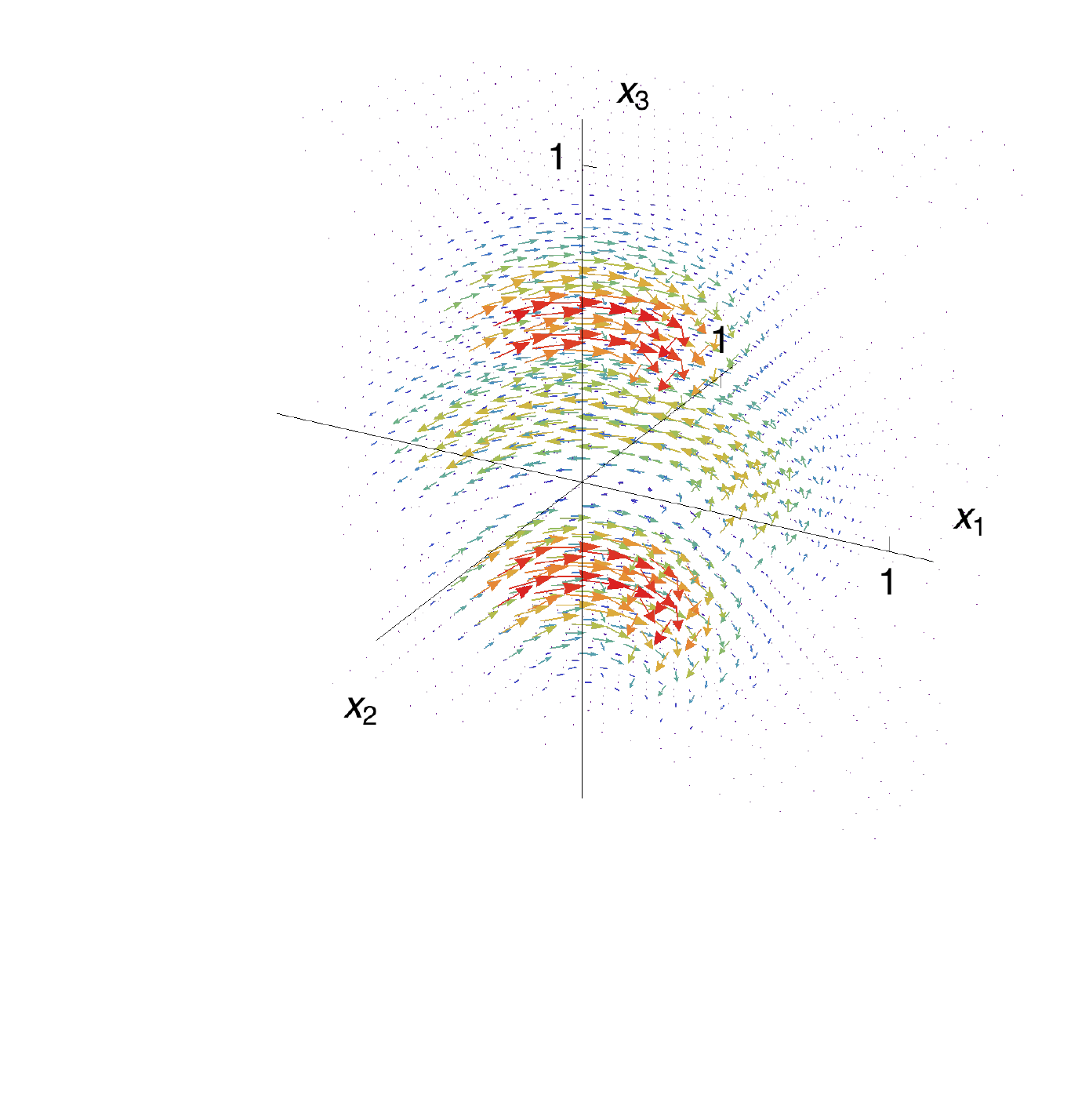}
  \includegraphics[trim={80 100 30 20},clip,width=0.32\textwidth]{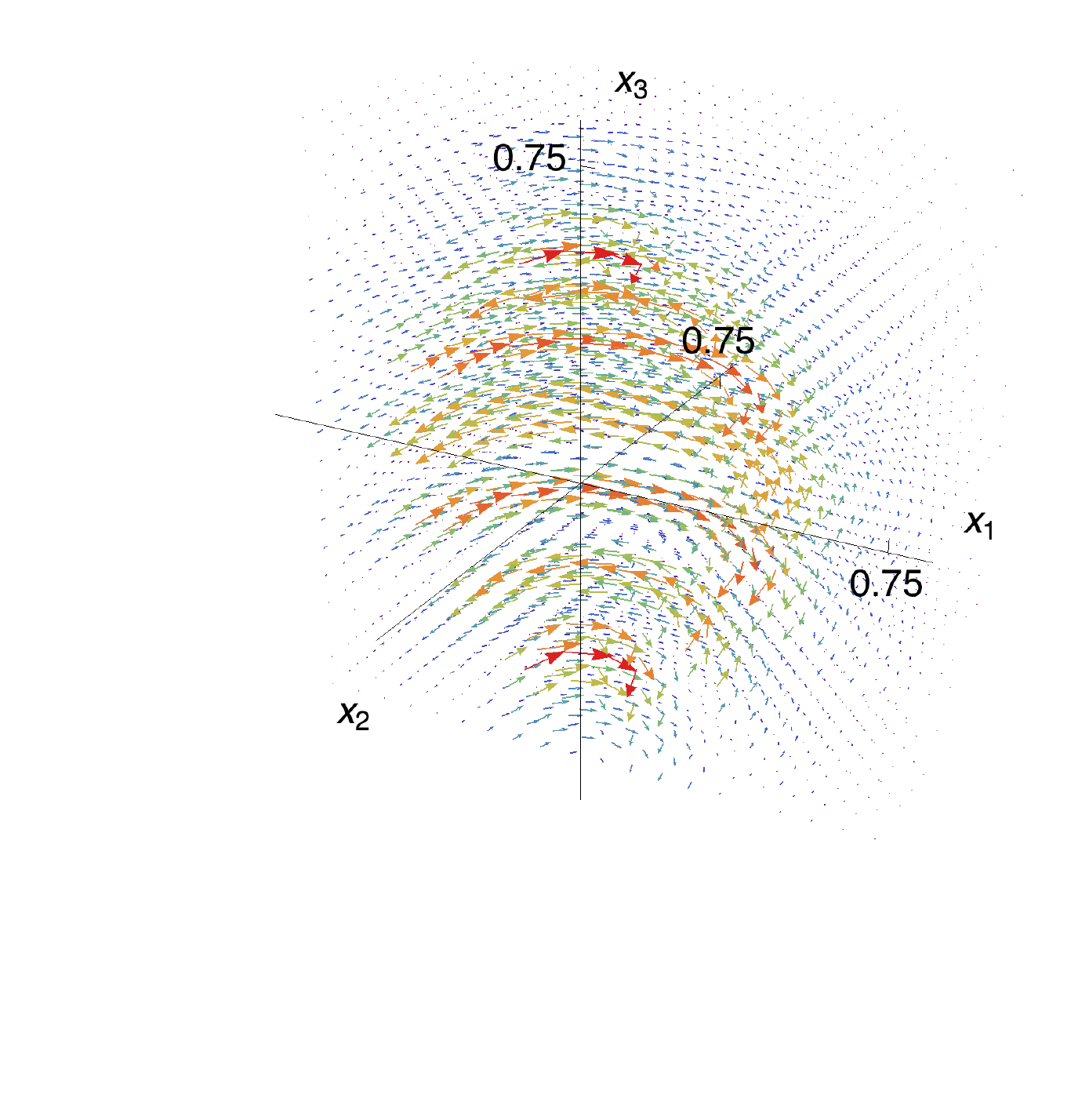}
  \caption{Divergence free wavelets in three dimensions, top row in the frequency domain and bottom row in the spatial one. The left function is isotropic around the $x_3$ axis and the other ones are directional modeling a high frequency feature across the $x_1$-$x_2$ plane.}
  \label{fig:psis:3d}
\end{figure}

\subsubsection{Directional Wavelets}
\label{sec:construction:2d:directional}

The angular windows $\hat{\gamma}_j(\theta)$, which are defined through the Fourier series coefficients $\beta_m^j$, provide directional localization in the Fourier domain.
By Eq.~\ref{eq:psi:divfree:2d:space}, the directionality is preserved by the inverse Fourier transform.
In the scalar theory, such a directional localization was first introduced for steerable wavelets~\cite{Perona1991,Freeman1991,Simoncelli1995,Perona1991} to detect oriented features and later, in the work on curvelets and similar constructions~\cite{Candes1999a,Candes2005a,Labate2005,Do2005a}, it was realized that it is also critical for the optimally efficient approximation of discontinuous signals.
See Fig.~\ref{fig:wf} for some intuition.

For divergence free vector fields, our directional, divergence free wavelets yield analogous behavior than steerable wavelets and curvelets in the scalar setting.
This can be seen by computing the frame coefficient $u_s$ for a fluid vector field $\vec{u}(x)$, i.e.
\begin{subequations}
\begin{align}
    u_s = \big\langle \vec{u}(x) , \vec{\psi}_s(x) \big\rangle
  &= \Big\langle \hat{\vec{u}}(\xi) , \hat{\vec{\psi}}_s(\xi) \Big\rangle .
\end{align}
Since both $\vec{u}(x)$ and $\vec{\psi}_s(x)$ are divergence free, and hence tangential to $S_{\vert \xi \vert}^1$, they both have $\vec{e}_{\theta_{\xi}}$ as vector component.
Using that $\vec{e}_{\theta_{\xi}}$ is a unit vector one obtains
\begin{align}
  \label{eq:2d:approx:curvelets:2}
  u_s
  &= \big\langle \hat{u}(\xi) \, \vec{e}_{\theta_{\xi}} \, , \, -i \, \hat{\gamma}_s(\theta_{\xi}) \, \hat{h}(2^{-j_s} \vert \xi \vert) \, \vec{e}_{\theta_{\xi}} \big\rangle
  \\[3pt]
  \label{eq:2d:approx:curvelets:3}
  &= \big\langle \hat{u}(\xi) \, , \, -i \, \hat{\gamma}_s(\theta_{\xi}) \, \hat{h}(2^{-j_s} \vert \xi \vert)  \big\rangle
\end{align}
\end{subequations}
with the inner product in the second line again being those for scalar functions.
Thus, with suitable angular window functions, e.g. ones corresponding to steerable filters or wavelets~\cite{Perona1991,Freeman1991,Simoncelli1995,Unser2013} (for which, at least empirically, the angular sensitivity is well established) our directional, divergence free wavelets detect oriented features in a flow.
For wavelets aligned with, e.g. the wake behind an object, the frame coefficients $u_s$ will thus have a larger magnitude than for unaligned ones.
A numerical demonstration is shown in Fig.~\ref{fig:ball}.

\begin{figure}
  \includegraphics[width=0.8\textwidth]{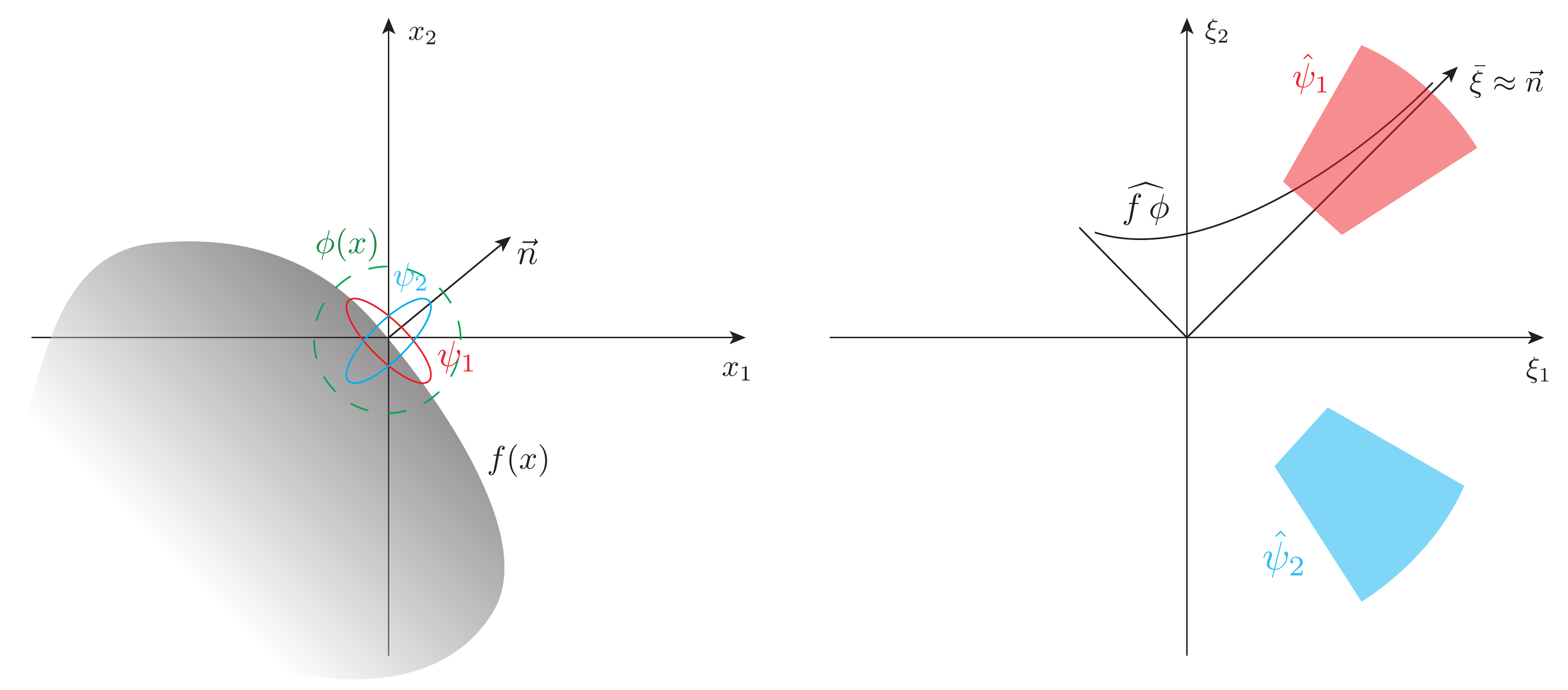}
  \caption{\emph{Left:} Discontinuous, scalar signal $f(x)$ with normal pointing towards $(1,1)$ around the origin. The ellipses indicate directional wavelets aligned (red) and unaligned (cyan) with the discontinuity. \emph{Right:} Conceptual view of the localized Fourier transform $\widehat{f \phi}$. It decays slowly only along the directions $\bar{\xi} \approx \vec{n}$. Hence, a wavelet $\hat{\psi}(\xi)$ supported sufficiently far from the origin will overlap $\widehat{f \phi}$ only when its orientation is aligned with the direction of slow decay, and hence with the normal of the discontinuity.
    We refer to~\cite[Ch. XIV]{Stein1993} and~\cite[Vol. 1, Ch. VIII]{Hoermander2004} for precise statements about the local Fourier transform of discontinuous signals and to~\cite{Candes2005a} for their approximation using wavelet-like functions.}
  \label{fig:wf}
\end{figure}

Moreover, since $\hat{\gamma}_j(\theta_{\xi}) \, \hat{h}(2^{-j} \vert \xi \vert)$ is a scalar, curvelet-like function
existing results for curvelets~\cite[Thm. 1.2, Thm. 1.3]{Candes2004} imply the following.

\begin{proposition}
  \label{prop:curvelet_approx}
  Let $\vec{u}(x) \in L_2^{\mathrm{div}}(\mathbb{R}^{2,2})$ be a $C^2$-smooth divergence free vector field away from $C^2$ discontinuities with $\mathcal{F}^{-1}(\hat{u}) \in \mathcal{E}^2(A)$~\cite[Def. 1]{Candes2004}, which we write as $\vec{u} \in \mathcal{E}_{\mathrm{div}}^2(A)$.
  When the windows $\hat{\gamma}_j(\theta_{\xi})$ and $\hat{h}(\vert \xi \vert)$ satisfy the admissibility conditions of second generation curvelets~\cite[Sec. 2]{Candes2004}, then the $n$-largest coefficient $\vert u_s \vert_{n}$ in the coefficient sequence $(\vert u_s \vert)_{n}$ satisfies
  \begin{align}
    \sup_{\vec{u} \in \mathcal{E}_{\mathrm{div}}^2(A)} \vert u_s \vert_{n} \leq C \cdot n^{-3/2} \, (\log{n})^{3/2} .
  \end{align}
\end{proposition}

Critical for the above result is that the radial and angular windows satisfy a parabolic scaling law so that the size of the wavelets in the spatial domain is approximately $2^{-j/2} \times 2^{-j}$ and they can be defined on a grid with a corresponding anisotropic spacing.
In the  frequency domain, the angular windows thereby get narrower as $j$ increases and they ``zoom in'' on a direction of slow decay.
Next to the sparsity results in Proposition~\ref{prop:curvelet_approx}, this implies a stronger notion of angular selectivity than provided by classical steerable filters and wavelets where the support of the angular window is typically independent of the radial level~\cite{Unser2013}.

Interestingly, in the vector-valued case one has an additional degree of freedom not present in the scalar setting.
For real-valued signals, the Fourier transform is symmetric around the origin and hence $\vec{e}_{\theta_{\xi}}$ can point either in the same or opposite directions in the two half-spaces, see the last two columns in Fig.~\ref{fig:psis:2d}, top.
The two choices yield starkly different spatial wavelets $\vec{\psi}_s(x)$: when $\vec{e}_{\theta_{\xi}}$ points in opposite directions one obtains a shear-like behavior while a consistent direction of $\vec{e}_{\theta_{\xi}}$ yields a localized uni-directional streaming-like flow, see again the two rightmost columns in Fig.~\ref{fig:psis:2d}.
Directional wavelets corresponding to the two behaviors are conveniently modeled using Fourier series coefficients $\beta_m$ that are nonzero only for even $m$, which yields shear-like flows, or nonzero only for odd $m$, for streaming-like flows, see Fig.~\ref{fig:beta_m}.

\subsection{Divergence free polar wavelets in 3D}
\label{sec:construction:3d}

To carry the foregoing construction over to three dimensions, we need an equivalent of the tangent vector $\vec{e}_{\theta_{\xi}}$, which was the crucial ingredient ensuring divergence freedom in $\mathbb{R}_x^2$.
By the Hedgehog theorem, however, any smoothly changing basis tangential to the sphere will become singular at some point on $S^2$.
We escape the theorem with the following construction.

\begin{proposition}
  Let $\vec{\tau}_a$ be the tangent vector field of $S^2$ induced by the longitudal geographic coordinate $\phi_a$ with respect to the $a^{\textrm{th}}$ axis, i.e. $\vec{\tau}_a = \partial / \partial \phi_a$.
  Then for every $\omega \in S^2$ the set
  \begin{align}
    H(\omega) = \big\{ \vec{\tau}_1(\omega) , \, \vec{\tau}_2(\omega) , \, \vec{\tau}_3(\omega) \big\}
  \end{align}
  forms a Parseval tight frame for $T_{\omega} S^2$.
  We will refer to $H$ as the Hedgehog frame.
\end{proposition}

The proposition can be checked by an explicit calculation and, in a more general form, can also be found in~\cite{Freeman2014}.
Intuitively, it holds because the three vectors compensate each other, e.g. $\vec{\tau}_3$ vanishes at the North pole but this is exactly where $\vec{\tau}_1$ and $\vec{\tau}_2$ form an orthonormal basis.

\begin{figure}
  \includegraphics[width=\textwidth]{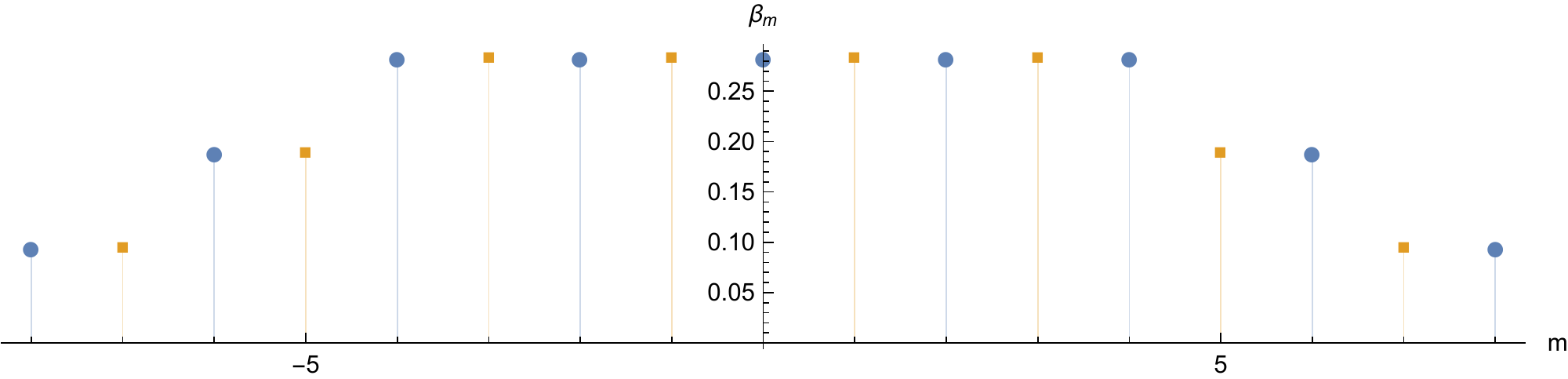}
  \caption{Examples for coefficients $\beta_m$ that define the angular localization windows $\hat{\gamma}(\theta_{\xi})$. Functions yielding a shear-like flow have $\beta_m$ that are non-zero only for even $m$ (blue) while those yielding a streaming-like flow behavior have nonzero $\beta_m$ only for odd $m$ (yellow).}
  \label{fig:beta_m}
\end{figure}

With the Hedgehog frame, we can proceed as before to construct a divergence free wavelet frame: we choose a vectorial part that is inherently tangential to the frequency sphere, namely $\vec{\tau}_a$, and obtain angular localization with a scalar window function.
Our mother wavelets are thus defined as
\begin{align}
  \label{eq:wavelets:3d:frequency}
  \hat{\vec{\psi}}_a(\xi)
  = \hat{\psi}_a(\xi) \, \vec{\tau}_a
  = -i \, \left( \sum_{l=0}^L \sum_{m=-l}^l \kappa_{lm} \, y_{lm}(\bar{\xi}) \right) \hat{h}( \vert \xi \vert) \, \vec{\tau}_a
\end{align}
with now one for each coordinate axis.
By expanding $\vec{\tau}_a$ into spherical harmonics, the spatial representation $\vec{\psi}_a(\xi)$ can be computed in closed form using a calculation analogous to those in two dimensions, see Appendix~\ref{sec:appendix:spatial:3d} for details.
It is given by
\begin{align}
  \label{eq:wavelets:3d:general}
  \vec{\psi}_3(x) =
  \frac{2}{\sqrt{3}} \sum_{l,m} & \, \kappa_{lm}
   \sum_{l_2=l-1}^{l+1} i^{l_2} \,  h_{l_2}( \vert x \vert)
  \sum_{\sigma \in \pm 1}  G_{l m;1 \sigma}^{l_2,m+\sigma} \, y_{l_2,m+\sigma}(\bar{x})
  \begin{pmatrix}
    -1
    \\[-1pt]
    i \, \sigma
    \\[-1pt]
    0
  \end{pmatrix}
\end{align}
and analogously for the other two coordinate axes.
Here the $G_{l m;1 \sigma}^{l_2,m+\sigma}$ are spherical harmonics product coefficients, see Appendix~\ref{sec:preliminaries:sh}.
We also again have that the spatial representation of an isotropic wavelet $\smash{\hat{\vec{\psi}}_a(\xi)} = \hat{h}( \vert \xi \vert) \, \vec{\tau}_a$ is isotropic in space, i.e. $\vec{\psi}_a(x) = -\sqrt{2 / \pi} \, h_1 \big(\vert x \vert\big) \, \vec{\tau}_a$, with isotropy now around the $a^{\textrm{th}}$ axis.

\subsubsection{Properties of wavelets}
As in the scalar case, under suitable admissibility conditions the just defined wavelets form a Parseval tight frame for the space $L_2^{\textrm{div}}(\mathbb{R}^{3,3})$ of divergence free vector fields in $\R^3$ with finite $L_2$-norm.

\begin{proposition}
  \label{prop:tight_frame:3d}
  Let $w_{j,t}$ be the $(L_j+1)^2$-dimensional vector formed by the rotated angular localization coefficients $\kappa_{lm}^{j,t} = \sum_{m' = -l}^l W(\lambda_t)_{l,m}^{m'} \kappa_{l,m'}$ for a localization window centered at $\lambda_t$, where $W(\lambda_t)_{l,m}^m$ is the Wigner-D matrix implementing rotation in the spherical harmonics domain, and let $G^{lm}$ be the $(L_j+1)^2 \times (L_j+1)^2$ dimensional matrix formed by the spherical harmonics product coefficients for fixed $(l,m)$.
  When the Cald{\`e}ron condition $\sum_{j \in \mathbb{Z}} \big\vert \hat{h}( 2^{-j} \vert \xi \vert ) \big\vert^2 = 1$, $\forall \xi \in \mathbb{R}^3$ is satisfied and $\delta_{l,0} \delta_{m,0} = \sum_{t=0}^{M_j} w_{j,t} \, G^{lm} \, w_{j,t}$ (where $\delta_{i,j}$ is the Kronecker delta) then any $\vec{u}(x) \in L_2^{\textrm{div}}(\mathbb{R}^{3,3})$ has the representation
  \begin{subequations}
  \begin{align}
    \label{eq:tight_frame:3d}
    \vec{u}(x) = \sum_{a=1}^3 \sum_{j \in \mathbb{Z}} \sum_{k \in \mathbb{Z}^3} \sum_{t=1}^{M_j} \big\langle \vec{u}(y) , \vec{\psi}_{j,k,t}^a(y) \big\rangle \, \vec{\psi}_{j,k,t}^a(x)
  \end{align}
  with frame functions
  \begin{align}
    \vec{\psi}_{j,k,t}^a(x) = \frac{2^{3j/2}}{(2\pi)^{3/2}} \, \vec{\psi}_a \big( R_{\lambda_t} (2^{j} x -  k ) \big) ,
  \end{align}
  \end{subequations}
  for $\vec{\psi}_a(x)$ defined in Eq.~\ref{eq:wavelets:3d:general} and $R_{\lambda_t}$ the rotation from the North pole to $\lambda_t$.

\end{proposition}

\begin{proof}
Taking the Fourier transform of Eq.~\ref{eq:tight_frame:3d}, using Parseval's theorem, and with Eq.~\ref{eq:wavelets:3d:frequency} we obtain
\begin{subequations}
\begin{align}
  \hat{\vec{u}}(x)
  &= \sum_{a=1}^3 \sum_{j \in \mathbb{Z}} \sum_{k \in \mathbb{Z}^3} \sum_{t=1}^{M_j} \Big\langle \hat{\vec{u}}(\eta) \, , \, \hat{\psi}_{j,k,t}(\eta) \, \vec{\tau}_a(\eta) \Big\rangle \, \hat{\psi}_{j,k,t}^a(\xi) \, \vec{\tau}_a(\xi)
  \\
   &= \sum_{a=1}^3 \underbrace{\sum_{j \in \mathbb{Z}} \sum_{k \in \mathbb{Z}^3} \sum_{t=1}^{M_j} \Big\langle \hat{u}_a(\eta) \, , \, \hat{\psi}_{jkt}(\eta) \Big\rangle \, \hat{\psi}_{jkt}^a(\xi)}_{\textrm{scalar polar wavelet frame for $\hat{u}_a(\eta)$}} \, \vec{\tau}_a(\xi) .
\end{align}
\end{subequations}
Here $\hat{u}_a(\eta) = \hat{\vec{u}}(\eta) \cdot \vec{\tau}_a(\eta)$, i.e. the pointwise projection of the vector $\hat{u}_a(\eta)$ onto the Hedgehog frame vector $\vec{\tau}_a(\eta)$ at $\eta$, so that $\hat{\vec{u}}(\eta) = \sum_a \hat{u}_a(\eta) \, \vec{\tau}_a(\eta)$ by the tight frame property of the $\vec{\tau}_a$.
For the coordinate function $\hat{u}_a(\xi)$ one has a scalar polar wavelet frame expansion.
Since the proposition uses the same conditions that are required for scalar polar wavelets in $\R^3$ to form a Parseval tight frame, the result follows from the existing proofs for the scalar case.
See~\cite[Proposition 4.1]{Ward2014} for the isotropic result and~\cite[Theorem 2.4]{Ward2014} for those for the anisotropic case.
\end{proof}

The three dimensional divergence free wavelets have the same useful properties we already observed in two dimensions, in particular they are divergence free in an ideal, analytic sense, they have closed form expressions in frequency and space, a multi-resolution structure and fast transforms.

\begin{figure}[t]
  \includegraphics[width=0.1725\textwidth]{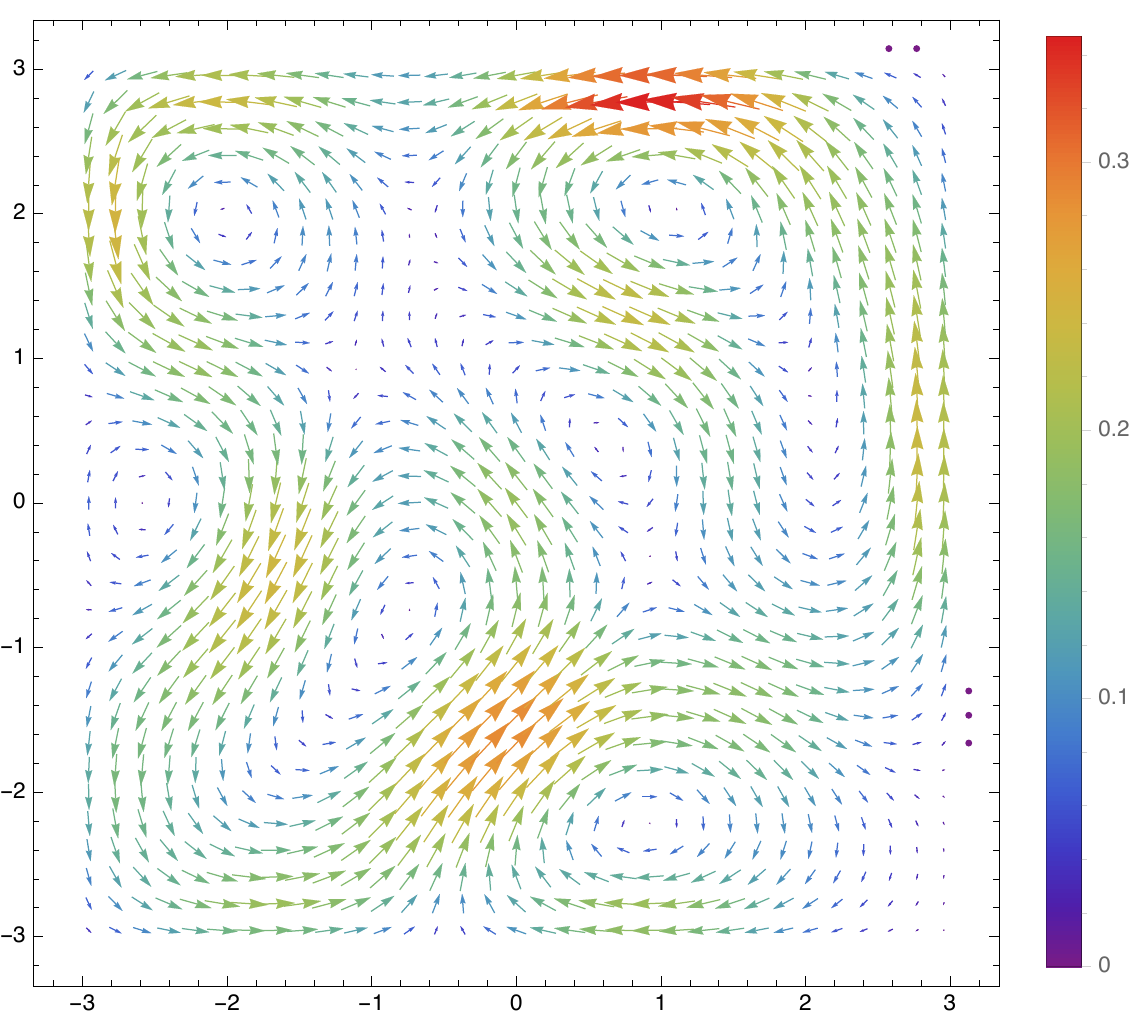}
  \includegraphics[width=0.155\textwidth]{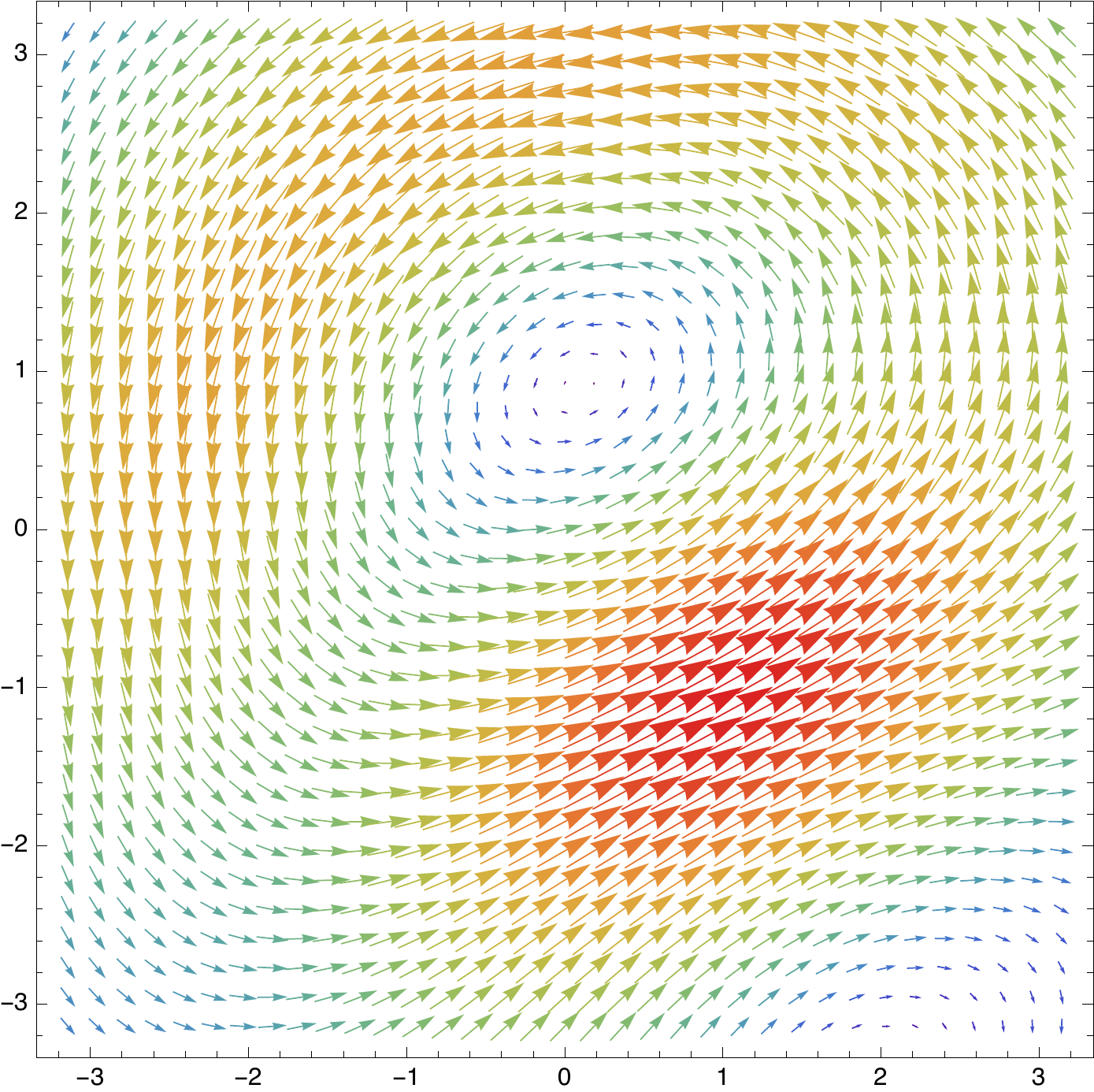}
  \includegraphics[width=0.155\textwidth]{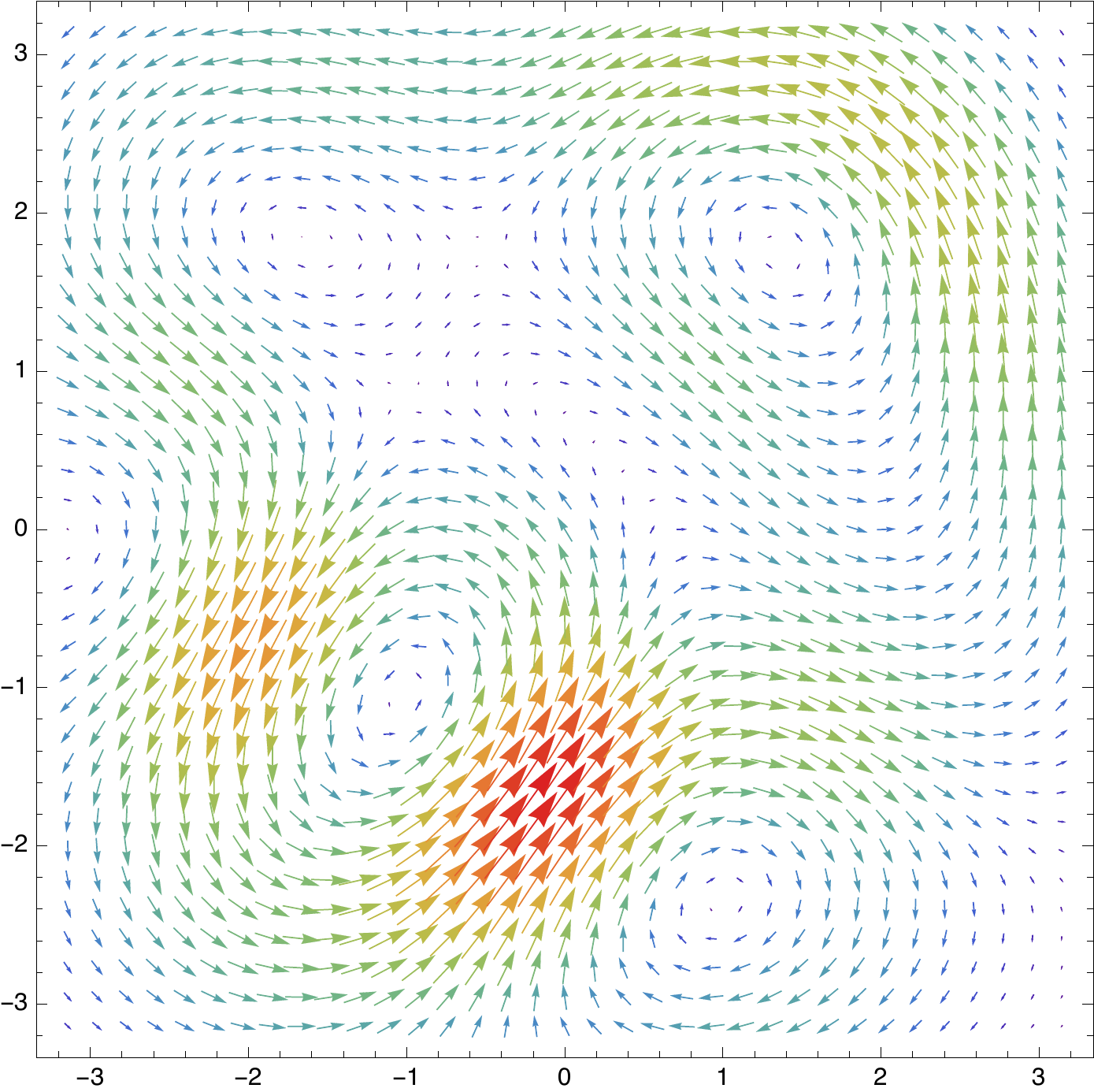}
  \includegraphics[width=0.155\textwidth]{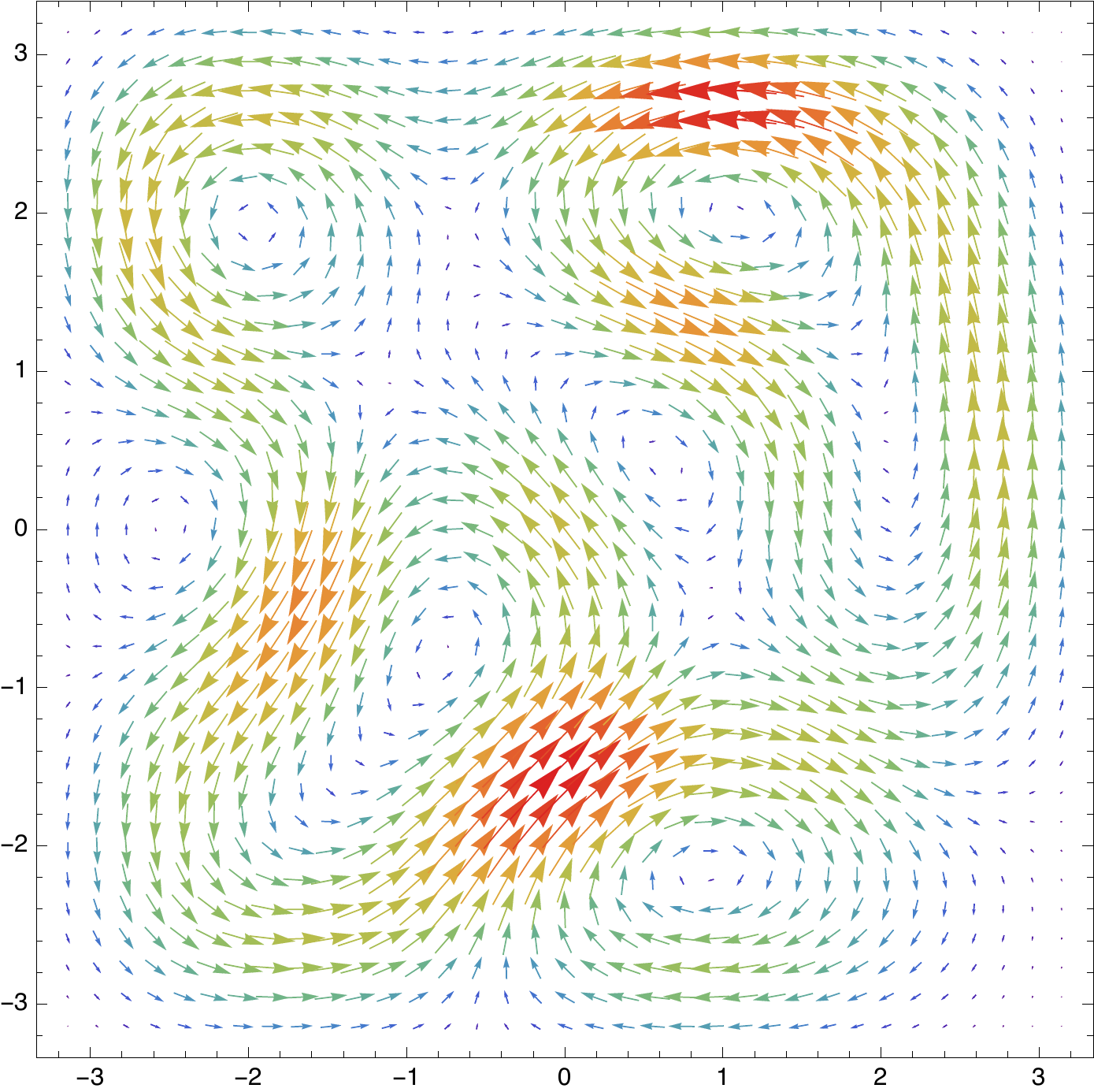}
  \includegraphics[width=0.155\textwidth]{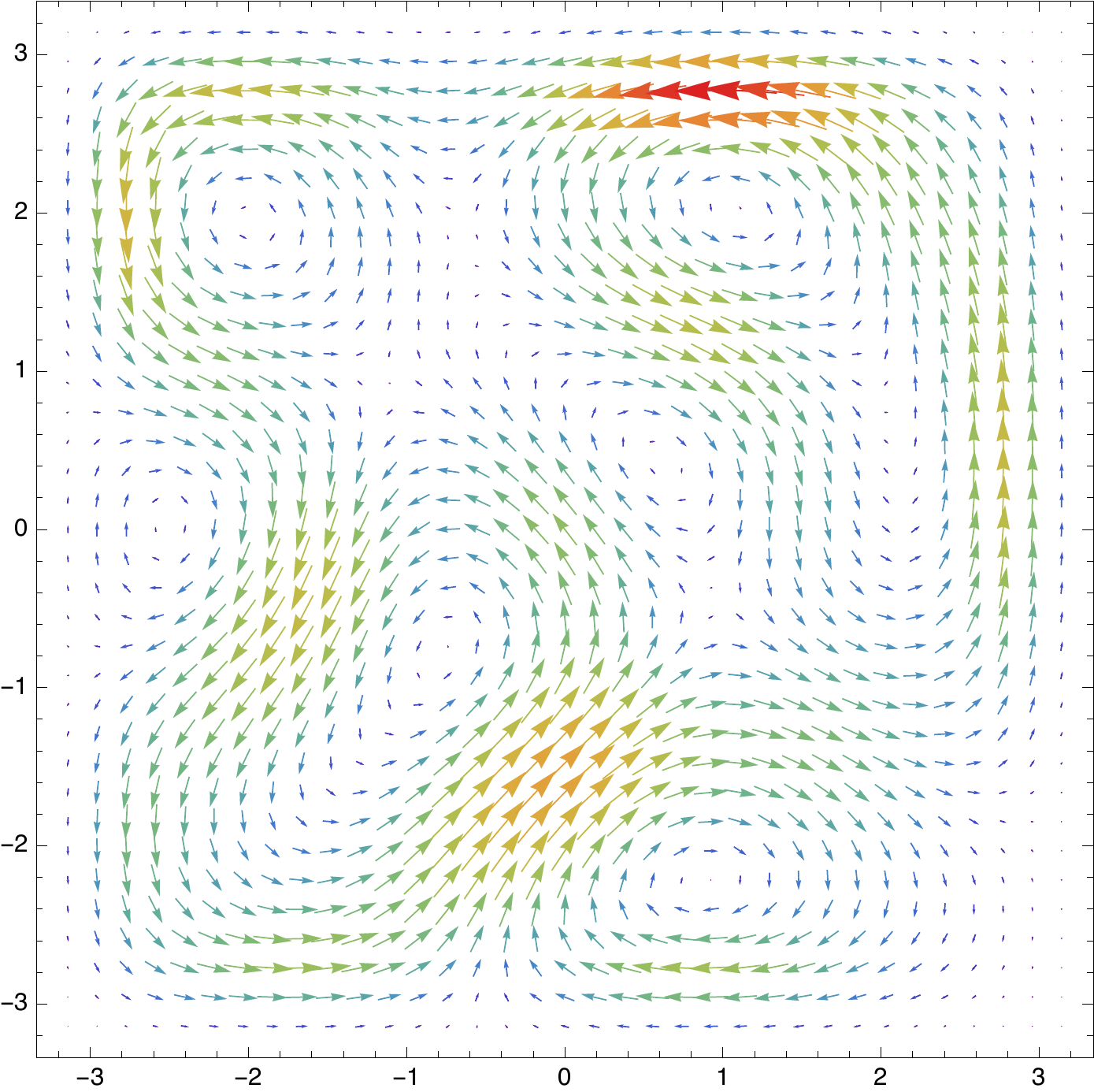}
  \includegraphics[width=0.155\textwidth]{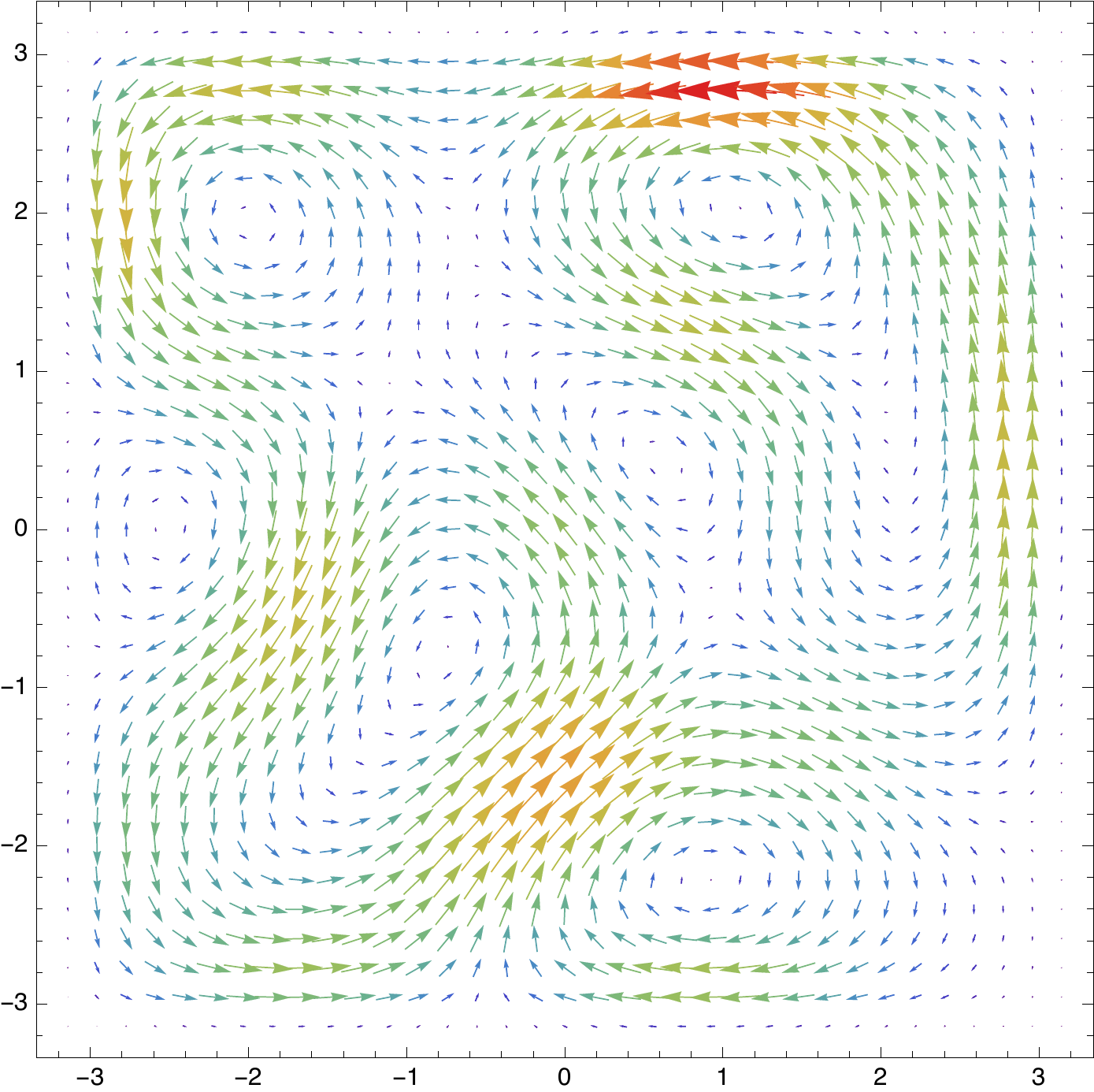}
  \\
  \includegraphics[width=0.1725\textwidth]{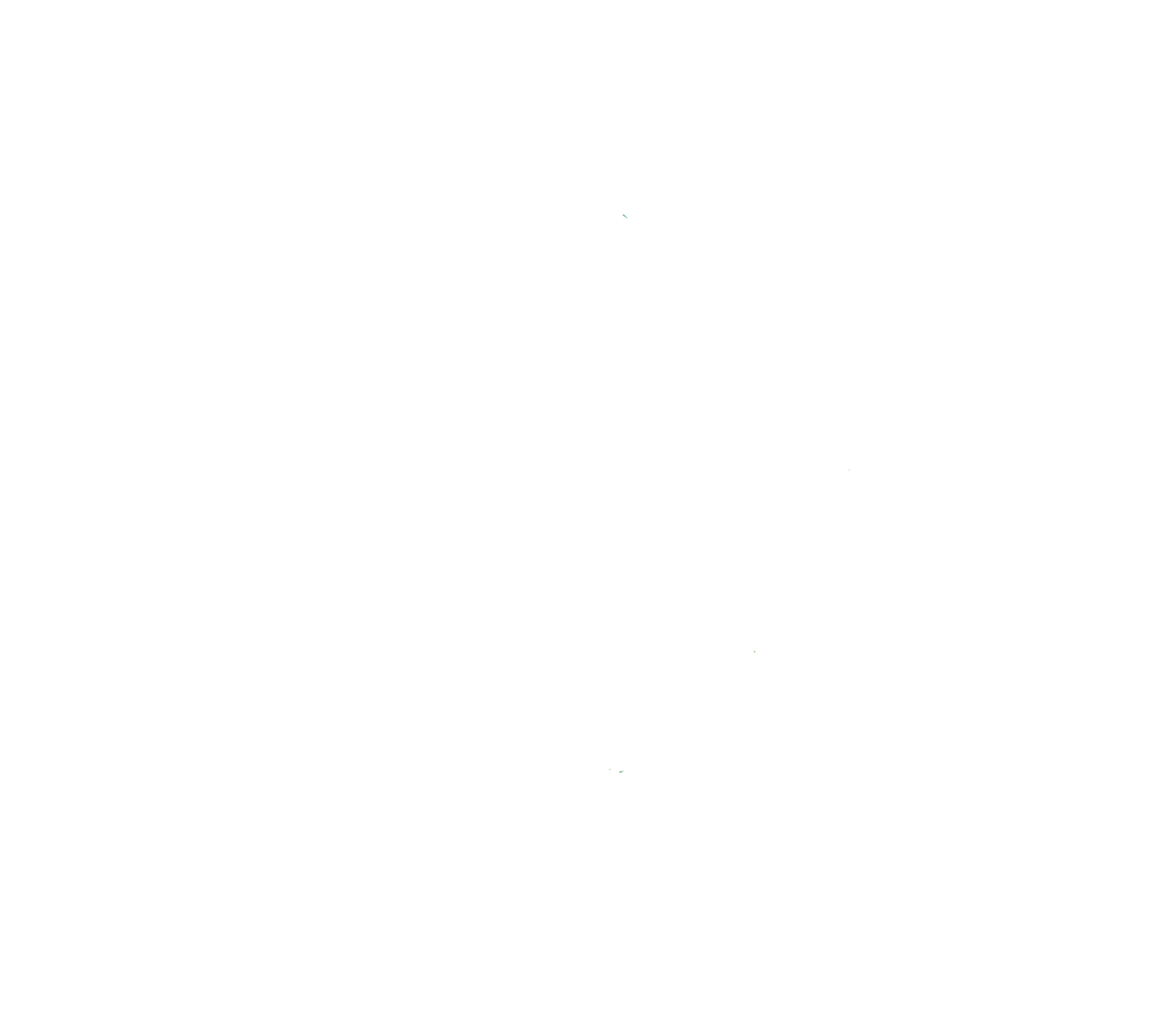}
  \includegraphics[width=0.155\textwidth]{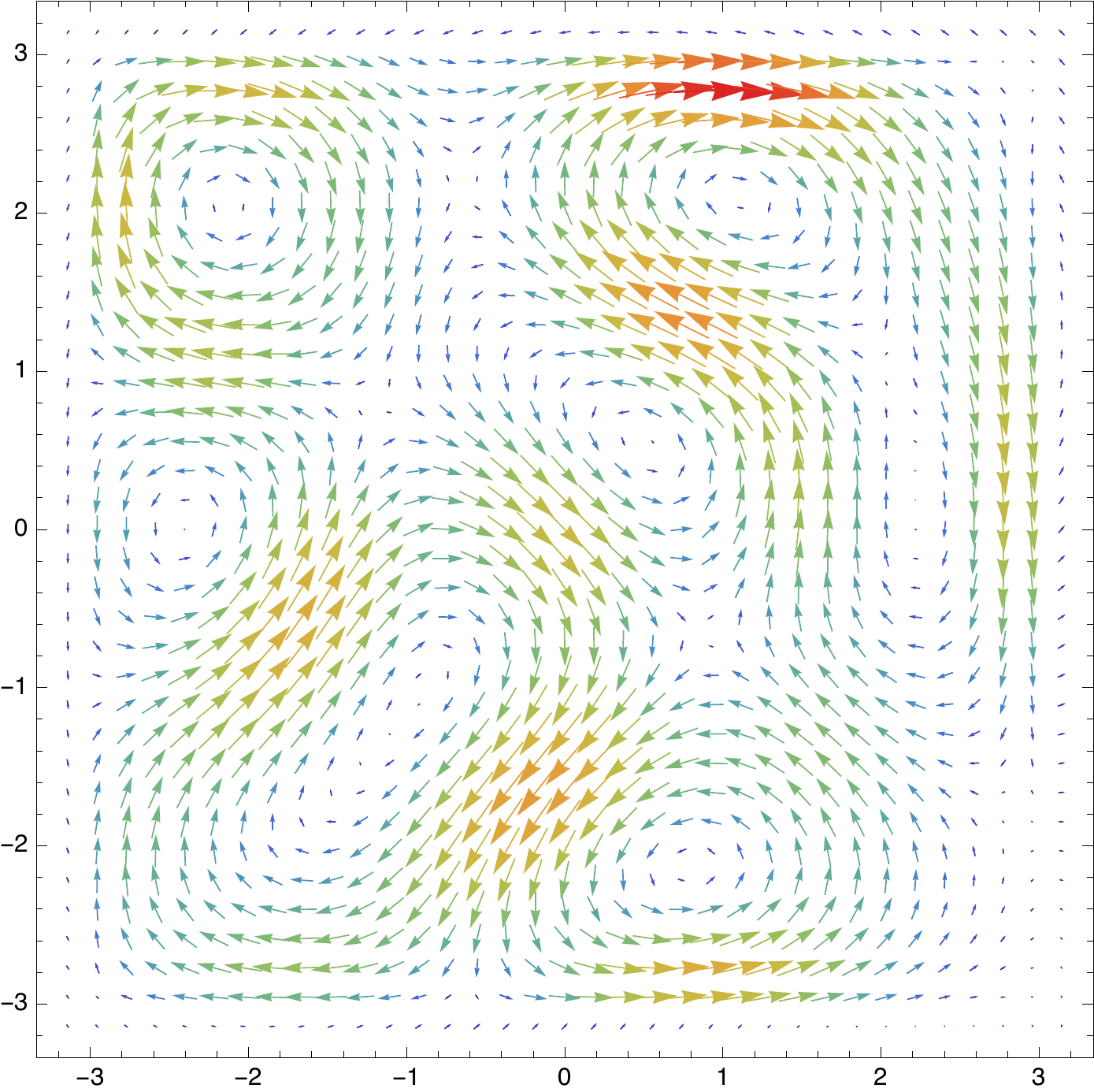}
  \includegraphics[width=0.155\textwidth]{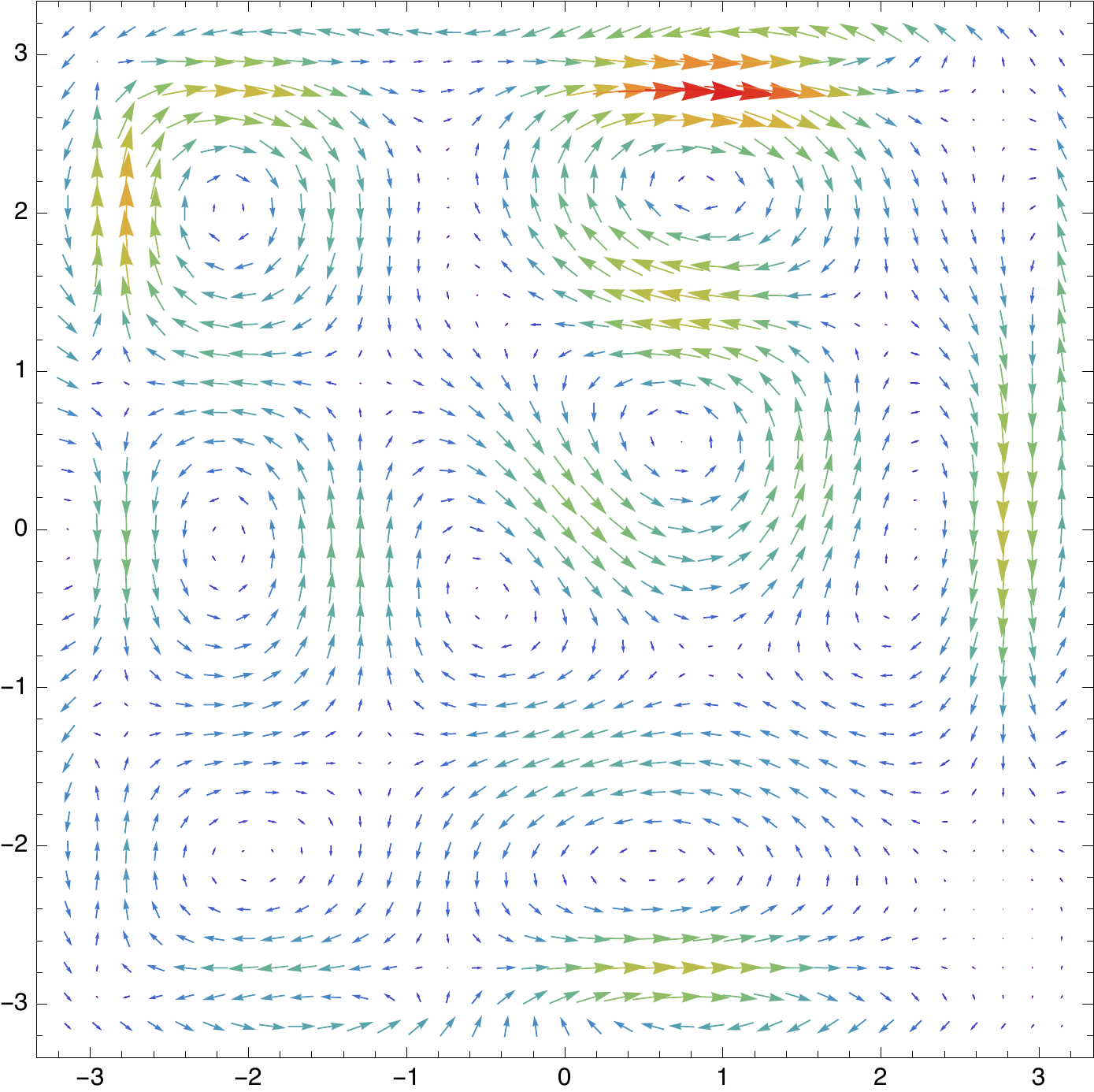}
  \includegraphics[width=0.155\textwidth]{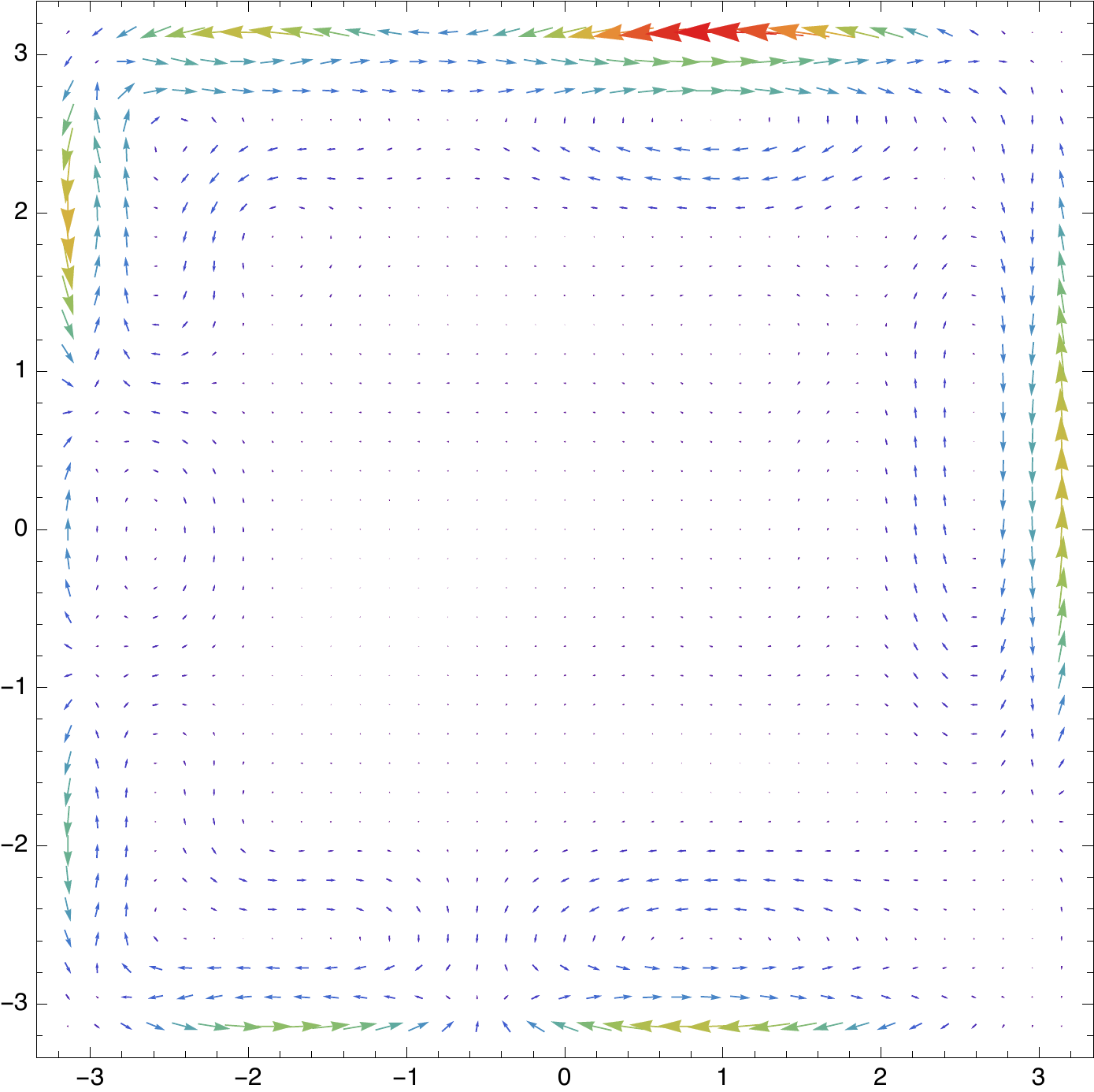}
  \includegraphics[width=0.155\textwidth]{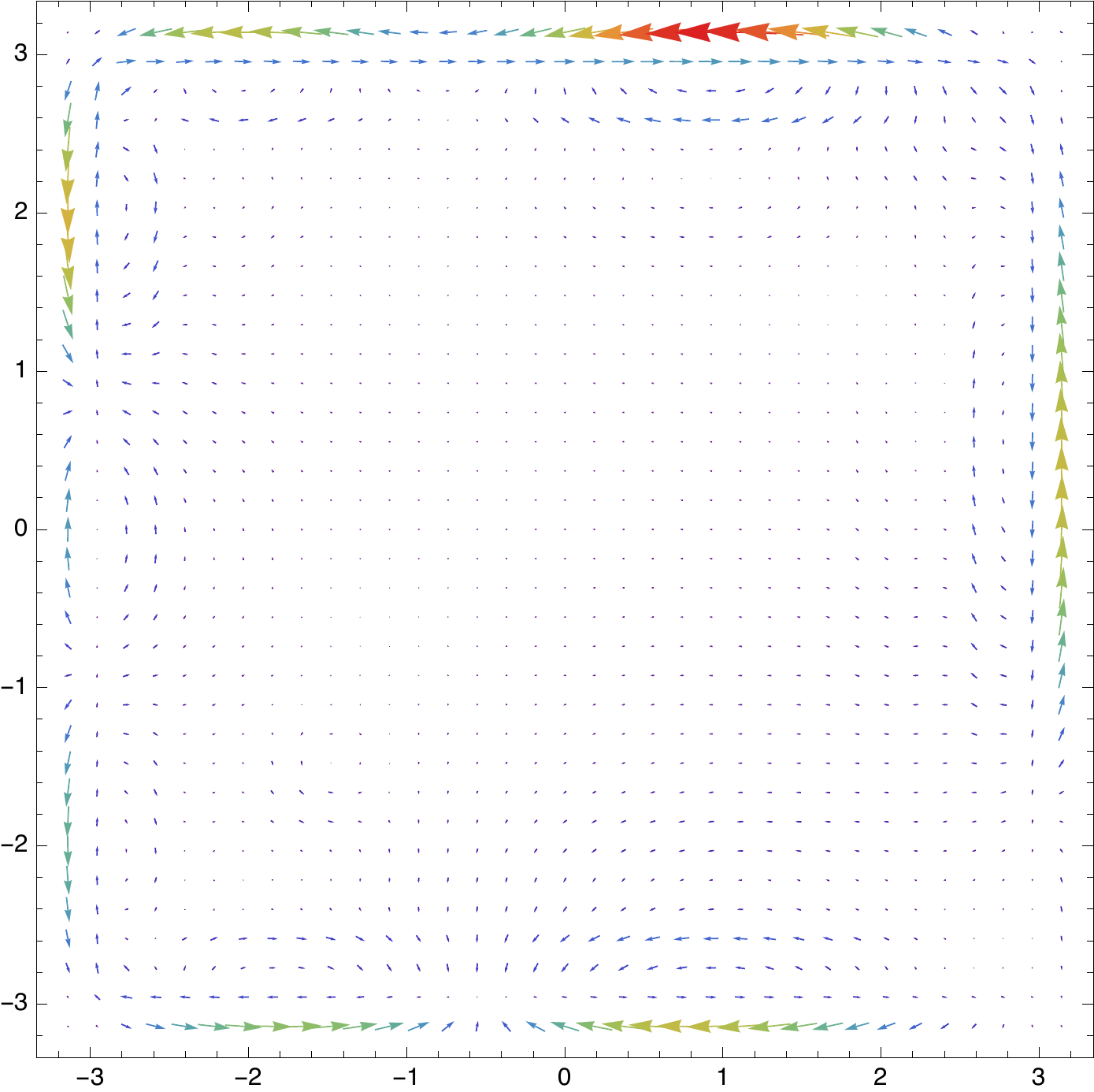}
  \includegraphics[width=0.155\textwidth]{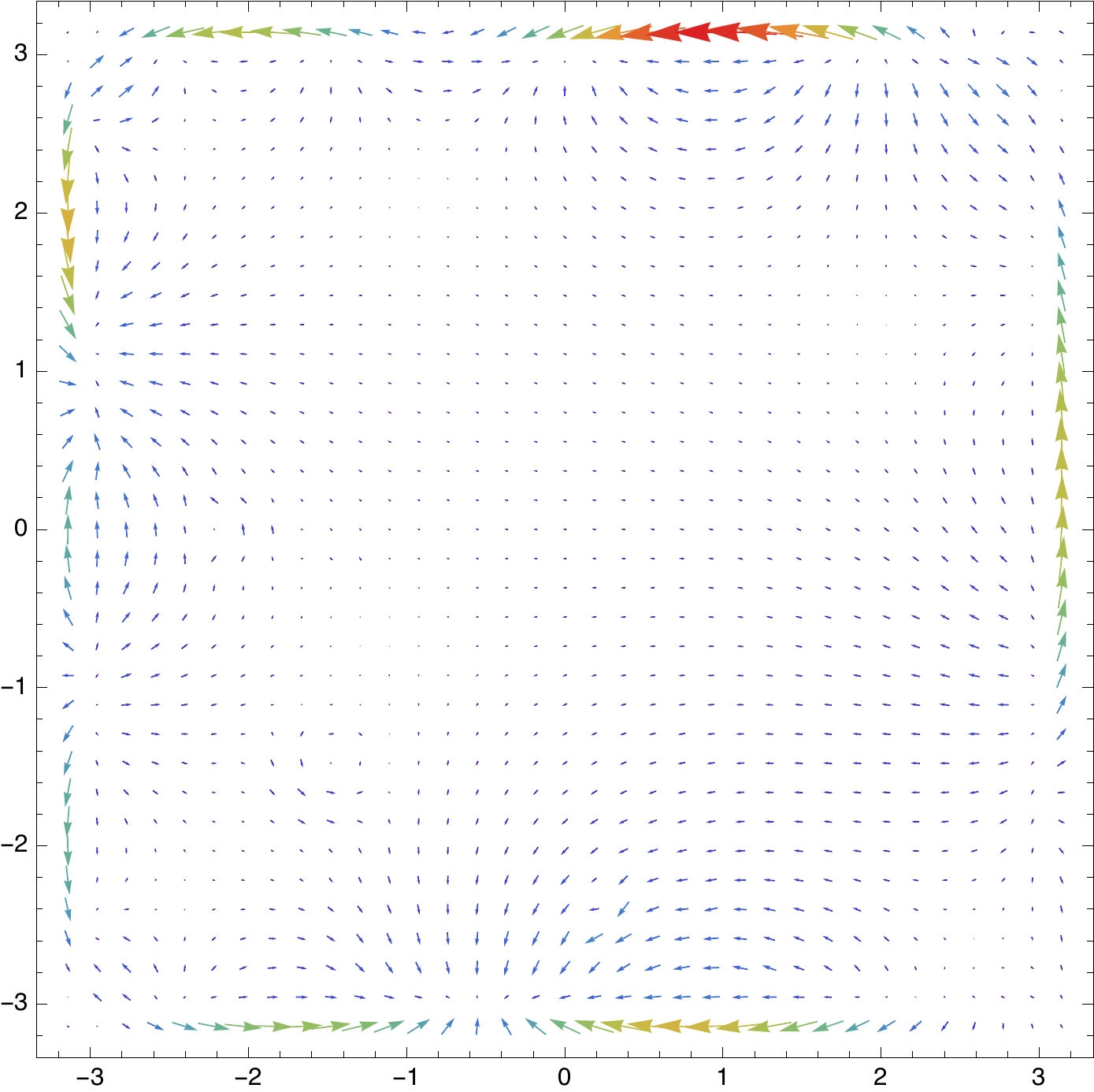}
  \caption{Level-by-level reconstruction of a divergence free fluid vector field (top left) with free slip boundary conditions with only isotropic basis functions from $j=-1$ (second column from left), i.e using only scaling functions, to $j=3$ (right). The second row shows the error in the reconstructed vector field. For $j=3$ the vector field is well represented except close to the boundaries where one still has a relatively large error since, firstly, the number of levels is too small to fully resolve the signal and, secondly, the apron region around the flow domain is relatively small.
  With anisotropic functions the error can be substantially reduced, see Fig.~\ref{fig:sig:boundary}.}
  \label{fig:rec_err_sig}
\end{figure}

\subsection{Fast Transform}
\label{sec:div_free:fast_transform}

The fast wavelet transform provides an efficient means to obtain the wavelet coefficients of a signal, in our case a fluid velocity vector field, without the need to use numerical quadrature.
It proceeds by splitting the signal $f_{j+1}$ at scale $j+1$ into a low frequency part $f_j$ on level $j$, represented by scaling function coefficients, and a high frequency part, represented by wavelet coefficients.
By linearity it suffices to determine the projection for the basis functions, which are described by the filter taps
\begin{subequations}
\begin{align}
  \alpha_{j,k} &= \left\langle \vec{\phi}_{j,0}(x) , \vec{\phi}_{j+1,k}(x) \right\rangle
  \\[5pt]
  \beta_{j,k,t} &= \left\langle \vec{\psi}_{j,0,t}(x) , \vec{\phi}_{j+1,k}(x) \right\rangle
  \\[5pt]
  \gamma_{j,k,t,t'} &= \left\langle \vec{\psi}_{j,0,t}(x) , \vec{\psi}_{j+1,k,t'}(x) \right\rangle
  \\[5pt]
  \delta_{j,k,t} &= \left\langle \vec{\psi}_{j,0,t}(x) , \vec{\phi}_{j,k}(x) \right\rangle
\end{align}
\end{subequations}
with the $\gamma_{j,k,t,t'}$ only necessary for analysis, i.e. projecting a signal into the wavelet basis, and $\delta_{j,k,t}$ only for the reconstruction, see Appendix~\ref{sec:appendix:fast_transform} for details.
Our construction enables the computation of the filter taps in closed form, avoiding expensive numerical quadrature and ensuring that they can be obtained to any required accuracy.

Without loss of generality, we will consider the calculation for the $\beta_{j,k}$ in two dimensions.
Using Parseval's theorem we obtain
\begin{subequations}
\begin{align}
  \beta_{j,k} = \left\langle \hat{\vec{\psi}}_{j,0}(\xi)  \, ,  \, \hat{\vec{\phi}}_{j,k+1}(\xi) \right\rangle
\end{align}
and since both $\hat{\vec{\psi}}_{j,0}(\xi)$ and $\hat{\vec{\phi}}_{j,k+1}(\xi)$ have the same vectorial part, namely $\vec{e}_{\theta}$, this reduces to
\begin{align}
  \beta_{j,k} = \left\langle \hat{\psi}_{j,0}(\xi)  \, ,  \, \hat{\phi}_{j,k+1}(\xi) \right\rangle .
\end{align}
where the inner product is now those for scalar functions and the filter taps computation is the same as for scalar polar wavelets~\cite{Lessig2018a}.
In particular, expanding $\hat{\psi}_{j,0}(\xi)$ and $\hat{\phi}_{j,k+1}(\xi)$ using their definitions one obtains
\begin{align}
  \label{eq:fast_transform:beta}
  \beta_{j,k}
  &= 2\pi \sum_{n} i^n \, \beta_n \, \int_{\R_{\vert \xi \vert}^+} \hat{h}(2^{-j} \vert \xi \vert) \, \hat{g}(2^{-j+1} \vert \xi \vert) \, J_m(\vert \xi \vert \, \vert k \vert) \, \vert \xi \vert \, d\vert \xi \vert
  \\[5pt]
  &= 2\pi \sum_{n} i^n \, \beta_n \, B_{j}(k)
\end{align}
\end{subequations}
with $B_{j}(k)$ having a closed form expression and being scale invariant, see the reference implementation.
Hence, the $\beta_{j,k}$ can be evaluated using a finite computation and to any required accuracy.
An analogous calculation applies to the $\alpha_{j,k}$, $\gamma_{j,k,t,t'}$ and $\delta_{j,k,t}$.
In three dimensions, the tightness of the Hedgehog frame allows a similar reduction of the vector-valued case to the computation of filter taps for scalar polar wavelets.

Since our wavelets do not have compact support in space, the fast transform requires formally infinite sums.
They become finite in the finite precision available on practical computers but the filters still have a very large number of taps.
One thus typically truncates the filters, which, however, also results in reduced accuracy.
A computation of the frame coefficients in the frequency domain based on the FFT can hence be more accurate and efficient than the fast transform.
See again Appendix~\ref{sec:appendix:fast_transform} for a more detailed discussion.

\section{Experiments}
\label{sec:experiments}

\begin{figure}
  \begin{center}
  \includegraphics[width=0.65\columnwidth]{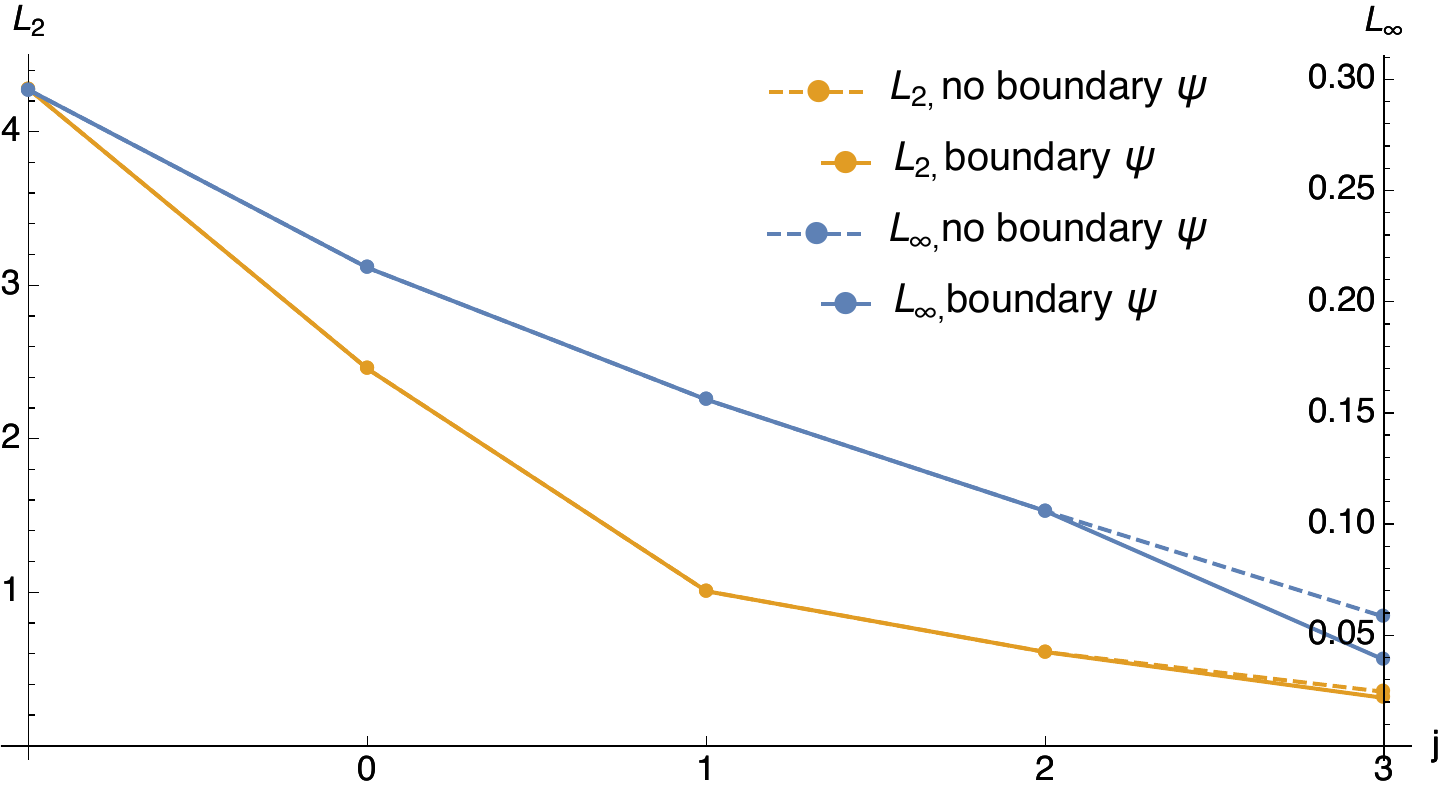}
  \end{center}
  \caption{$L_{2}$ (yellow) and $L_{\infty}$ error rates for the level-by-level reconstruction shown in Fig.~\ref{fig:rec_err_sig}.}
  \label{fig:rec_err_sig:errrates}
\end{figure}

\subsection{Implementation of wavelets}
\label{sec:experiments:implementation}

We developed a reference implementation of our divergence free polar wavelets in Mathematica.\footnote{\url{http://graphics.cs.uni-magdeburg.de/projects/psidiv/mathematica.zip}}
For the radial window function we used the Portilla-Simoncelli window~\cite{Portilla2000}, which yields closed form expressions for the wavelets in the spatial domain~\cite{Lessig2018a}.
In two dimensions, the angular window functions $\hat{\gamma}_j(\theta)$ are an extension of those used in~\cite{Lessig2018a} which are based on the wavelets for the circle developed by Walter and Cai~\cite{Walter1999}, see Fig.~\ref{fig:beta_m}.
Since for the two different types of directional wavelets only even and odd Fourier series coefficients are nonzero, respectively, they can be combined without interference and it is straight forward to construct a tight frame containing both types of boundary functions.
Note that our window functions $\hat{\gamma}_j(\theta)$ are described by finite Fourier series and hence do not satisfy the assumptions from Proposition~\ref{prop:curvelet_approx} where the angular window is compactly supported; however, they can be understood as a bandlimited approximation and we chose them such that the wavelets approximately satisfy the parabolic scaling law that is the critical ingredient for optimal approximation rates.
In three dimensions, the angular windows are those from~\cite{McEwen2016} for wavelets on the sphere and the orientations, which are non-trivial because on $S^2$ one has no equivalent of a uniform grid, are a custom construction ensuring that the discrete spherical wavelets form a tight frame and satisfy the requirements in Proposition~\ref{prop:tight_frame:3d}~\cite{Lessig2018a}.

\begin{figure}
  \includegraphics[width=0.32\columnwidth]{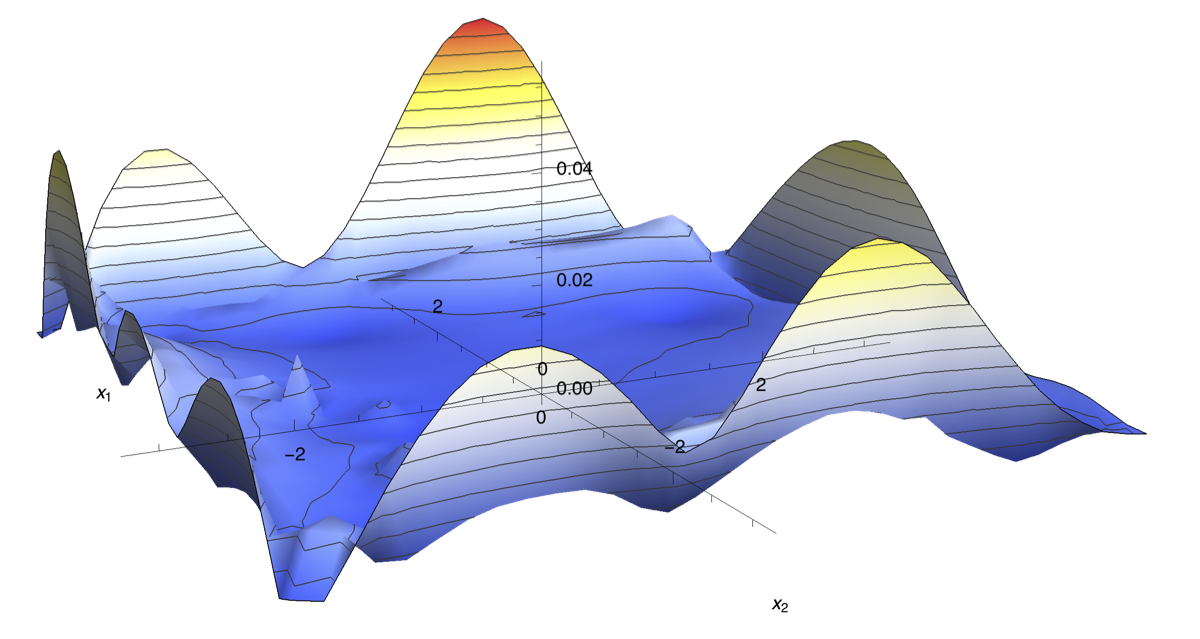}
  \includegraphics[width=0.32\columnwidth]{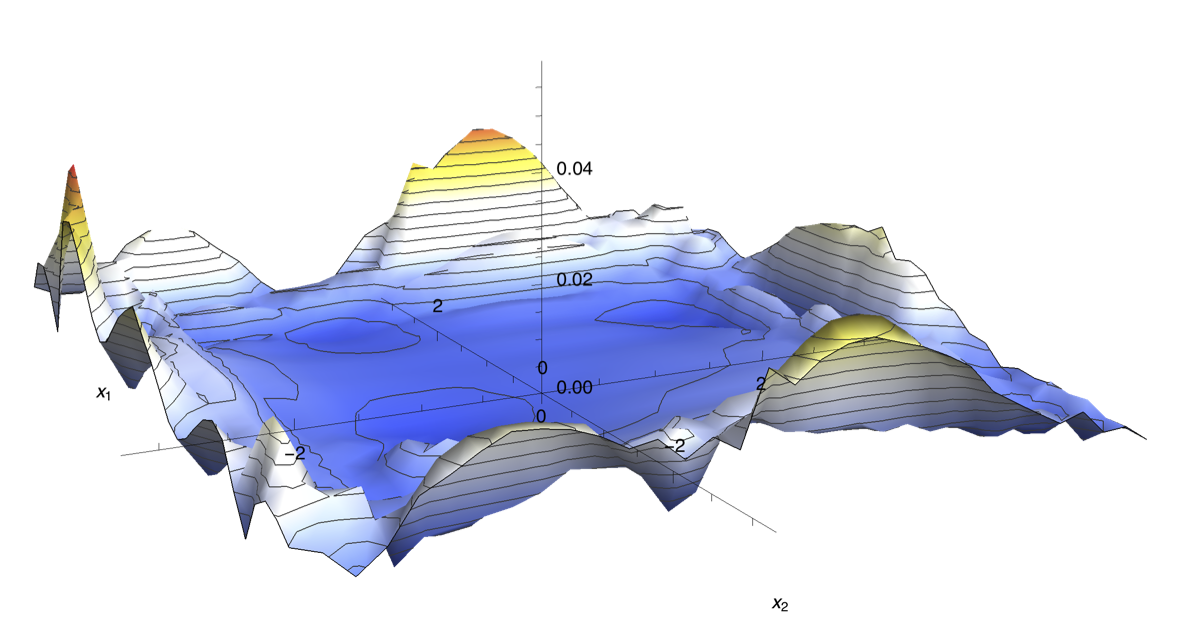}
  \includegraphics[width=0.32\columnwidth]{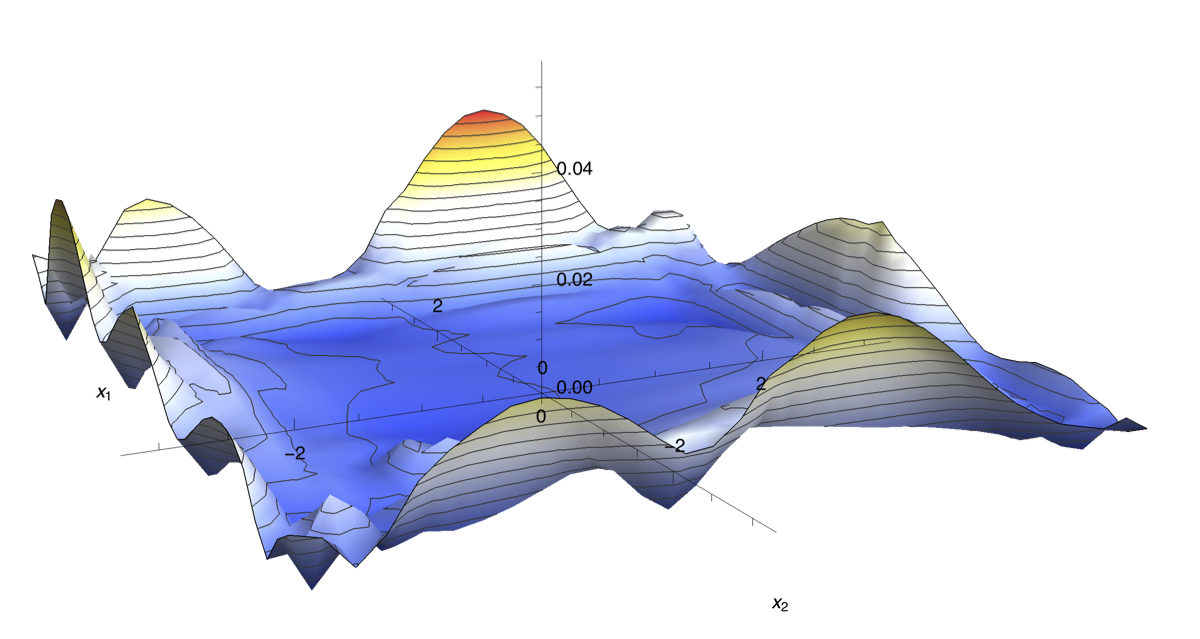}
  \caption{Reconstruction error for the fluid velocity vector field in Fig.~\ref{fig:rec_err_sig} with different functions on the finest level $j=3$: only isotropic functions (left), with all directional functions (middle), with only the directional functions aligned with the boundary (right). The middle and right plot differ only slightly. One can hence predict the sparsity, i.e. the non-negligible frame coefficients, when the boundary orientation is known and does not have to compute all of them.}
  \label{fig:sig:boundary}
\end{figure}

\subsection{Analysis of fluid velocity vector fields}
\label{sec:experiments:analysis}

In our experiments we considered two fluid velocity vector fields, one that is divergence free in the analytic sense, see Fig.~\ref{fig:rec_err_sig}, and a classical flow around a ball obtained with an FEM simulation that is approximately divergence free, see Fig.~\ref{fig:ball}.

The ideal divergence free flow was obtained as a linear combination of divergence free eigenfunctions of the vector-Laplacian for the square.
In the experiment we first studied the approximation power with only isotropic functions.
Fig.~\ref{fig:rec_err_sig} shows that with four levels the velocity vector field is already well resolved except around the boundaries of the domain.
There the vector field is discontinuous (when considered in $\R^2$) and the small number of levels is not sufficient to fully resolve the signal.
Additionally, the apron region around the flow domain is relatively small.
We also studied the approximation with anisotropic functions on the finest level.
The $L_2$ and $L_{\infty}$ error rates for the experiment as a function of the level and with and without directional wavelets on the finest level are provided in Fig.~\ref{fig:rec_err_sig:errrates}.
Fig.~\ref{fig:sig:boundary} demonstrates that, thanks to the directional selectivity, at the boundary the wavelets aligned with the discontinuity are sufficient to reconstruct the signal.
Hence, when the boundary of a flow domain is known one can exploit it to compute immediately the sparse signal representation.
We want to explore this and its application to fluid simulation in more detail in future work.

Fig.~\ref{fig:rec_err_sig:errrates} and Fig.~\ref{fig:sig:boundary} show that the oriented wavelets themselves only yield a modest error reduction.
However, their efficiency in the sense of Proposition~\ref{prop:curvelet_approx}, i.e. the sparsity of the signal representation, also comes from a less dense sampling along the elongated direction of the wavelet and not directly from lower error rates.
In the original work on scalar curvelets, the anisotropic wavelet spacing relied on angular windows $\hat{\gamma}_j(\theta_{\xi})$ with compact support.
Numerical experiments indicate that such a spacing can be used to good approximation also for the bandlimited, i.e. non-compactly supported, angular windows employed in our experiments.
Will study this in more detail in future work.

\begin{figure}
  \includegraphics[trim={1740pt 280pt 0pt 410pt},clip,width=\columnwidth]{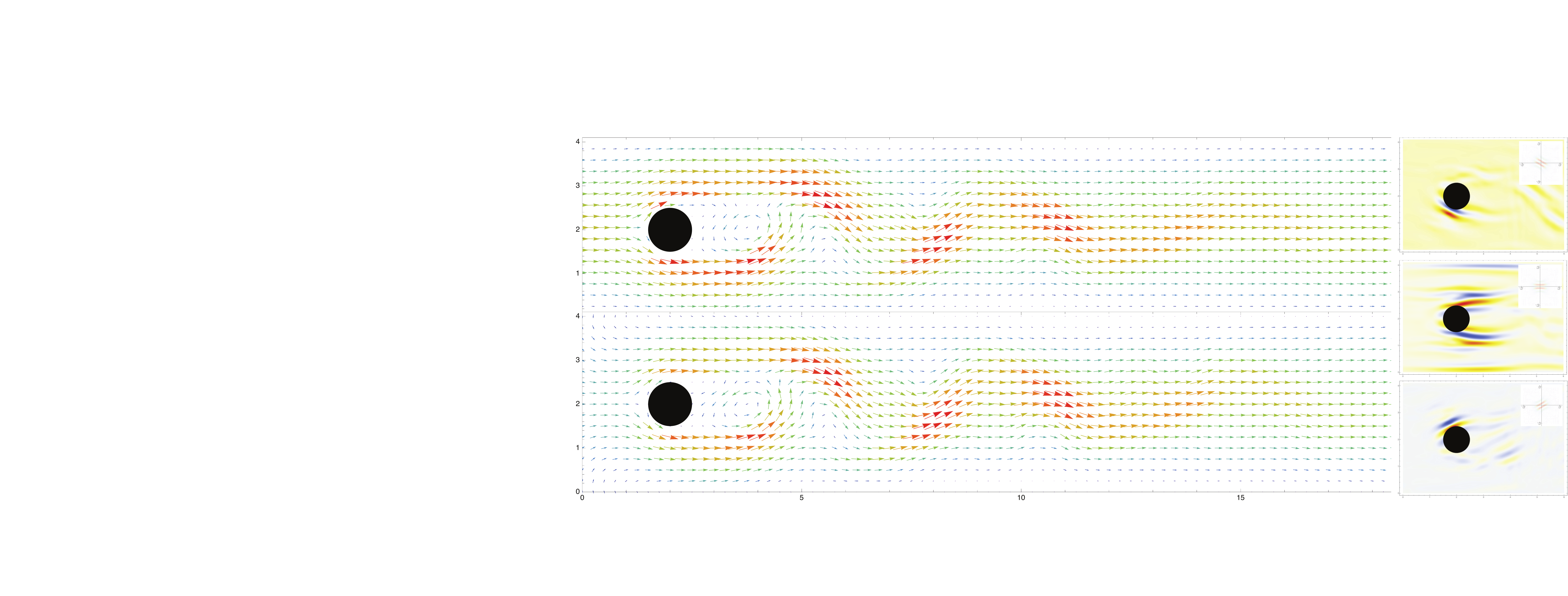}
  \caption{The classical fluid flow around a ball with a von K{\'a}rm{\'a}n vortex street. Data set from an FEM simulation (top) and wavelet reconstruction with three levels (bottom). Slight differences, e.g. on the left, are a consequence of the input data only being approximately divergence free. The right column depicts basis function coefficients on the finest level for the orientations in the inset in the vicinity of the ball. The plots validate that our directional wavelets indeed provide good angular selectivity and yield sparse representations.}
  \label{fig:ball}
\end{figure}

For the flow around the ball we studied the sparsity and the potential of the wavelets to detect specific flow features.
The coefficient plots in Fig.~\ref{fig:ball} demonstrate that our directional wavelets provide good angular selectivity and yield a correspondingly high sparsity.
They also demonstrate that directional wavelets are not only useful for modeling boundaries but also for other highly anisotropic features such as a von K{\'a}rm{\'a}n vortex street.

\section{Conclusion}
\label{sec:conclusion}

We introduced a Parseval tight frame of divergence free wavelets that, next to other desirable properties, provides well defined angular localization.
We also showed that the wavelets have analytic expressions in frequency and space and that the filter taps for the fast transform can be computed in closed form and to arbitrary precision.
The beneficial properties of the wavelets result from a construction that intrinsically ensures that they are tangential to the frequency sphere, and hence divergence free in space, and from using windows functions defined in spherical coordinates.
Our windows are bandlimited in the frequency domain, so that the wavelets do not have compact support in space.
This complicates working with the wavelets in the spatial domain and, for example, the filter taps for the fast transform are typically truncated.
Nonetheless, our construction yields a practical implementation of divergence-free, directional wavelets and we validated this with numerical experiments.

In future work, we would like to perform more extensive numerical experiments and investigate more applications of our construction.
We are currently working on a local spectral fluid simulator that retains the advantages of spectral methods but allows for spatial adaptivity.
In this work we also exploit the closed form expression for $\nabla \times \vec{\psi}_s$, i.e. that vorticity can be computed analytically.
We would also like to investigate in more detail the two different types of boundary functions that we constructed in two dimensions and to better understand to which flow feature these correspond.
Another interesting direction for future work is to extend Proposition~\ref{prop:curvelet_approx} to the bandlimited window functions $\hat{\gamma}_j(\theta_{\xi})$ used in our experiments.
In three dimensions, the boundary functions also need to be investigated in more detail.

\section*{Acknowledgements}

The author thanks Eugene Fiume for continuing support.
Michal Jarzabek and Philipp Herholz helped with implementing the construction.

\bibliographystyle{plain}
\bibliography{fluids}

\appendix

\section{Conventions}
\label{sec:conventions}

\subsection{Polar and Spherical Coordinates}
\label{sec:preliminaries:coordinates}

We will work primarily in spherical coordinates.
The angular variable will be denoted as $\bar{x} = x / {\vert x \vert}$ and, as usual, it is undefined when the radial variable $\vert x \vert$ is zero.
In $\R^2$, we will use the standard notation for polar coordinates with $\theta$ being the angle between the positive $x_1$-axis and the described point.
In the frequency domain, polar and spherical coordinates are defined analogously and when ambiguity might arise we will indicate the domain with a subscript, writing for example $\theta_{\xi}$ for the polar angle in frequency space.

When we integrate in spherical coordinates then the angular domain of integration is $S^{n-1}$, the sphere in $\R^n$, and the radial one is $\R_{\vert x \vert}^+$, i.e. the non-negative real numbers.

\subsection{The Fourier transform}
\label{sec:preliminaries:fourier}

The unitary Fourier transform of a function $f : \R^n \to \mathbb{C}$ in $L_1(\R^n) \cap L_2(\R^n)$ is defined as
\begin{subequations}
\label{eq:fourier_transform}
\begin{align}
  \label{eq:fourier_transform:forward}
  \mathcal{F}(f)(\xi) = \hat{f}(\xi) = \frac{1}{(2\pi)^{n/2}} \int_{\R_x^n} f(x) \, e^{-i \langle x , \xi \rangle} \, d x
\end{align}
with inverse transform
\begin{align}
  \label{eq:fourier_transform:inverse}
  \mathcal{F}^{-1}(f)(\xi) = f(x) = \frac{1}{(2\pi)^{n/2}} \int_{\R_{\xi}^n} \hat{f}(\xi) \, e^{i \langle \xi , x \rangle} \, d\xi .
\end{align}
\end{subequations}

\subsection{Spherical harmonics}
\label{sec:preliminaries:sh}

The analogue of the Fourier transform in Eq.~\ref{eq:fourier_transform} on the sphere is the spherical harmonics expansion.
For any $f \in L_2(S^2)$ it is given by
\begin{subequations}
  \label{eq:spherical_harmonics:expansion}
\begin{align}
  f(\omega)
  &= \sum_{l=0}^{\infty} \sum_{m=-l}^l \langle f(\eta) , y_{lm}(\eta) \rangle \, y_{lm}(\omega)
  \\
  &= \sum_{l=0}^{\infty} \sum_{m=-l}^l f_{lm} \, y_{lm}(\omega)
\end{align}
\end{subequations}
where $\langle \cdot , \cdot \rangle$ denotes the $L_2$ inner product on $S^2$.
We use standard (geographic) spherical coordinates with $\theta \in [0,\pi]$ being the polar angle and $\phi \in [0,2\pi]$ the azimuthal one.
The spherical harmonics basis functions in Eq.~\ref{eq:spherical_harmonics:expansion} are given by
\begin{align}
  \label{eq:spherical_harmonics}
  y_{lm}(\omega) = y_{lm}(\theta,\phi) = C_{lm} \, P_l^m( \cos{\theta} ) \, e^{i  m  \phi}
\end{align}
where the $P_l^m(\cdot)$ are associated Legendre polynomials and $C_{lm}$ is a normalization constant so that the $y_{lm}(\omega)$ are orthonormal over the sphere.

\paragraph{Clebsch-Gordon coefficients}

In contrast to the Fourier series, where the product of two basis functions $e^{i m_1 \theta}$ and $e^{i m_2 \theta}$ is immediately given by one other Fourier series function, namely $e^{i (m_1 + m_2) \theta}$, for spherical harmonics the product is not diagonal but characterized by Clebsch-Gordon coefficients $C_{l_1,m_1; l_2,m_2}^{l,m}$.
In particular, the projection of the product of $y_{l_1,m_1}(\omega)$ and $y_{l_2,m_2}(\omega)$ onto the spherical harmonics $y_{lm}(\omega)$ is given by
\begin{subequations}
  \label{eq:sh:product}
\begin{align}
\int_{S^2} y_{l_1 m_1}(\omega) \, y_{l_2 m_2}(\omega) \, y_{l m}^*(\omega) \, d\omega =
  \sqrt{\frac{(2l_1 + 1) (2l_2 + 1)}{4\pi (2l+1)}} C_{l_1,0; l_2,0}^{l,0} \, C_{l_1,m_1; l_2,m_2}^{l,m}
  \nonumber
\end{align}
and we call
\begin{align}
  G_{l_1,m_1; l_2,m_2}^{l,m}
  &\equiv \sqrt{\frac{(2l_1 + 1) (2l_2 + 1)}{4\pi (2l+1)}} \, C_{l_1,0; l_2,0}^{l,0} \, C_{l_1,m_1; l_2,m_2}^{l,m}
\end{align}
\end{subequations}
the spherical harmonics product coefficient.
The Clebsch-Gordon coefficients are sparse and non-zero only when $m = m_1 + m_2$,
and $l_1 + l_2 - l \geq 0$, $l_1 - l_2 + l \geq 0$, $-l_1 + l_2 + l \geq 0$.

\subsection{Fourier Transform in Polar and Spherical Coordinates}

To compute the Fourier transform in polar and spherical coordinates one requires and expression for the complex exponential in these coordinate systems.
These are given by the Jacobi-Anger and Rayleigh formulas.

\paragraph{Jacobi-Anger formula}
In the plane, the complex exponential can be expanded as
\begin{align}
  \label{eq:jacobi_anger}
  e^{i \langle \xi , x \rangle} = \sum_{m \in \Z} i^m \, e^{i m (\theta_x - \theta_{\xi})} J_m( \vert \xi \vert \, \vert x \vert ) ,
\end{align}
which is known as Jacobi-Anger formula~\cite{Watson1922}.
In Eq.~\ref{eq:jacobi_anger}, $J_m( z )$ is the Bessel function of the first kind and $(\theta_x,\vert x \vert)$ and $(\theta_{\xi},\vert \xi \vert)$ are polar coordinates for the spatial and frequency domains, respectively.

\paragraph{Rayleigh formula}
The analogue of the Jacobi-Anger formula in three dimensions is the Rayleigh formula,
\begin{subequations}
  \label{eq:rayleigh_formula}
\begin{align}
  e^{i \langle \xi , x \rangle}
  &= \sum_{l=0}^{\infty} i^l \, (2l+1) \, P_l\big(\bar{\xi} \cdot \bar{x}\big) \, j_l( \vert \xi \vert \, \vert x \vert )
  \\[5pt]
  &= 4 \pi \sum_{l=0}^{\infty} \sum_{m=-l}^l i^l \, y_{lm}\big(\bar{\xi}\big) \, y_{lm}^*\big(\bar{x}\big) \, j_l( \vert \xi \vert \, \vert x \vert )
\end{align}
\end{subequations}
where $j_l(\cdot)$ is the spherical Bessel function and the second line follows by the spherical harmonics addition theorem.

\section{Spatial Representation of Wavelets}
\label{sec:appendix:spatial}

In this appendix we detail the computation of the inverse Fourier transform of our wavelets yielding their spatial representations.

\subsection{Spatial Representation of Wavelets in Two Dimensions}
\label{sec:appendix:spatial:2d}

In two dimensions the spatial wavelet is given by the inverse Fourier transform
\begin{align}
  \label{eq:psi:divfree:dir}
  \vec{\psi}(x) = \frac{1}{2\pi} \int_{\R_{\xi}^2} \Big( -i \, \sum_{n} \beta_n \, e^{i n \theta_{\xi}}  \hat{h}( \vert \xi \vert) \, \vec{e}_{\theta_{\xi}} \Big) \, e^{i \langle \xi , x \rangle} \, d\xi
\end{align}
Writing $\vec{e}_{\theta_{\xi}}$ in Cartesian coordinates we obtain
\begin{align}
  \label{eq:psi:divfree:dir:2}
  \hat{\vec{\psi}}(\xi) = -\frac{i}{2\pi} \int_{\R_{\xi}^2} \sum_{n} \beta_n \, e^{i n \theta_{\xi}} \, \hat{h}( \vert \xi \vert) \, \begin{pmatrix}
    \sin{\theta_{\xi}}
    \\[-3pt]
    -\cos{\theta_{\xi}}
  \end{pmatrix}  \, e^{i \langle \xi , x \rangle} \, d\xi .
\end{align}
The two components of $\vec{\psi}(x)$ can thus be obtained by computing the inverse Fourier transform independently for each of them.
Writing the trigonometric functions in the Cartesian representation of $\vec{e}_{\theta_{\xi}}$ as Fourier series we obtain for the first components of $\hat{\vec{\psi}}(\xi)$
\begin{align}
  \label{eq:fhat1:general}
  \hat{\psi}_1(\xi) &= \frac{1}{2} \left( e^{-i \theta_{\xi}} - e^{i \theta_{\xi}} \right) \left( \sum_{n} \beta_n \, e^{i n \theta_{\xi}} \right) \, \hat{h}(\vert \xi \vert)
  \\
  &= \frac{1}{2} \left( \sum_{n} \beta_n \, e^{i (n-1) \theta_{\xi}} \, \hat{h}(\vert \xi \vert) - \sum_{n} \beta_n \, e^{i (n+1) \theta_{\xi}} \, \hat{h}(\vert \xi \vert)  \right) \nonumber .
\end{align}
Since $\hat{\psi}_1(\xi)$ is described in polar coordinates its inverse Fourier transform is most conveniently computed using these coordinates.
Writing the complex exponential using the Jacobi-Anger formula in Eq.~\ref{eq:jacobi_anger} we obtain
\begin{align}
  \psi_1(x) =
   -\frac{i}{4\pi} \int_{\R_{\vert \xi \vert}^+} \int_{S_{\xi}^2}
    & \left( \sum_{\sigma \in \{ -1, 1\} } \sum_{n} \beta_n \, e^{i (n + \sigma) \theta_{\xi}} \, \hat{h}(\vert \xi \vert) \right)
    \\[4pt]
    & \times \sum_{m \in \Z} i^m \, e^{i m (\theta_x - \theta_{\xi})} J_m( \vert \xi \vert \, \vert x \vert )
  \, \vert \xi \vert \, d\theta_{\xi} \, d\vert \xi \vert .
  \nonumber
\end{align}
Using linearity this equals
\begin{align}
  \psi_1(x) =
   -\frac{i}{4\pi}
   \sum_{\sigma \in \{ -1, 1\} } \sum_{n}
   \sum_{m \in \Z} & i^m \, \beta_n \, e^{i m \theta_x}
   \int_{S_{\xi}^2} \, e^{i (n + \sigma) \theta_{\xi}}  \, e^{-i m \theta_{\xi}} \, d\theta_{\xi}
    \\[4pt]
    & \times  \int_{\R_{\vert \xi \vert}^+} \, \hat{h}(\vert \xi \vert) \,  J_m( \vert \xi \vert \, \vert x \vert )
  \, \vert \xi \vert \, \, d\vert \xi \vert .
  \nonumber
\end{align}
By the orthogonality of the Fourier series the integral over $\theta_{\xi}$ equals $2\pi \delta_{n+\sigma,m}$.
Denoting the radial integral as $h_m(\vert x \vert)$ we obtain for the spatial representation of the first component
\begin{subequations}
  \label{eq:div_free:space}
\begin{align}
  \psi_1(x) &= \frac{1}{2} \sum_{\sigma \in \{ -1, 1\} } \sum_{m} \sigma \, i^{m + \sigma} \beta_m \, e^{i (m + \sigma) \theta_x} \, h_{m + \sigma}(\vert x \vert)
\end{align}
An analogous calculation for the second component yields
\begin{align}
  \psi_2(x) &= \frac{i}{2} \sum_{\sigma \in \{ -1, 1\} } \sum_{m} i^{m + \sigma} \beta_m \, e^{i (m + \sigma) \theta_x} \, h_{m + \sigma}(\vert x \vert) .
\end{align}
\end{subequations}
Taken together this yields Eq.~\ref{eq:psi:divfree:2d:space} in Sec.~\ref{sec:construction:2d}.

\subsection{Spatial Representation of Wavelets in Three Dimensions}
\label{sec:appendix:spatial:3d}

In three dimensions, the spatial representation of our divergence free wavelets is given by
\begin{subequations}
\begin{align}
  \vec{\psi}^a(x) = \frac{1}{(2\pi)^{3/2}} \int_{\R_{\xi}^3} \left( -i \, \sum_{lm} \kappa_{lm} \, y_{lm}(\bar{\xi}) \, \hat{h}\big( 2^{j_r} \vert \xi \vert\big) \, \vec{\tau}_a(\bar{\xi}) \right) e^{i \langle \xi , x \rangle} \, d\xi .
\end{align}
To compute it, it will be convenient to write $\vec{\tau}_a(\bar{\xi})$ in its spherical harmonics representation with respect to Cartesian coordinates.
For $\vec{\tau}_3(\bar{\xi})$ it is given by
\begin{align}
  \vec{\tau}_3(\bar{\xi}) =
  \left( \! \!
  \begin{array}{c}
    - \sin{\theta} \, \sin{\varphi}
    \\
    \quad \sin{\theta} \, \cos{\varphi}
    \\
    0
  \end{array}
  \! \! \right)
  =
  \sqrt{\frac{2 \pi}{3}} \left( \! \!
  \begin{array}{c}
    -i \, y_{1,-1}(\omega) - i \, y_{1,1}(\omega)
    \\
    y_{1,-1}(\omega) -  y_{1,1}(\omega)
    \\
    0
  \end{array}
  \! \! \right) .
\end{align}
and the expressions for $\vec{\tau}_1(\bar{\xi})$ and $\vec{\tau}_2(\bar{\xi})$ can be obtained using Wigner-D matrices.
The first component $\psi_{1}^1(x)$ of the inverse Fourier transform $\vec{\psi}^1(x)$ is thus
\begin{align}
  \psi_{1}^1(x)
  &= \frac{-1}{2\pi \sqrt{3}} \int_{\R_{\xi}^3} \sum_{lm} \kappa_{lm} \, y_{lm}(\bar{\xi}) \, \big( y_{1,-1}(\bar{\xi}) + y_{1,1}(\bar{\xi}) \big) \, \hat{h}\big( 2^{j_r} \vert \xi \vert\big) \, e^{i \langle \xi , x \rangle} \, d\xi .
  \nonumber
\end{align}
Changing to spherical coordinates and using the Rayleigh formula in Eq.~\ref{eq:rayleigh_formula} then yields
\begin{align}
  &\psi_{r,1}^1(x) = \frac{-2}{\sqrt{3}} \sum_{l,m} \kappa_{lm}^{j,t}
  \sum_{l_2,m_2} i^{l_2} \, y_{l_2 m_2}(\bar{x})
  \\[5pt]
  & \  \ \times \sum_{\sigma \in \{-1,1\}}
  \int_{S_{\bar{\xi}}^2} y_{lm}(\bar{\xi}) \, y_{1\sigma}(\bar{\xi}) \, y_{l_2 m_2}(\bar{\xi}) \, d\bar{\xi}
  \int_{\R_{\vert \xi \vert}^+}  \hat{h}\big( 2^{j_r} \vert \xi \vert\big) \, j_{l_2}( \vert \xi \vert \, \vert x \vert) \, \vert \xi \vert^2 \, d\vert \xi \vert .
  \nonumber
\end{align}
The angular integral is the projection of the product of two spherical harmonics into a third one, which is given by the spherical harmonics product coefficient introduced in Eq.~\ref{eq:sh:product}.
For the inverse Fourier transform of the first component we hence have
\begin{align}
  \psi_{r,1}^1(x) &= \frac{-2}{\sqrt{3}} \sum_{\substack{l,m \\ \sigma \in \pm 1}} \kappa_{lm}^{i,t}
  \sum_{l_2=l-1}^{l+1}
  i^{l_2} \, G_{l m;1 \sigma}^{l_2,m+\sigma} \, y_{l_2 m+\sigma}(\bar{x}) \, h_{l_2}( \vert x \vert) .
\end{align}
An analogous calculation for the second component yields
\begin{align}
  \psi_{r,2}^1(x) &= \frac{2 i}{\sqrt{3}} \sum_{\substack{l,m \\ \sigma \in \pm 1}} \kappa_{lm}^{i,t}
  \sum_{l_2=l-1}^{l+1}
  \sigma \, i^{l_2} \, G_{l m;1 \sigma}^{l_2,m+\sigma} \, y_{l_2 m+\sigma}(\bar{x}) \, h_{l_2}( \vert x \vert) .
\end{align}
\end{subequations}
Taken together, we have Eq.~\ref{eq:wavelets:3d:general} in Sec.~\ref{sec:construction:3d}.

\section{Fast Wavelet Transform}
\label{sec:appendix:fast_transform}

The fast wavelet transform determines the wavelet coefficients of a given signal from the scaling function coefficients $f_{j',k} = \langle f(x) , \psi_{j',k}(x) \rangle$ on some finest level $j'$ (often approximated by pointwise values, i.e. $f_{j',k} \approx f(2^{-j'} k)$, see~\cite{Abry1994,Zhang1996} and~\cite[p. 301]{Mallat2009} on this point).
The transform proceeds level-by-level, on each of them computing the approximation of the signal on the next coarser scale $j-1$ and the wavelet coefficients that describe the difference between the levels~\cite[p. 298]{Mallat2009}, and hence the information that is lost by using the coarser one.
In the literature, the algorithm is typically described for orthogonal wavelets, e.g.~\cite[Sec. 7.3.1]{Mallat2009}, where the spaces $W_j$ spanned by the wavelet functions are orthogonal.
This does not hold for our wavelets and correspondingly some modifications compared to the classical scheme are necessary.
We hence describe the derivation in some detail.
For simplicity we will present it in $\R^2$. It holds with obvious modifications in $\R^3$.

Our input signal is a divergence free vector field $\vec{u}(x) \in L_2^{\mathrm{div}}(\R^{2,2})$ and we assume we know the corresponding scaling function coefficients $u_{j'k}^s$ on some finest level $j'$.
For the first level, the transform is then given by
\begin{subequations}
  \label{eq:fwt:forward}
\begin{align}
  u_{j'-1,k}^s
  &= \left\langle \sum_{k' \in \Z^2} u_{j'k'}^s \, \vec{\phi}_{j',k'}(x) \, , \, \vec{\phi}_{j'-1,k}(x) \right\rangle
  \\[4pt]
   u_{j'-1,k,t}^w
  &= \left\langle \sum_{k' \in \Z^2} u_{j'k'}^s \, \vec{\phi}_{j',k'}(x) \, , \, \vec{\psi}_{j'-1,k,t}(x) \right\rangle .
\end{align}
Using linearity we obtain
\begin{align}
  \label{eq:fwt:alpha}
  u_{j'-1,k}^s &= \sum_{k' \in \Z^2} u_{j'k'}^s \underbrace{\left\langle \vec{\phi}_{j,k'}(x) \, , \, \vec{\phi}_{j-1,k}(x) \right\rangle}_{\displaystyle \alpha_{j,k'}}
  \\[4pt]
  \label{eq:fwt:beta}
  u_{j'-1,k,t}^w &= \sum_{k' \in \Z^2} u_{j'k'}^s \underbrace{\left\langle \vec{\phi}_{j,k'}(x) \, , \, \vec{\psi}_{j-1,k,t}(x) \right\rangle}_{\displaystyle \beta_{j,k',t}}
\end{align}
where $\alpha_{j,k'}$ and $\beta_{j,k',t}$ do not depend on $k$ by the discrete shift invariance of the basis functions.
Since the subspaces are not orthogonal (in contrast to the situation for orthogonal wavelets) we have on all finer levels
\begin{align}
  \label{eq:fwt:gamma}
  u_{j-1,k,t}^w &= \sum_{k' \in \Z^2} u_{jk'}^s \underbrace{\left\langle \vec{\phi}_{j,k'}(x) \, , \, \vec{\psi}_{j-1,k,t}(x) \right\rangle}_{\displaystyle \beta_{j,k',t}}
  \\[0pt]
  & \ \ \ + \sum_{t'=1}^{M_{j}} \sum_{k' \in \Z^2} u_{jk't'}^w \underbrace{\left\langle \vec{\psi}_{j,k',t'}(x) \, , \, \vec{\psi}_{j-1,k,t}(x) \right\rangle}_{\displaystyle \gamma_{j,k',t,t'}}
  \nonumber
\end{align}
\end{subequations}
with the $\gamma_{j,k',t,t'}$ describing the coupling between the wavelet functions on adjacent levels.
Eqs.~\ref{eq:fwt:alpha},~\ref{eq:fwt:beta},~\ref{eq:fwt:gamma} provide a means to compute the wavelet coefficients $u_{j,k,t}^w$ on all coarser levels $j \leq j'$ through discrete convolutions.
These together with the scaling function coefficients $u_{0,k}^s$ on a coarsest level $j=0$ fully characterize the input signal.

Analogous to the analysis in Eq.~\ref{eq:fwt:forward}, also the reconstruction can be implemented in a pyramidal scheme.
The reconstruction proceeds by reconstructing the scaling function coefficients on level $j+1$ from the scaling function and wavelet coefficients on level $j$.
Since the wavelet coefficients on all levels $0 \leq j < j'$ have been stored, this can be iterated until one obtains the scaling function coefficients on level $j'$ which were the input to the forward transform.
To obtain the coefficients $u_{j+1,k}^s$ on level $j+1$ from the available information we have to compute
\begin{subequations}
\begin{align}
  u_{j+1,k}^s
  &= \left\langle \sum_{k' \in \Z^2} u_{jk'}^s \vec{\phi}_{jk'}(x) + \sum_{t'=1}^{M_j} \sum_{k' \in \Z^2} u_{jk't'}^w \, \vec{\psi}_{j,k',t'}(x)  \,  ,  \, \vec{\phi}_{j+1,k}(x) \right\rangle
  \\[4pt]
  & \quad \quad \quad \quad \ \ \ + \left\langle \sum_{t'=1}^{M_j+1} \sum_{k' \in Z^2} u_{j+1,k',t'}^w \,\vec{\psi}_{j+1,k',t'}(x) \, ,  \, \vec{\phi}_{j+1,k}(x) \right\rangle
  \nonumber
\end{align}
\end{subequations}
where the $u_{j+1,kt'}^w$ have to be considered since the wavelet spaces $W_j$ are not orthogonal.
Using linearity we obtain
\begin{subequations}
  \label{eq:fwt:inverse:final}
\begin{align}
  u_{j+1,k}^s
  &= \sum_{k' \in \Z^2} u_{jk}^s \underbrace{\left\langle \vec{\phi}_{jk}(x) , \vec{\phi}_{j+1,k}(x) \right\rangle}_{\displaystyle \alpha_{jk'}}
  \\
  & \quad + \sum_{t'=1}^{M_j} \sum_{k' \in \Z^2} u_{jkt'}^w \underbrace{\left\langle \vec{\psi}_{j,k',t'}(x) , \vec{\phi}_{j+1,k}(x) \right\rangle}_{\displaystyle \beta_{jk't}}
  \\
  & \quad + \sum_{t'=1}^{M_j+1} \sum_{k' \in \Z^2} u_{j+1,kt'}^w \underbrace{\left\langle \vec{\psi}_{j+1,k',t'}(x) , \vec{\phi}_{j+1,k}(x) \right\rangle}_{\displaystyle \delta_{jk't}}
\end{align}
\end{subequations}
where the $\alpha_{jk}$ and $\beta_{jkt}$ are the same coefficients as in the forward transform.
Eq.~\ref{eq:fwt:inverse:final} shows that also the inverse wavelet transform can be realized through discrete convolutions with filters $\alpha_{jk}$, $\beta_{jkt}$ and $\delta_{jkt}$.

As stated above, the forward and inverse transforms are realized using infinite sums, which is not numerically practical.
By the decay of the wavelets in the spatial domain, which implies an analogous decay of the filter taps, the sums become finite when implemented in the finite precision available on practical computers.
Furthermore, the filter taps can also be truncated when the values become sufficiently small to be negligible for the application at hand (similar to the $\epsilon$-thresholding used with wavelets in many applications).
For instance, for the signal in Fig.~\ref{fig:rec_err_sig} one obtains for filters with $301$ taps reconstruction errors of $L_{\infty} = 7.6 \times 10^{-6}$ and $L_2 = 3.5 \times 10^{-5}$ while with $81$ taps the errors are $L_{\infty} = 1.5 \times 10^{-5}$ and $L_2 = 2.7 \times 10^{-4}$.
With $300$ taps one has hence approximately single precision accuracy, which is, for example, sufficient for applications in computer graphics.
Note that a larger number might be required in 3D.

The computational complexity of the fast wavelet transform is $\mathcal{O}(n)$ where $n$ is the length of the input signal and the constant is controlled by the number of non-zero filter taps~\cite[Sec. 7.3.1]{Mallat2009}.
This number is very large for our wavelets, and much larger than for classical, Daubechies-style discrete wavelets.
In practice, an implementation that computes coefficients in the frequency domain using the fast Fourier transform might hence be faster than the fast transform outlined above.

One disadvantage of an FFT-based computation of the transform is that it is difficult to exploit sparsity.
With the fast wavelet transform, in contrast, sparsity is more easily used to reduce storage and computations by having the sums in Eq.~\ref{eq:fwt:forward} and Eq.~\ref{eq:fwt:inverse:final} only run over the nonzero (or thresholded nonzero) coefficients.
For example, the signal in Fig.~\ref{fig:rec_err_sig} can be reconstructed with $11\%$ $L_2$ error with the $20\%$ largest coefficients (we currently only use four levels, with more levels also better compression and error rates could be achieved).

In the scalar case, the scaling function coefficients that serve as input to the fast transform are typically approximated by pointwise values, i.e. $f_{j',k} \approx f(2^{-j'} k)$, see e.g.~\cite{Abry1994,Zhang1996} and~\cite[p. 301]{Mallat2009}.
This does not easily carry over to the vector-valued case and we hence leave a thorough investigation of this point to future work.

\end{document}